\documentclass[11pt,a4paper]{amsart}
\usepackage[utf8]{inputenc}
\usepackage[T1]{fontenc}
\usepackage[english]{babel}

\usepackage{sidecap}

\usepackage{verbatim}
\usepackage{lmodern}
\usepackage{amsmath}
\usepackage{amssymb} 
\usepackage{amsthm} 
\usepackage{thmtools}
\usepackage{enumitem}
\usepackage{graphicx}
\usepackage{hyperref}
\usepackage{mathtools}
\usepackage{indentfirst}
\usepackage{tabularx}
\usepackage{bbm}
\usepackage[capitalise,nosort]{cleveref}
\usepackage{stmaryrd}
\usepackage{wasysym}

\usepackage{tikz} %pictures
\usepackage{tikz-cd} %diagrams
\usetikzlibrary{arrows,calc,decorations.markings}
\usepackage{xparse}

\usepackage[left=1in,top=2cm,right=1in,bottom=2cm]{geometry}

\usepackage{blindtext}

\makeatletter
\newcommand{\@bbify}[1]{
  \ifcsname b#1\endcsname
  \message{WARNING: Overwriting b#1 with blackboard letter!}
  \fi
  \expandafter\edef\csname b#1\endcsname
  {\noexpand\ensuremath{\noexpand\mathbb #1}\noexpand\xspace}}
\newcommand{\@calify}[1]{
  \ifcsname c#1\endcsname
  \message{WARNING: Overwriting c#1 with calligraphic letter!}
  \fi 
  \expandafter\edef\csname c#1\endcsname
  {\noexpand\ensuremath{\noexpand\mathcal #1}\noexpand\xspace}}
\newcommand{\@bfify}[1]{
  \ifcsname bf#1\endcsname
  \message{WARNING: Overwriting c#1 with bold letter!}
  \fi
  \expandafter\edef\csname bf#1\endcsname
  {\noexpand\ensuremath{\noexpand\mathbf #1}\noexpand\xspace}}
\newcounter{@letter}\stepcounter{@letter}
\loop\@bbify{\Alph{@letter}}\@calify{\Alph{@letter}}\@bfify{\Alph{@letter}}
\ifnum\the@letter<26\stepcounter{@letter}\repeat
\makeatother

\newenvironment{tz}{\begin{center}\begin{tikzpicture}}{\end{tikzpicture}\end{center}}

\tikzstyle{d}=[double distance=.3ex]
\tikzstyle{w}=[preaction={draw=white, -,line width=4pt}]
\NewDocumentCommand{\punctuation}{ m m O{5pt} }{\node at ($(#1.east)-(0,#3)$) {#2};}

\newcounter{diagram}
\renewcommand{\thediagram}{\thetheorem}
\newenvironment{diagram}{
\setcounter{diagram}{\value{theorem}}
\refstepcounter{theorem}\refstepcounter{diagram}
\begin{center}
\normalfont{(\thediagram)}\hfill\begin{tikzpicture}[baseline=(current bounding box.center)]}
    {\end{tikzpicture}\hfill\text{ }\end{center}}

\tikzset{%
node distance=1.5cm, la/.style={scale=0.8}, rr/.style={xshift=1.5cm},
space/.style={xshift=.5cm}, over/.style={auto=false,fill=white,inner sep=1.5pt, minimum size=0, outer sep=0},
    symbol/.style={%
        draw=none,
        every to/.append style={%
            edge node={node [sloped, allow upside down, auto=false]{$#1$}}},
            
    }, pro/.style={postaction={decorate,decoration={
        markings,
        mark=at position .5 with {\node at (0,0) {$\bullet$};}
      }},
      inner sep=.9ex,
      },
  n/.style={double equal sign distance, -implies}, t/.style={double distance=2.5pt, -implies, postaction={draw,-}},
}

\newcommand{\pushout}[1]{\node at ($({#1})-(10pt,-10pt)$) {$\ulcorner$};}
\newcommand{\pullback}[1]{\node at ($({#1})+(10pt,-10pt)$) {$\lrcorner$};}
\newcommand{\Hom}{\mathrm{Hom}}

\newlist{rome}{enumerate}{7}
\setlist[rome]{label=(\roman*)}

\newtheorem{theorem}{Theorem}[section]

\newtheorem{cor}[theorem]{Corollary}
\newtheorem{prop}[theorem]{Proposition}
\newtheorem{lemma}[theorem]{Lemma}
\declaretheorem[name=Theorem,numbered=yes]{theoremA}

\theoremstyle{definition}
\newtheorem{defn}[theorem]{Definition}
\newtheorem{ex}[theorem]{Example}
\newtheorem{notation}[theorem]{Notation}

\newtheorem{constr}[theorem]{Construction}

\theoremstyle{remark}
\newtheorem{rem}[theorem]{Remark}

\crefname{theorem}{Theorem}{Theorems}
\crefname{cor}{Corollary}{Corollaries}
\crefname{prop}{Proposition}{Propositions}
\crefname{lemma}{Lemma}{Lemmas}

\crefname{defn}{Definition}{Definitions}
\crefname{ex}{Example}{Examples}
\crefname{notation}{Notation}{Notations}
\crefname{descr}{Description}{Descriptions}
\crefname{constr}{Construction}{Constructions}

\crefname{rem}{Remark}{Remarks}

\newcommand{\FibP}{\pi_P}

\newcommand{\FibF}{\mathrm{Fib}(H)}
\newcommand{\FibA}{\mathrm{Fib}(P)}
\newcommand{\FibB}{\mathrm{Fib}(Q)}

\newcommand{\un}{\mathop{\scalebox{1.2}{\raisebox{-0.1ex}{$\ast$}}}}
\newcommand{\Ab}{\mathrm{Ab}}
\newcommand{\Set}{\mathrm{Set}}
\newcommand{\sSet}{\mathrm{sSet}}
\newcommand{\Top}{\mathrm{Top}}
\newcommand{\Cat}{\mathrm{Cat}}
\newcommand{\Dbl}{\mathrm{Dbl}}
\DeclareMathOperator{\id}{id}
\newcommand{\op}{\mathrm{op}}
\DeclareMathOperator{\ev}{ev}
\newcommand{\cst}{{\mathrm{cst}}}
\DeclareMathOperator{\Ob}{ob}

\newcommand{\VCat}{{\cV\text{-}\Cat}}
\newcommand{\CatV}{{\Cat(\cV)}}
\DeclareMathOperator{\Int}{Int}
\newcommand{\Dfib}{\mathrm{DFib}}
\newcommand{\Dopfib}{\mathrm{DOpfib}}

\DeclareMathOperator{\Und}{Und}
\newcommand{\slice}[2]{{#1}_{/{#2}}}
\newcommand{\sliceunder}[2]{{}^{{#2}/}{#1}}

\title[Internal Grothendieck construction for enriched categories]{Internal Grothendieck construction \\ for enriched categories}

\author[L.\ Moser]{Lyne Moser}
\address{Fakultät für Mathematik, Universität Regensburg, 93040 Regensburg, Germany}
\email{lyne.moser@ur.de}

\author[M.\ Sarazola]{Maru Sarazola}
\address{School of Mathematics, University of Minnesota, Minneapolis MN, 55455, USA}
\email{maru@umn.edu}

\author[P.\ Verdugo]{Paula Verdugo}
\address{Department of Mathematics and Statistics, Macquarie University, NSW 2109, Australia}
\email{paula.verdugo@hdr.mq.edu.au}

\date{}

\begin{document}

\begin{abstract}
    Given a cartesian closed category $\cV$, we introduce an internal category of elements $\int_\cC F$ associated to a $\cV$-functor $F\colon \cC^{\op}\to \cV$. When $\cV$ is extensive, we show that this internal Grothendieck construction gives an equivalence of categories between $\cV$-functors $\cC^{\op}\to \cV$ and internal discrete fibrations over~$\cC$, which can be promoted to an equivalence of $\cV$-categories.
    
    Using this construction, we prove a representation theorem for $\cV$-categories, stating that a $\cV$-functor $F\colon \cC^{\op}\to \cV$ is $\cV$-representable if and only if its internal category of elements $\int_\cC F$ has an internal terminal object. We further obtain a characterization formulated completely in terms of $\cV$-categories using shifted $\cV$-categories of elements. Moreover, in the presence of $\cV$-tensors, we show that it is enough to consider $\cV$-terminal objects in the underlying $\cV$-category $\Und\int_\cC F$ to test the representability of a $\cV$-functor $F$. We apply these results to the study of weighted $\cV$-limits, and also obtain a novel result describing weighted $\cV$-limits as certain conical internal limits.
\end{abstract}

\maketitle

\setcounter{tocdepth}{1}
\tableofcontents

\section{Introduction}\label{section:intro}

\subsection*{The classical picture}\label{introsubsec:set}
Category theory plays an important role in modern mathematics, as it allows us to abstract concepts and study them formally. Notably, fundamental constructions in algebra, geometry, and topology can be understood as categorical concepts defined by certain \emph{universal properties}. For example, the cartesian product of sets, the kernel of a linear map between vector spaces, and the fiber over a point in a scheme or topological space, are all instances of a universal construction called \emph{limit}. It then comes as no surprise that the study of universal properties plays a very prominent role in category theory.

Universal properties are defined in terms of the representability of certain functors. For instance, we say that a functor $G\colon \cI\to \cC$ admits a \emph{limit} if the functor $\cC^{\cI}(\Delta(-),G)\colon \cC^{\op}\to \Set$ sending an object $A\in \cC$ to the set of cones with summit $A$ over $G$ is \emph{representable}; i.e., if there is a pair $(L,\lambda)$ of an object $L\in \cC$ and a cone $\lambda\colon \Delta L\Rightarrow G$ inducing a natural isomorphism of sets $\lambda^*\colon \cC(-,L)\cong \cC^{\cI}(\Delta(-),G)$. While this is a working definition, in practice it is often useful to have another approach to test the representability of a certain functor. For example, a classical result characterizes limits of a functor $G$ as terminal objects in the category $\Delta\downarrow G$ of cones over $G$. This characterization has been especially useful when generalizing limits to the $\infty$-categorical setting; see \cite{joyal2008notes,htt}. 

The leading example of limits is in fact a special case of a much more general result, which involves a discrete version of the \emph{Grothendieck construction}
\[ \textstyle \int_\cC \colon \Cat(\cC^{\op},\Set)\to \slice{\Cat}{\cC}. \]
This functor takes a functor $F\colon\cC^\op\to\Set$ to its category of elements $\int_\cC F$, whose objects are the pairs $(A,x)$ of an object $A\in\cC$ and an element $x\in FA$, and whose morphisms $(A,x)\to(B,y)$ are the morphisms $f\colon A\to B$ in $\cC$ such that $Ff(y)=x$. A classical theorem then shows that a functor $F\colon \cC^{\op}\to \Set$ is representable if and only if its category of elements $\int_\cC F$ has a terminal object; see e.g.~\cite[Proposition 2.4.8]{Riehlcontext}. In particular, as the category of elements of the functor $\cC^\cI(\Delta(-),G)$ considered above can be identified with the category $\Delta\downarrow G$ of cones over~$G$, we retrieve  the above characterization of limits.  

In fact, the Grothendieck construction is fully faithful and its essential image consists of the \emph{discrete fibrations} over $\cC$. In other words, we get an equivalence of categories
\[\textstyle\int_\cC\colon \Cat(\cC^{\op},\Set)\xrightarrow{\simeq}\slice{\Dfib}{\cC},\] 
where $\slice{\Dfib}{\cC}$ is the full subcategory of $\slice{\Cat}{\cC}$ consisting of the discrete fibrations over $\cC$; see e.g.~\cite[Theorem 2.1.2]{LoregianRiehl}. Hence, this construction provides a useful perspective on presheaves, which allows for a \emph{fibrational} approach to representability problems. 

 The Grothendieck construction has a further application, which we will study in this paper. Going back to the definition of limits, one may wonder why we only consider cones with a single summit, which we can encode through a natural transformation $\Delta\{*\}\Rightarrow \cC(L,G-)$ using the constant functor $\Delta\{*\}\colon\cI\to\Set$ at the one-point set, instead of studying more general diagram shapes $W\colon\cI\to\Set$. The reason is that this latter notion, called a \emph{$W$-weighted limit}, can be subsumed under the notion of a conical limit by using the Grothendieck construction. More precisely, given functors $G\colon \cI\to \cC$ and $W\colon \cI\to \Set$, \cite[(3.33)]{Kelly} shows that the $W$-weighted limit of $G$ corresponds to the conical limit of the composite of functors $\int_\cI W\to\cI\xrightarrow{G}\cC$. In particular, this is one of the reasons why weighted limits do not play a crucial role in ordinary category theory.   

\subsection*{Enriched category theory}
While categories classically encompass mathematical objects and sets of morphisms between them, the categories one meets in practice often have additional structure. One way to encode this richer data is by using \emph{enriched categories} over a monoidal category $\cV$, also known as \emph{$\cV$-categories}. A $\cV$-category $\cC$ consists of a set of objects, and a collection of hom-objects $\cC(A,B)\in \cV$ for all objects $A,B\in \cC$---as opposed to the classical hom-sets---with suitable identities and compositions, and they are a truly ubiquitous type of structure. For instance, abelian categories are enriched over $\cV=\Ab$, the category of abelian groups; this encodes the fact that we can add and subtract morphisms. Letting $\cV=\Cat$, the category of categories, yields $2$-categories, where we also have $2$-morphisms between the morphisms; a prototypical example is the $2$-category of categories, functors, and natural transformations. Categories enriched in $\cV=\Top$, the category of topological spaces, or alternatively, in $\cV=\sSet$, the category of simplicial sets, provide a good environment to model $\infty$-categories \cite{htt,bergner}.

Of course, $\cV$-categories come with an appropriate notion of $\cV$-functors which preserve all the relevant data, and hence they form a category $\VCat$. We then expect universal constructions such as (weighted) limits and representability to be suitably compatible with the enrichment. For instance, when computing limits in $\Cat$, as it is naturally a $2$-category, these limits are not only expected to satisfy a universal property with respect to functors but also with respect to natural transformations; this is the notion of a $2$-categorical limit.

A natural question then arises: how much of the behavior we see in the classical setting is still present in the enriched setting? More concretely, we ask the following:
\begin{enumerate}[label=(\Alph*)]
    \item Do we have access to a fibrational perspective on $\cV$-presheaves? That is, is there an equivalence of categories between the category of $\cV$-functors $\cC^{\op}\to \cV$ and certain fibrations over $\cC$?
    \item Can we characterize the $\cV$-representability of a $\cV$-functor $\cC^{\op}\to \cV$ in terms of the existence of certain terminal objects?
    \item Can we encode weighted $\cV$-limits by using certain conical limits? 
\end{enumerate}

As in the classical setting, these questions rely on the existence of a Grothendieck construction, which now takes $\cV$-functors $\cC^\op\to\cV$ as inputs and suitably encodes all the relevant data. The goal of this paper is to introduce such a construction, and show how it provides positive answers to all the above questions.

In the literature, there is already a Grothendieck construction involving enriched categories due to Beardsley–-Wong \cite{BeardsleyWong}. It is defined for an ordinary category $\cC$ as a functor
\[\textstyle\int_\cC\colon \mathrm{Ps}(\cC^\op,\VCat)\to\slice{\VCat}{\cC_\cV}, \]
that takes a pseudofunctor $F\colon\cC^\op\to\VCat$, and produces a $\cV$-category $\int_\cC F$ with a projection to $\cC_\cV$, the free $\cV$-category on $\cC$. A similar construction by Tamaki \cite{tamaki} deals with lax functors. However, these constructions cannot consider the general case of a $\cV$-category $\cC$ (only an ordinary one), as the category $\VCat$ of $\cV$-categories is generally not $\cV$-enriched, even when $\cV$ is monoidal closed; and furthermore the $\cV$-functors considered do not even take value in the $\cV$-category $\cV$, which is the desired target to formulate representability problems.

\subsection*{The case $\cV=\Cat$}\label{introsubsec:cat}

We can begin to understand the complexity of the questions we posed, as well as get ideas for how to answer them, by looking at the case $\cV=\Cat$. A first positive answer to (A) is given by Buckley \cite{Buckley} who constructs a (lax) Grothendieck construction \[ \textstyle\int_\cC^{lax}\colon 2\Cat(\cC^{\op},\Cat)\to \slice{2\Cat}{\cC} \]
and shows that it induces an equivalence with respect to certain \emph{$2$-fibrations} over $\cC$. However, clingman and the first author \cite{cM1} determine that there is no relation between the representability of a $2$-functor and the existence of $2$-terminal objects in this lax version of the $2$-category of elements. There is yet another approach by Gagna--Harpaz--Lanari \cite{GHL} in the case of \emph{bilimits} which requires changing the notion of \emph{terminal objects} to \emph{bifinal objects}; however, this approach will not be generalizable to any monoidal closed category $\cV$ as it takes into account the specific $2$-categorical structure of $\Cat$ which allows for ``lax'' constructions.

In fact, all of the naive approaches to questions (B) and (C) which stay entirely within the context of $2$-categories are doomed to fail. For (B), \cite{cM1} provides small counterexamples showing that $2$-limits of a $2$-functor $G$ are in general not the same as $2$-terminal objects in any type of $2$-category whose objects are $2$-cones over $G$. This indicates that the $2$-representability of a $2$-functor $F\colon \cC^{\op}\to \Cat$ will not be in general captured by the notion of $2$-terminal objects in any $2$-category of elements. As for (C), it was observed by Kelly in \cite[\S 3.9]{Kelly} that the notion of weighted 2-limits is richer than that of conical 2-limits.

The solution, suggested by Verity \cite{VerityThesis} and Grandis--Par\' e \cite{GrandisPare}, is to pass to \emph{double categories}: another type of $2$-dimensional structure which have objects, two types of morphisms between objects---the horizontal and the vertical morphisms---, and $2$-morphisms. In other words, these are precisely \emph{internal categories} to $\Cat$. In \cite[\S 1.2]{GraParPersistentII} Grandis--Par\'e introduce a double Grothendieck construction 
\[ \textstyle\int_\cC \colon 2\Cat(\cC^{\op},\Cat)\to \slice{\Dbl\Cat}{\Int\cC}, \]
where $\Int\colon 2\Cat\to\Dbl\Cat$ denotes the canonical inclusion of $2$-categories into double categories which sees a $2$-category as a horizontal double category with only trivial vertical morphisms.
This construction was later used by clingman and the first author to show in \cite[Theorem 6.8]{cM2} that a $2$-functor $F\colon \cC^{\op}\to \Cat$ is $2$-representable if and only if its double category of elements $\int_\cC F$ has a \emph{double terminal object}, thus settling question (B) after a passage to double categories. As a special case, one obtains that a $2$-functor $G$ has a $2$-limit if and only if the double category $\Delta\downarrow \Int G$ of $2$-cones over $G$ has a double terminal object, a fact that was previously understood in \cite{VerityThesis,GrandisPare}. Moreover, Grandis--Par\'e  show in \cite[Proposition~1.4]{GraParPersistentII} that the $W$-weighted $2$-limit of a $2$-functor $G$ is equivalently given by the \emph{double limit}---defined as a double terminal object in the corresponding double category of cones---of the double functor 
\[ \textstyle\int_\cI W\to \Int\cI\xrightarrow{\Int G} \Int\cC. \]
This gives a positive answer to (C), and proves that \emph{conical} double limits are enough to capture weighted $2$-limits, and hence also lax $2$-limits by \cite[(5.8)]{Kellyelementary}. 

\subsection*{New results for more general \texorpdfstring{$\cV$}{V}}

For a cartesian closed category $\cV$ with pullbacks, we can also make sense of notions of \emph{enriched} and \emph{internal} categories in $\cV$. Using the strong link between these two notions highlighted in the case $\cV=\Cat$, in this paper we introduce a Grothendieck construction for $\cV$-categories, which gives positive answers to questions (A), (B), and (C).

Recall that an \emph{internal category} $\bA$ to $\cV$ consists of an object $\bA_0\in \cV$ of \emph{objects} and an object $\bA_1\in \cV$ of \emph{morphisms}, with suitable identities and compositions. For example, when $\cV=\Set$ we retrieve the notion of an ordinary category, and when $\cV=\Cat$ the notion of a double category. Internal categories come with a natural notion of internal functors and form a category $\CatV$. As for their link with $\cV$-categories, when $\cV$ has all small coproducts there is an \emph{internalization functor} $\Int\colon \VCat\to \CatV$ due to Cottrell--Fujii--Power \cite{Extensive} generalizing the horizontal embedding mentioned above for $\cV=\Cat$.

In order to define our construction, we are inspired by the fact that the classical category of elements $\int_\cC F$ of a functor $F\colon \cC^{\op}\to \Set$ can be expressed as the category whose set of objects and morphisms are given by the sets 
\[ \textstyle\bigsqcup_{A\in \cC} FA \quad \text{and} \quad \bigsqcup_{A,B\in \cC} \cC(A,B)\times FB. \]
We can replicate this description in the case of a $\cV$-functor $F\colon\cC^{\op}\to \cV$ and find that the above coproducts both live in $\cV$, so that this defines an internal category $\int_\cC F$ to $\cV$ which comes with a projection to $\Int\cC$. We then introduce our \emph{internal Grothendieck construction} in \cref{defn:groth} as a functor
\[ \textstyle \int_\cC\colon \VCat(\cC^{\op},\cV)\to \slice{\CatV}{\Int\cC} .\]

Analogous constructions have also been provided by Verity \cite{VerityThesis} in the case of a locally presentable category, by Toma\v{s}i\'{c} \cite{Tomasic} in the case of a topos, by Heuts--Moerdijk \cite{HeutsMoerdijk} and Rasekh \cite{RasekhYoneda} in the case of spaces---with enriched categories being $\infty$-categories---, by Rasekh \cite{RasekhD} in the case of simplicial presheaves, and by the first author with Rasekh--Rovelli \cite{MRR2} in the case of $\Theta_{n-1}$-spaces---with enriched categories being $(\infty,n)$-categories. 

When we consider our internal Grothendieck construction under the additional hypothesis that the base category~$\cV$ is \emph{extensive}---roughly meaning that coproducts and pullbacks in $\cV$ interact well with each other---, we find that the functor $\int_\cC$ is fully faithful, and its essential image consists of the \emph{internal discrete fibrations} over $\Int\cC$ (recalled in \cref{defn:discfib}). In other words, combining \cref{thm:equiv,thm:enrichedequiv} we give a positive answer to question (A) as follows.

\begin{theoremA} \label{thmA:equiv}
 The internal Grothendieck construction gives an equivalence of categories
    \[ \textstyle\int_\cC\colon \VCat(\cC^{\op},\cV)\to \slice{\Dfib}{\Int\cC} \]
    between the category of $\cV$-functors $\cC^{\op}\to\cV$, and the full subcategory $\slice{\Dfib}{\Int\cC}$ of $\slice{\CatV}{\Int\cC}$ consisting of the internal discrete fibrations over $\Int\cC$. Furthermore, the above equivalence can be promoted to an enriched equivalence of $\cV$-categories.
\end{theoremA}

We then turn to question (B), and show in \cref{thm:representation} how $\cV$-representable $\cV$-functors can be characterized by the existence of an \emph{internal terminal object} (see \cref{defn:internalterminal}) in their internal category of elements.

\begin{theoremA}\label{thmA:representation}
Let $F\colon \cC^{\op}\to \cV$ be a $\cV$-functor. Given a pair $(C,x)$ of an object $C\in \cC$ and an element $x\colon \un\to FC$ in $\cV$, the following are equivalent:
\begin{rome}
    \item the $\cV$-functor $F$ is $\cV$-representable by $(C,x)$,
    \item the pair $(C,x)$ is an internal terminal object in the internal category of elements $\int_\cC F$.
\end{rome}
\end{theoremA}

Finally, we apply our results to the case of weighted $\cV$-limits (see  \cref{defn:weightedlim}), and provide the following answer to question (C), which can be found as \cref{thm:repweightedlimits,thm:weightedasconical}.

\begin{theoremA} \label{thmA:weightasconical}
Let $G\colon \cI\to \cC$ and $W\colon \cI\to \cV$ be $\cV$-functors. Given a pair $(L,\lambda)$ of an object $L\in \cC$ and a $\cV$-natural transformation $\lambda\colon W\Rightarrow \cC(-,G)$ in $\VCat(\cI,\cV)$, the following are equivalent: 
    \begin{rome}
        \item the pair $(L,\lambda)$ is a $W$-weighted limit of $G$, 
        \item the pair $(L,\lambda)$ is an internal terminal object in the internal category $W\downarrow \Int\cC(-,G)$ of weighted $\cV$-cones over $G$, 
        \item the pair $(L,\lambda)$ is an internal limit of the internal functor 
        \[ \textstyle\int_\cI W\xrightarrow{\pi_W} \Int\cI\xrightarrow{\Int G} \Int\cC. \]
    \end{rome}
\end{theoremA}

In the special case where $W=\Delta \un\colon \cI\to \cV$ is the constant functor at the terminal object $\un\in \cV$, we obtain \cref{cor:repconicallimit,cor:enrichedlimitasinternallimit} which characterize \emph{conical} $\cV$-limits of a $\cV$-functor $G\colon \cI\to \cC$ as internal terminal objects in the internal category $\Int \Delta\downarrow G$ of $\cV$-cones over $G$ and as internal limits of the internal functor $\Int G$.

Throughout the paper, we will take as our leading example the case where $\cV=(n-1)\Cat$ is the category of $(n-1)$-categories for $n\geq 1$, with the convention that $0$-categories are sets. Inductively, \emph{$n$-categories} are defined as categories enriched over $(n-1)\Cat$. On the other hand, we refer to internal categories to $(n-1)\Cat$ as \emph{double $(n-1)$-categories}. The above results can then be translated as follows:
\begin{enumerate}[label=(\Alph*)]
    \item First, \cref{thmA:equiv} provides an $n$-equivalence between the $n$-categories of $n$-functors $\cC^{\op}\to (n-1)\Cat$ and of \emph{double $(n-1)$-discrete fibrations} over $\Int\cC$; see \cref{ex:equiv}.
    \item Next, \cref{thmA:representation} shows that an $n$-functor $\cC^{\op}\to (n-1)\Cat$ is $n$-representable if and only if its double $(n-1)$-category of elements has a \emph{double $(n-1)$-terminal object}; see \cref{ex:repthm}.
    \item Finally, \cref{thmA:weightasconical} tells us that (weighted) $n$-limits correspond to double $(n-1)$-terminal objects in the corresponding double $(n-1)$-category of (weighted) $n$-cones, and can further be obtained as certain conical \emph{double $(n-1)$-limits}; see \cref{ex:weightedlimit,ex:weightedasconical}.
\end{enumerate}  
In particular, in the case of conical $n$-limits, this provides intuition for the $\infty$-analogue \cite{MRR3}, namely the case of $(\infty,n)$-categories, where $(\infty,n)$-limits are defined as terminal objects in the corresponding double $(\infty,n-1)$-category of cones.

\subsection*{Fully enriched statements}

While \cref{thmA:representation} gives a first characterization of representable $\cV$-functors, one may wish to stay in the enriched setting whenever possible, instead of passing to the internal setting. We therefore provide a reformulation of the above statement in a fully enriched setting by considering ``shifted'' $\cV$-categories of generalized elements in \cref{enrichedstatement}. 

In order to state the result, we call a family of objects in $\cV$ \emph{conservative} if it creates isomorphisms (see \cref{defn:conservative}), and we denote by $\Und\colon \CatV\to \VCat$ the right adjoint of the internalization functor $\Int$ sending an internal category to $\cV$ to its \emph{underlying $\cV$-category}, 
and by $\llbracket -,-\rrbracket$ the internal hom bifunctor of $\CatV$. 

\begin{theoremA}
    Let $\cG$ be a conservative family of objects in $\cV$, and $F\colon \cC^{\op}\to \cV$ be a $\cV$-functor. Given a pair $(C,x)$ of an object $C\in \cC$ and an element $x\colon \un\to FC$ in $\cV$, the following are equivalent: 
\begin{rome}
    \item the $\cV$-functor $F$ is $\cV$-representable by $(C,x)$, 
    \item for every object $X\in \cG$, the pair $(C,x)$ is a $\cV$-terminal object in the $\cV$-category of generalized elements $\Und\llbracket X,\int_\cC F\rrbracket$.
\end{rome}
\end{theoremA}

In the case where $\cV=(n-1)\Cat$ and of the conservative family consisting of the free-living $(n-1)$-morphism $C_{n-1}$, this shows that an $n$-functor $F\colon \cC^{\op}\to (n-1)\Cat$ is $n$-representable if and only if its \emph{$n$-category of $(n-1)$-morphisms}, whose objects are pairs $(C,\varphi)$ of an object $C\in \cC$ and an $(n-1)$-morphism $\varphi$ in $FC$, has an $n$-terminal object; see \cref{ex:enrichedrepthm}. In the case where $n=2$, this retrieves \cite[Theorem 6.8]{cM2}.

In the presence of enriched tensors (see \cite[\S 3.7]{Kelly}), one can even simplify this statement by testing for terminality only in the underlying $\cV$-category, as we prove in \cref{thm:enrichedrepthm}. 

\begin{theoremA}
    Let $\cG$ be a conservative family of objects in $\cV$, $\cC$ be a $\cV$-category that admits all $\cV$-tensors by objects $X\in \cG$, and $F\colon \cC^{\op}\to \cV$ be a $\cV$-functor that preserves them. Given a pair $(C,x)$ of an object $C\in \cC$ and an element $x\colon \un\to FC$ in $\cV$, the following are equivalent: 
\begin{rome}
    \item the $\cV$-functor $F$ is $\cV$-representable by $(C,x)$, 
    \item the pair $(C,x)$ is a $\cV$-terminal object of the underlying $\cV$-category of elements $\Und\int_\cC F$. 
\end{rome}
\end{theoremA}

Considering $\cV=(n-1)\Cat$ and the conservative family consisting of $C_{n-1}$, this shows that, in the case of an $n$-category $\cC$ with $n$-tensors by $C_{n-1}$ and of an $n$-functor $F\colon \cC^{\op}\to (n-1)\Cat$ preserving these tensors, then $F$ is $n$-representable if and only if its \emph{$n$-category of elements}, whose objects are pairs $(C,x)$ of objects $C\in \cC$ and $x\in FC$, has an $n$-terminal object; see \cref{ex:repthmtensors}. Letting $n=2$ retrieves \cite[Theorem 6.15]{cM2}.

In the special case of weighted $\cV$-limits, the $\cV$-functor $\cC^{\op}\to \cV$ defining them automatically preserves $\cV$-tensors and so the result reduces to checking whether $\cC$ has $\cV$-tensors by a conservative family; see \cref{thm:limitswithtensors}. Hence, considering $\cV=(n-1)\Cat$ and the conservative family consisting of~$C_{n-1}$, this shows that, in the case of an $n$-category $\cC$ with $n$-tensors by~$C_{n-1}$, weighted $n$-limits in $\cC$ correspond to $n$-terminal objects in the $n$-category of weighted $n$-cones; see \cref{ex:weightedlimittensor}.  

\subsection*{Acknowledgments}

The authors are deeply grateful to tslil clingman, who was a vital part of the initial drive of this project, and to Nima Rasekh for discussions that prompted the inspiration for the approach we consider in this paper. We also thank Léonard Guetta for several interesting conversations. During the realization of this work, the last-named author was partially supported by an international Macquarie University Research Excellence Scholarship and the Centre of Australian Category Theory.

\section{Enriched and internal categories}

 We begin by collecting several background definitions and results on enriched categories in \cref{subsec:enrichedcat}, on internal categories in \cref{subsec:internalcat}, on the relation between them in \cref{subsec:relevant}, and on extensive categories in \cref{subsec:extensive}, that will be used throughout the paper.

\subsection{Enriched categories} \label{subsec:enrichedcat}

In this subsection, let $\cV$ denote a category with finite products, and let $\un$ denote its terminal object. We first introduce the category of categories enriched in the cartesian category $\cV$; see also \cite[\S 1.2]{Kelly} for definitions in the more general case of a monoidal category.

\begin{defn}
A \textbf{$\cV$-category} $\cC$, or \textbf{$\cV$-enriched category}, consists of 
\begin{rome}
\item a set $\Ob\cC$ of objects,
\item a hom-object $\cC(A,B)\in \cV$, for all objects $A,B\in \cC$, 
\item a composition map in $\cV$
\[ c_{A,B,C}\colon \cC(A,B)\times \cC(B,C)\to \cC(A,C), \]
for all objects $A,B,C\in \cC$, 
\item an identity map $i_A\colon \un\to \cC(A,A)$ in $\cV$, for every object $A\in \cC$,
\end{rome}
satisfying the following conditions.
\begin{enumerate}
    \item \textbf{Unitality}: for all objects $A,B\in \cC$, the following diagrams in $\cV$ commute.
\begin{tz}
\node[](1) {$\un\times \cC(A,B)\cong \cC(A,B)$}; 
\node[below of=1](2) {$\cC(A,A)\times \cC(A,B)$}; 
\node[right of=2,xshift=2.4cm](4) {$\cC(A,B)$}; 

\draw[->] (1) to node[left,la]{$i_A\times \id_{\cC(A,B)}$} (2);
\draw[->] (2) to node[below,la]{$c_{A,A,B}$} (4);
\draw[d] (1) to (4);

\node[right of=1,xshift=6cm](1) {$\cC(A,B)\times \un\cong \cC(A,B)$}; 
\node[below of=1](2) {$\cC(A,B)\times \cC(B,B)$}; 
\node[right of=2,xshift=2.4cm](4) {$\cC(A,B)$}; 

\draw[->] (1) to node[left,la]{$\id_{\cC(A,B)}\times i_B$} (2);
\draw[->] (2) to node[below,la]{$c_{A,B,B}$} (4);
\draw[d] (1) to (4);
\end{tz}
    \item \textbf{Associativity}: for all objects $A,B,C,D\in \cC$, the following diagram in $\cV$ commutes.
\begin{tz}
\node[](1) {$\cC(A,B)\times \cC(B,C)\times \cC(C,D)$}; 
\node[right of=1,xshift=5cm](2) {$\cC(A,B)\times \cC(B,D)$}; 
\node[below of=1](3) {$\cC(A,C)\times \cC(C,D)$}; 
\node[below of=2](4) {$\cC(A,D)$}; 

\draw[->] (1) to node[above,la]{$\id_{\cC(A,B)}\times c_{B,C,D}$} (2); 
\draw[->] (1) to node[left,la]{$c_{A,B,C}\times \id_{\cC(C,D)}$} (3); 
\draw[->] (2) to node[right,la]{$c_{A,B,D}$} (4);
\draw[->] (3) to node[below,la]{$c_{A,C,D}$} (4);
\end{tz}
\end{enumerate}
\end{defn}

\begin{defn}\label{defn:Vfunctor}
Let $\cC$ and $\cD$ be $\cV$-categories. A \textbf{$\cV$-functor} $F\colon \cC\to \cD$ consists of 
\begin{rome}
\item a map $\Ob F\colon \Ob\cC\to \Ob\cD$, $A\mapsto FA$ between sets of objects, 
\item a map $F_{A,B}\colon \cC(A,B)\to \cD(FA,FB)$ in $\cV$, for all objects $A,B\in \cC$,
\end{rome}
satisfying the following conditions. 
\begin{enumerate}
    \item \textbf{Compatibility with identities}: for every object $A\in \cC$, the following diagram in $\cV$ commutes.
\begin{tz}
\node[](1) {$\un$}; 
\node[right of=1,xshift=1cm](2) {$\cC(A,A)$}; 
\node[below of=2](4) {$\cD(FA,FA)$}; 

\draw[->] (1) to node[above,la]{$i_A$} (2);
\draw[->] (2) to node[right,la]{$F_{A,A}$} (4);
\draw[->] (1) to node[below,la,xshift=-3pt]{$i_{FA}$} (4);
\end{tz}
    \item \textbf{Compatibility with compositions}: for all objects $A,B,C\in \cC$, the following diagram in $\cV$ commutes. 
\begin{tz}
\node[](1) {$\cC(A,B)\times \cC(B,C)$}; 
\node[right of=1,xshift=3.6cm](2) {$\cC(A,C)$}; 
\node[below of=1](3) {$\cD(FA,FB)\times \cD(FB,FC)$}; 
\node[below of=2](4) {$\cD(FA,FC)$}; 

\draw[->] (1) to node[above,la]{$c_{A,B,C}$} (2); 
\draw[->] (1) to node[left,la]{$F_{A,B}\times F_{B,C}$} (3); 
\draw[->] (2) to node[right,la]{$F_{A,C}$} (4);
\draw[->] (3) to node[below,la]{$c_{FA,FB,FC}$} (4);
\end{tz}
\end{enumerate}
\end{defn}

\begin{notation}
We write $\VCat$ for the category of $\cV$-categories and $\cV$-functors. 
\end{notation}

\begin{ex}\label{ex:enriched}
We now give some examples of enriched categories for certain choices of~$\cV$.
\begin{enumerate}
    \item When $\cV=\Set$, a $\Set$-category has a set of objects, and hom-sets. This is just an ordinary category. Moreover, a $\Set$-functor is an ordinary functor, and so we get $\Cat=\Set\text{-}\Cat$.
    \item  When $\cV=\Cat$, a $\Cat$-category has a set of objects, and hom-categories. This is precisely a $2$-category. Moreover, a $\Cat$-functor is a $2$-functor, and so we get $2\Cat=\Cat\text{-}\Cat$. As a reference on $2$-categories, we recommend \cite{JohYau}.
    \item[($n$)] When $\cV=(n-1)\Cat$ for $n\geq 3$, we can define inductively the category of \emph{$n$-categories} and \emph{$n$-functors} to be the category of $(n-1)\Cat$-categories and $(n-1)\Cat$-functors, and we denote it by $n\Cat\coloneqq (n-1)\Cat\text{-}\Cat$.
\end{enumerate}
\end{ex}

\begin{ex}
If $\cV$ is cartesian closed, then it is $\cV$-enriched. We denote the hom-objects in $\cV$ by $[X,Y]$, for all objects $X,Y\in \cV$.
\end{ex}

As for categories, there is a notion of opposite $\cV$-categories. 

\begin{defn}
    Let $\cC$ be a $\cV$-category. Its \textbf{opposite $\cV$-category} $\cC^{\op}$ is such that 
    \begin{rome}
        \item its set of objects is $\Ob\cC$, 
        \item for objects $A,B\in \cC$, its hom-object $\cC^{\op}(B,A)$ is given by $\cC^{\op}(B,A)\coloneqq \cC(A,B)$,
        \item its composition and identity maps are the same as those of $\cC$. 
    \end{rome}
\end{defn}

In the case where $\cV$ is cartesian closed, we get an alternative characterization of $\cV$-functors into $\cV$ which is quite useful in practice. As most of the $\cV$-functors into $\cV$ we will consider in this paper are contravariant, we present our descriptions directly in this format. To introduce the desired characterization, we need the following notation. 

\begin{notation}\label{notation:evmap}
    Let $\cV$ be a cartesian closed category. Given a $\cV$-functor $F\colon \cC^{\op}\to \cV$ and objects $A,B\in \cC$, we denote by 
    \[ \ev^F_{A,B}\colon \cC(A,B)\times FB\to FA\]
    the unique map corresponding to the map $F_{A,B}\colon \cC(A,B)\to [FB,FA]$ under the adjunction $(-)\times FB\dashv~[FB,-]$.
\end{notation}

\begin{lemma} \label{lem:VfunctortoV}
    Let $\cV$ be a cartesian closed category. Then the data of a $\cV$-functor $F\colon \cC^{\op}\to \cV$ can equivalently be given as follows: it consists of
    \begin{rome}
        \item a map $\Ob F\colon \Ob\cC\to \Ob\cD$, $A\mapsto FA$ between sets of objects, 
        \item a map $\ev^F_{A,B}\colon \cC(A,B)\times FB\to FA$ in $\cV$, for all objects $A,B\in \cC$, 
    \end{rome}
    satisfying the following conditions. 
    \begin{enumerate}
        \item \textbf{Compatibility with identities}: for every object $A\in \cC$, the following diagram in $\cV$ commutes.
\begin{tz}
\node[](1) {$FA$}; 
\node[right of=1,xshift=1.8cm](2) {$\cC(A,A)\times FA$}; 
\node[below of=2](4) {$FA$}; 

\draw[->] (1) to node[above,la]{$i_{A}\times \id_{FA}$} (2);
\draw[->] (2) to node[right,la]{$\ev^F_{A,A}$} (4);
\draw[->] (1) to node[below,la,xshift=-3pt]{$\id_{FA}$} (4);
\end{tz}
    \item \textbf{Compatibility with compositions}: for all objects $A,B,C\in \cC$, the following diagram in $\cV$ commutes. 
\begin{tz}
\node[](1) {$\cC(A,B)\times \cC(B,C)\times FC$}; 
\node[right of=1,xshift=4cm](2) {$\cC(A,C)\times FC$}; 
\node[below of=1](3) {$\cC(A,B)\times FB$}; 
\node[below of=2](4) {$FA$}; 

\draw[->] (1) to node[above,la]{$c_{A,B,C}\times \id_{FC}$} (2); 
\draw[->] (1) to node[left,la]{$\id_{\cC(A,B)}\times \ev^F_{B,C}$} (3); 
\draw[->] (2) to node[right,la]{$\ev^F_{A,C}$} (4);
\draw[->] (3) to node[below,la]{$\ev^F_{A,B}$} (4);
\end{tz}
\end{enumerate}
\end{lemma}

\begin{proof}
    This follows directly from the universal property of the adjunctions $(-)\times FA\dashv [FA,-]$, for objects $A\in \cC$. 
\end{proof}

We can use the characterization in the above lemma to define representable $\cV$-functors. 

\begin{defn}
Let $\cV$ be a cartesian closed category, $\cC$ be a $\cV$-category and $C\in\cC$ be an object. The \textbf{representable $\cV$-functor} at $C$ is the $\cV$-functor $\cC(-,C)\colon \cC^{\op}\to \cV$ which sends
    \begin{rome}
        \item an object $A\in \cC$ to the hom-object $\cC(A,C)$ in $\cV$, 
        \item for objects $A,B\in \cC$, the map $\ev^{\cC(-,C)}_{A,B}$ is given by the composition map in $\cV$
        \[ c_{A,B,C}\colon \cC(A,B)\times \cC(B,C)\to \cC(A,C). \]
    \end{rome}
    Compatibility with identities and compositions follow from the unitality and associativity of the composition maps of $\cC$. 
\end{defn}

The $\cV$-functors are the objects of a category, whose morphisms we now describe.

\begin{defn}
Let $F,G\colon \cC\to \cD$ be $\cV$-functors. A \textbf{$\cV$-natural transformation} $\alpha\colon F\Rightarrow G$ consists of a map $\alpha_A\colon 1\to \cD(FA,GA)$ in $\cV$ for every object $A\in \cC$, such that for all objects $A,B\in \cC$, the following diagram in $\cV$ commutes.
\begin{tz}
\node[](1) {$\cC(A,B)$}; 
\node[below of=1](3) {$\cD(GA,GB)$};
\node[right of=1,xshift=2.2cm](2) {$\cD(FA,FB)$}; 
\node[below of=2](4) {$\cD(FA,GB)$}; 

\draw[->] (1) to node[above,la]{$F_{A,B}$} (2);
\draw[->] (2) to node[right,la]{$(\alpha_B)_*$} (4);
\draw[->] (1) to node[left,la]{$G_{A,B}$} (3);
\draw[->] (3) to node[below,la]{$(\alpha_A)^*$} (4);
\end{tz}
\end{defn}

\begin{notation}
Let $\cC$ and $\cD$ be $\cV$-categories. We write $\VCat(\cC,\cD)$ for the category of $\cV$-functors $\cC\to \cD$ and $\cV$-natural transformations between them.
\end{notation}

In the case where $\cV$ is cartesian closed, we also get an alternative characterization of $\cV$-natural transformations between $\cV$-functors into $\cV$. To introduce it, we need the following notation. 

\begin{notation} \label{notn:evalpha} 
    Let $\cV$ be a cartesian closed category. Given a $\cV$-natural transformation $\alpha\colon F\Rightarrow G$ between $\cV$-functors $F,G\colon \cC^{\op}\to \cV$ and an object $A\in \cC$, we denote by 
    \[ \ev^\alpha_{A}\colon FA\to GA\]
    the unique map in $\cV$ corresponding to the map $\alpha_A\colon \un\to [FA,GA]$ under the adjunction $(-)\times FA\dashv [FA,-]$. 
\end{notation}

\begin{lemma} \label{lem:VnattoV}
    Let $\cV$ be a cartesian closed category, and $F,G\colon \cC^{\op}\to \cV$ be $\cV$-functors. Then the data of a $\cV$-natural transformation $\alpha\colon F\Rightarrow G$ can equivalently be given as follows: it consists of a map $\ev^\alpha_A\colon FA\to GA$ in $\cV$ for every object $A\in \cC$, such that for all objects $A,B\in \cC$, the following diagram in $\cV$ commutes. 
    \begin{tz}
\node[](1) {$\cC(A,B)\times FB$}; 
\node[below of=1](3) {$\cC(A,B)\times GB$};
\node[right of=1,xshift=1.8cm](2) {$FA$}; 
\node[below of=2](4) {$GA$}; 

\draw[->] (1) to node[above,la]{$\ev^F_{A,B}$} (2);
\draw[->] (2) to node[right,la]{$\ev^\alpha_A$} (4);
\draw[->] (1) to node[left,la]{$\id_{\cC(A,B)}\times \ev^\alpha_B$} (3);
\draw[->] (3) to node[below,la]{$\ev_{A,B}^G$} (4);
\end{tz}
\end{lemma}

\begin{proof}
    This follows directly from the universal property of the adjunction $(-)\times FA\dashv [FA,-]$, for objects $A\in \cC$.
\end{proof}

In fact, when $\cV$ is cartesian closed, $\cV$-functors assemble into a $\cV$-category; see \cite[\S 2.2]{Kelly}.

\begin{prop}
    If $\cV$ is cartesian closed, so is $\VCat$.
\end{prop}

\begin{notation} \label{rem:internalhomVcat}
    Given $\cV$-categories $\cC$ and $\cD$, we denote by $\cD^\cC$ the internal hom-$\cV$-category. Recall that this is the $\cV$-category such that 
    \begin{rome}
        \item its objects of $\cD^\cC$ are $\cV$-functors $\cC\to \cD$, 
        \item for $\cV$-functors $F,G\colon \cC\to \cD$, the hom-object $\cD^\cC(F,G)$ is given by the end in $\cV$
\[ \textstyle \int_{\cC} \cD(F-,G-)\cong \mathrm{eq}(\prod_{A\in \cC} \cD(FA,GA)\rightrightarrows \prod_{A,B\in \cC} [\cC(A,B),\cD(FA,GB)] ), \]
        where the equalizer is taken over the two parallel maps given by the unique maps in $\cV$ corresponding under the adjunction $\cC(A,B)\times(-)\cong [\cC(A,B),-]$ to the composites
        \[  \cC(A,B)\times \cD(FA,GB)\xrightarrow{F_{A,B}\times \id} \cD(FA,FB)\times \cD(FB,GB) \xrightarrow{c_{FA,FB,GB}} \cD(FA,GB), \]
        \[ \cD(FA,GA)\times \cC(A,B)\xrightarrow{\id\times G_{A,B}} \cD(FA,GA)\times\cD(GA,GB)\xrightarrow{c_{FA,GA,GB}} \cD(FA,GB), \]
        respectively.
        \end{rome}
        
        Moreover, given a $\cV$-functor $F\colon \cC\to \cC'$ and a $\cV$-category $\cD$, we write $F^*\coloneqq \cD^F\colon \cD^{\cC'}\to \cD^\cC$ and $F_*\coloneqq F^\cD\colon \cC^\cD\to (\cC')^\cD$ for the induced $\cV$-functors. 
\end{notation}

\begin{rem}
    Note that, by \cite[(2.16)]{Kelly}, a map $\un\to \cD^\cC(F,G)$ in $\cV$ corresponds to a $\cV$-natural transformation $F\Rightarrow G$. 
\end{rem}

\subsection{Internal categories} \label{subsec:internalcat}

In this subsection, let $\cV$ denote a category with pullbacks.

\begin{defn}
An \textbf{internal category} $\bA$ to $\cV$ consists of 
\begin{rome}
\item two objects $\bA_0,\bA_1\in \cV$, where $\bA_0$ is called the \emph{object of objects} and $\bA_1$ the \emph{object of morphisms},
\item source and target maps $s,t\colon \bA_1\to \bA_0$ in $\cV$, 
\item a composition map $c\colon \bA_1\times_{\bA_0}\bA_1\to \bA_1$ in $\cV$, 
\item an identity map $i\colon \bA_0\to \bA_1$ in $\cV$,
\end{rome}
satisfying the following conditions.
\begin{enumerate}
    \item \textbf{Compatibility with source and target}: the following diagrams in~$\cV$ commute. 
\begin{tz}
\node[](1) {$\bA_0$}; 
\node[below of=1](2) {$\bA_1$}; 
\node[left of=2,xshift=-.7cm](3) {$\bA_0$};
\node[right of=2,xshift=.7cm](4) {$\bA_0$}; 

\draw[->] (1) to node[left,la]{$i$} (2);
\draw[->] (2) to node[below,la]{$t$} (4);
\draw[d] (1) to (4);
\draw[d] (1) to (3);
\draw[->] (2) to node[below,la]{$s$} (3);

\node[right of=1,xshift=2.7cm](5) {$\bA_1$};
\node[below of=5](6) {$\bA_0$};
\node[right of=5,xshift=1cm](1) {$\bA_1\times_{\bA_0} \bA_1$}; 
\node[below of=1](3) {$\bA_1$};
\node[right of=1,xshift=1cm](2) {$\bA_1$}; 
\node[below of=2](4) {$\bA_0$}; 

\draw[->] (1) to node[above,la]{$\pi_1$} (2);
\draw[->] (2) to node[right,la]{$t$} (4);
\draw[->] (1) to node[left,la]{$c$} (3);
\draw[->] (3) to node[below,la]{$t$} (4);
\draw[->] (3) to node[below,la]{$s$} (6);
\draw[->] (5) to node[left,la]{$s$} (6);
\draw[->] (1) to node[above,la]{$\pi_0$} (5);
\end{tz}
    \item \textbf{Unitality}: the following diagrams in $\cV$ commute.
\begin{tz}
\node[](5) {$\bA_1\cong \bA_1\times_{\bA_0} \bA_0$};
\node[right of=5,xshift=2.5cm](1) {$\bA_1\times_{\bA_0} \bA_1$}; 
\node[below of=1](3) {$\bA_1$};
\node[right of=1,xshift=2.5cm](2) {$\bA_0\times_{\bA_0}\bA_1\cong \bA_1$}; 

\draw[->] (2) to node[above,la]{$i\times \id_{\bA_1}$} (1);
\draw[d] (2) to (3);
\draw[->] (1) to node[left,la]{$c$} (3);
\draw[d] (5) to (3);
\draw[->] (5) to node[above,la]{$\id_{\bA_1}\times i$} (1);
\end{tz}
    \item \textbf{Associativity}: the following diagram in $\cV$ commutes.
\begin{tz}
\node[](1) {$\bA_1\times_{\bA_0} \bA_1\times_{\bA_0} \bA_1$}; 
\node[below of=1](3) {$\bA_1\times_{\bA_0} \bA_1$};
\node[right of=1,xshift=2.7cm](2) {$\bA_1\times_{\bA_0} \bA_1$}; 
\node[below of=2](4) {$\bA_1$}; 

\draw[->] (1) to node[above,la]{$c\times \id_{\bA_1}$} (2);
\draw[->] (2) to node[right,la]{$c$} (4);
\draw[->] (1) to node[left,la]{$\id_{\bA_1}\times c$} (3);
\draw[->] (3) to node[below,la]{$c$} (4);
\end{tz}
\end{enumerate}
\end{defn}

\begin{defn}
Let $\bA$, $\bB$ be internal categories to $\cV$. An \textbf{internal functor} $H\colon \bA\to \bB$ consists of maps $H_0\colon \bA_0\to \bB_0$ and $H_1\colon \bA_1\to \bB_1$ in $\cV$ satisfying the following conditions. 
\begin{enumerate}
    \item \textbf{Compatibility with source and target}: the following diagrams in~$\cV$ commute. 
\begin{tz}
\node[](1) {$\bA_1$}; 
\node[below of=1](3) {$\bB_1$};
\node[right of=1,xshift=.7cm](2) {$\bA_0$}; 
\node[below of=2](4) {$\bB_0$}; 

\draw[->] (1) to node[above,la]{$s$} (2);
\draw[->] (2) to node[right,la]{$H_0$} (4);
\draw[->] (1) to node[left,la]{$H_1$} (3);
\draw[->] (3) to node[below,la]{$s$} (4);

\node[right of=2,xshift=1.5cm](1) {$\bA_1$}; 
\node[below of=1](3) {$\bB_1$};
\node[right of=1,xshift=.7cm](2) {$\bA_0$}; 
\node[below of=2](4) {$\bB_0$}; 

\draw[->] (1) to node[above,la]{$t$} (2);
\draw[->] (2) to node[right,la]{$H_0$} (4);
\draw[->] (1) to node[left,la]{$H_1$} (3);
\draw[->] (3) to node[below,la]{$t$} (4);
\end{tz}
    \item \textbf{Compatibility with identity}: the following diagram in $\cV$ commutes. 
\begin{tz}
\node[](1) {$\bA_0$}; 
\node[below of=1](3) {$\bB_0$};
\node[right of=1,xshift=.7cm](2) {$\bA_1$}; 
\node[below of=2](4) {$\bB_1$}; 

\draw[->] (1) to node[above,la]{$i$} (2);
\draw[->] (2) to node[right,la]{$H_1$} (4);
\draw[->] (1) to node[left,la]{$H_0$} (3);
\draw[->] (3) to node[below,la]{$i$} (4);
\end{tz}
    \item \textbf{Compatibility with composition}: the following diagram in $\cV$ commutes. 
\begin{tz}
\node[](1) {$\bA_1\times_{\bA_0} \bA_1$}; 
\node[below of=1](3) {$\bB_1\times_{\bB_0} \bB_1$};
\node[right of=1,xshift=1.3cm](2) {$\bA_1$}; 
\node[below of=2](4) {$\bB_1$}; 

\draw[->] (1) to node[above,la]{$c$} (2);
\draw[->] (2) to node[right,la]{$H_1$} (4);
\draw[->] (1) to node[left,la]{$H_1\times H_1$} (3);
\draw[->] (3) to node[below,la]{$c$} (4);
\end{tz}
\end{enumerate}
\end{defn}

\begin{notation}
We write $\CatV$ for the category of internal categories to $\cV$ and internal functors.
\end{notation}

\begin{ex}\label{ex:internal}
We give some examples of internal categories for certain choices of~$\cV$.
\begin{enumerate}
    \item When $\cV=\Set$, an internal category to $\Set$ has a set of objects, and a set of morphisms. This is just an ordinary category. Moreover, an internal functor to $\Set$ is an ordinary functor, and so we get $\Cat=\Cat(\Set)$.
    \item When $\cV=\Cat$, an internal category to $\Cat$ has a category of objects and a category of morphisms. This is called a \emph{double category}, and we typically interpret the data in the category of objects as objects and vertical morphisms, and the data in the category of morphisms as horizontal morphisms and squares. Moreover, an internal functor to $\Cat$ corresponds to a \emph{double functor}, and we have $\Dbl\Cat=\Cat(\Cat)$. As a reference on double categories, we recommend \cite{Grandis}.
    \item[($n$)] When $\cV=(n-1)\Cat$ for $n\geq 3$, we define the category of \emph{double $(n-1)$-categories} and \emph{double $(n-1)$-functors} to be the category of internal categories to $(n-1)\Cat$ and internal functors, and we denote it by $\Dbl(n-1)\Cat\coloneqq \Cat((n-1)\Cat)$. When $n=3$, this gives a strict version of the double $2$-categories from \cite{CLPS}.
\end{enumerate}
\end{ex}

When $\cV$ is cartesian closed, internal functors assemble into an internal category to $\cV$; see \cite[Lemma B2.3.15(ii)]{elephant}.

\begin{prop}
    When $\cV$ is cartesian closed, so is $\CatV$.
\end{prop}

\begin{notation}
    Given internal categories $\bA$ and $\bB$ to $\cV$, we denote by $\llbracket \bA,\bB\rrbracket$ the internal hom internal category of $\CatV$. 

    Moreover, given an internal functor $H\colon \bA\to \bA'$ and an internal category $\bB$ to $\cV$, we write $H^*\coloneqq \llbracket H,\bB\rrbracket\colon \llbracket \bA',\bB\rrbracket\to \llbracket \bA,\bB\rrbracket$ and $H_*\coloneqq \llbracket \bB,H\rrbracket\colon \llbracket \bB,\bA\rrbracket\to \llbracket \bB,\bA'\rrbracket$ for the induced internal functors.
\end{notation}

 \subsection{Relevant functors} \label{subsec:relevant}

In this subsection, let $\cV$ denote a category with pullbacks (and hence finite products). There is a canonical inclusion of $\cV$ into the category $\CatV$ as follows. 

\begin{defn} 
The \textbf{constant functor} $\cst\colon \cV\to \CatV$ sends an object $X$ in $\cV$ to the internal category $\cst X$ to $\cV$ constant at $X$; i.e., the internal category with
\begin{rome}
    \item $(\cst X)_0=(\cst X)_1=X$, and
    \item  $s=t=i=c=\id_X$.
\end{rome}
 It sends a map $f\colon X\to Y$ in $\cV$ to the internal functor $\cst f\colon \cst X\to \cst Y$ constant at $f$; i.e., the internal functor with $(\cst f)_0=(\cst f)_1=f$.

 We call an internal category $\bA$ to $\cV$ that is in the essential image of $\cst\colon \cV\to \CatV$ a \textbf{constant internal category}. 
 \end{defn}

 This functor is left adjoint to the functor sending an internal category to its object of objects. 

 \begin{prop}
There is an adjunction 
\begin{tz}
\node[](1) {$\cV$}; 
\node[right of=1,xshift=1.1cm](2) {$\CatV$}; 
\punctuation{2}{.};

\draw[->] ($(1.east)+(0,5pt)$) to node[above,la]{$\cst$} ($(2.west)+(0,5pt)$);
\draw[->] ($(2.west)-(0,5pt)$) to node[below,la]{$(-)_0$} ($(1.east)-(0,5pt)$);

\node[la] at ($(1.east)!0.5!(2.west)$) {$\bot$};
\end{tz}
\end{prop}

 \begin{notation}
     For convenience, we write $\un$ for the internal category $\cst \un$ at the terminal object $\un$ of $\cV$. Note that this is a terminal object in $\CatV$. Given an internal category $\bA$, we call an internal functor $A\colon \un\to \bA$ an \textbf{element} of $\bA$.
 \end{notation}

 In the case where $\cV$ has small coproducts, we can relate the categories of enriched and internal categories sharing a common base category $\cV$ as follows; see \cite{Extensive} for further details.

\begin{defn}
The \textbf{internalization functor} $\Int\colon \VCat\to \CatV$ sends a $\cV$-ca\-te\-gory~$\cC$ to the internal category $\Int\cC$ to $\cV$ whose
\begin{rome}
\item object of objects and object of morphisms are given by the coproducts in $\cV$
\[ \textstyle(\Int\cC)_0\coloneqq \bigsqcup_{A\in \cC} \un \quad \text{and} \quad
 \textstyle(\Int\cC)_1\coloneqq \bigsqcup_{A,B\in \cC} \cC(A,B) , \]
\item source and target maps $s,t\colon (\Int\cC)_1\to (\Int\cC)_0$ are the unique maps in $\cV$ making the following diagrams commute, for all objects $A,B\in \cC$,
\begin{tz}
\node[](1) {$\cC(A,B)$}; 
\node[below of=1](3) {$\un$};
\node[right of=1,xshift=2.2cm](2) {$\bigsqcup_{A,B\in \cC} \cC(A,B)$}; 
\node[below of=2](4) {$\bigsqcup_{A\in \cC} \un$}; 

\draw[->] (1) to node[above,la]{$\iota_{A,B}$} (2);
\draw[->] (2) to node[right,la]{$s$} (4);
\draw[->] (1) to node[left,la]{$!$} (3);
\draw[->] (3) to node[below,la]{$\iota_A$} (4);

\node[right of=2,xshift=2cm](1) {$\cC(A,B)$}; 
\node[below of=1](3) {$\un$};
\node[right of=1,xshift=2.2cm](2) {$\bigsqcup_{A,B\in \cC} \cC(A,B)$}; 
\node[below of=2](4) {$\bigsqcup_{B\in \cC} \un$}; 

\draw[->] (1) to node[above,la]{$\iota_{A,B}$} (2);
\draw[->] (2) to node[right,la]{$t$} (4);
\draw[->] (1) to node[left,la]{$!$} (3);
\draw[->] (3) to node[below,la]{$\iota_B$} (4);
\end{tz}
    \item composition map $c\colon (\Int\cC)_1\times_{(\Int\cC)_0} (\Int\cC)_1\to (\Int\cC)_1$ is the unique map in $\cV$ making the following diagram commute, for all objects $A,B,C\in \cC$,
\begin{tz}
\node[](1) {$\cC(A,B)\times\cC(B,C)$}; 
\node[below of=1](3) {$\cC(A,C)$};
\node[right of=1,xshift=4.1cm](2) {$\bigsqcup_{A,B,C\in \cC} \cC(A,B)\times\cC(B,C)$}; 
\node[below of=2](4) {$\bigsqcup_{A,C\in \cC} \cC(A,C)$}; 

\draw[->] (1) to node[above,la]{$\iota_{A,B,C}$} (2);
\draw[->] (2) to node[right,la]{$c$} (4);
\draw[->] (1) to node[left,la]{$c_{A,B,C}$} (3);
\draw[->] (3) to node[below,la]{$\iota_{A,C}$} (4);
\end{tz}
    \item identity map $i\colon (\Int\cC)_0\to (\Int\cC)_1$ is the unique map in $\cV$ making the following diagram commute, for every object $A\in \cC$.
\begin{tz}
\node[](1) {$\un$}; 
\node[below of=1](3) {$\cC(A,A)$};
\node[right of=1,xshift=2.2cm](2) {$\bigsqcup_{A\in \cC} \un$}; 
\node[below of=2](4) {$\bigsqcup_{A,B\in \cC} \cC(A,B)$}; 

\draw[->] (1) to node[above,la]{$\iota_{A}$} (2);
\draw[->] (2) to node[right,la]{$i$} (4);
\draw[->] (1) to node[left,la]{$i_A$} (3);
\draw[->] (3) to node[below,la]{$\iota_{A,A}$} (4);
\end{tz}
Compatibility with source and target are straightforward, and unitality and associativity of composition follows from the respective properties  in $\cC$. 

On morphisms, the functor $\Int$ sends a $\cV$-functor $F\colon \cC\to \cD$ to the internal functor $\Int F\colon \Int\cC\to \Int\cD$ such that $(\Int F)_0$ and $(\Int F)_1$ are the unique maps in $\cV$ making the following diagrams commute, for all objects $A,B\in \cC$.
\begin{tz}
\node[](1) {$\un$}; 
\node[right of=1,xshift=1cm](2) {$\bigsqcup_{A\in \cC} \un$}; 
\node[below of=2](4) {$\bigsqcup_{X\in \cD} \un$}; 

\draw[->] (1) to node[above,la]{$\iota_{A}$} (2);
\draw[->] (2) to node[right,la]{$(\Int F)_0$} (4);
\draw[->] (1) to node[below,la,xshift=-6pt]{$\iota_{FA}$} (4);

\node[right of=2,xshift=2.2cm](1) {$\cC(A,B)$}; 
\node[below of=1](3) {$\cD(FA,FB)$};
\node[right of=1,xshift=2.6cm](2) {$\bigsqcup_{A,B\in \cC} \cC(A,B)$}; 
\node[below of=2](4) {$\bigsqcup_{X,Y\in \cD} \cD(X,Y)$}; 

\draw[->] (1) to node[above,la]{$\iota_{A,B}$} (2);
\draw[->] (2) to node[right,la]{$(\Int F)_1$} (4);
\draw[->] (1) to node[left,la]{$F_{A,B}$} (3);
\draw[->] (3) to node[below,la]{$\iota_{FA,FB}$} (4);
\end{tz}
Compatibility with source and target are straightforward, and compatibility with identity and composition follows from the respective properties of $F$. 
\end{rome}
\end{defn}

\begin{defn} 
The functor $\Und\colon \CatV\to \VCat$ sends an internal category $\bA$ to $\cV$ to its \textbf{underlying $\cV$-category} $\Und\bA$ whose 
\begin{rome}
\item object set is the set $\cV(\un,\bA_0)$, 
\item hom-object $\Und\bA(A,B)$ is given by the following pullback in $\cV$, 
\begin{tz}
\node[](1) {$\Und\bA(A,B)$}; 
\node[below of=1](3) {$\un$};
\node[right of=1,xshift=1.5cm](2) {$\bA_1$}; 
\node[below of=2](4) {$\bA_0\times \bA_0$}; 
\pullback{1};

\draw[->] (1) to node[above,la]{$\pi_{A,B}$} (2);
\draw[->] (2) to node[right,la]{$(s,t)$} (4);
\draw[->] (1) to (3);
\draw[->] (3) to node[below,la]{$(A,B)$} (4);
\end{tz} 
for all maps $A,B\colon \un\to \bA_0$ in $\cV$,
\item composition map $c_{A,B,C}\colon \Und\bA(A,B)\times \Und\bA(B,C)\to \Und\bA(A,C)$ is given by the unique map in $\cV$ making the following diagram commute, 
\begin{tz}
\node[](0) {$\Und\bA(A,B)\times \Und\bA(B,C)$};
\node[right of=0,xshift=3.7cm] (9) {$\bA_1\times_{\bA_0}\bA_1$};
\node[below of=0,xshift=2cm](1) {$\Und\bA(A,C)$}; 
\node[below of=1](3) {$\un$};
\node[right of=1,xshift=3.7cm](2) {$\bA_1$}; 
\node[below of=2](4) {$\bA_0\times \bA_0$}; 
\pullback{1};

\draw[->,bend right=30] (0) to node[left,la,xshift=-2pt]{$!$} (3);
\draw[->] (0) to node[right,la,yshift=3pt]{$c_{A,B,C}$} (1);
\draw[->] (0) to node[above,la]{$\pi_{A,B}\times \pi_{B,C}$} (9);
\draw[->] (9) to node[above,la,yshift=3pt]{$c$} (2);
\draw[->] (1) to node[above,la]{$\pi_{A,C}$} (2);
\draw[->] (2) to node[right,la]{$(s,t)$} (4);
\draw[->] (1) to (3);
\draw[->] (3) to node[below,la]{$(A,C)$} (4);
\end{tz}
for all maps $A,B,C\colon \un\to \bA_0$ in $\cV$,
\item identity map $i_A\colon \un\to \Und\bA(A,A)$ is given by the unique map in $\cV$ making the following diagram commute, 
\begin{tz}
\node[](0) {$\un$};
\node[right of=0,xshift=1.5cm] (9) {$\bA_0$};
\node[below of=0,xshift=2cm](1) {$\Und\bA(A,A)$}; 
\node[below of=1](3) {$\un$};
\node[right of=1,xshift=1.5cm](2) {$\bA_1$}; 
\node[below of=2](4) {$\bA_0\times \bA_0$};
\pullback{1};

\draw[d,bend right=30] (0) to (3);
\draw[->] (0) to node[right,la,yshift=6pt]{$i_A$} (1);
\draw[->] (0) to node[above,la]{$A$} (9);
\draw[->] (9) to node[above,la,yshift=3pt]{$i$} (2);
\draw[->] (1) to (2);
\draw[->] (2) to node[right,la]{$(s,t)$} (4);
\draw[->] (1) to (3);
\draw[->] (3) to node[below,la]{$(A,A)$} (4);
\end{tz}
for every map $A\colon \un\to \bA_0$ in $\cV$.
\end{rome}
Unitality and associativity of compositions follow from the respective properties in $\bA$. 

On morphisms, the functor $\Und$ sends an internal functor $H\colon \bA\to \bB$ to the $\cV$-functor $\Und H\colon \Und\bA\to \Und\bB$ given by $\cV(\un,H_0)\colon \cV(\un,\bA_0)\to \cV(\un,\bB_0)$ on objects, and such that $(\Und H)_{A,B}\colon \Und\bA(A,B)\to \Und\bB(HA,HB)$ is the unique map in $\cV$ making the following diagram commute, 
\begin{tz}
\node[](0) {$\Und\bA(A,B)$};
\node[right of=0,xshift=2.2cm] (9) {$\bA_1$};
\node[below of=0,xshift=2cm](1) {$\Und\bB(HA,HB)$}; 
\node[below of=1](3) {$\un$};
\node[right of=1,xshift=2.2cm](2) {$\bB_1$}; 
\node[below of=2](4) {$\bB_0\times \bB_0$}; 
\pullback{1};

\draw[->,bend right=30] (0) to node[left,la,xshift=-2pt]{$!$} (3);
\draw[->] (0) to node[right,la,yshift=6pt]{$(\Und H)_{A,B}$} (1);
\draw[->] (0) to node[above,la]{$\pi_{A,B}$} (9);
\draw[->] (9) to node[above,la,yshift=3pt]{$H_1$} (2);
\draw[->] (1) to node[above,la,pos=0.4]{$\pi_{HA,HB}$} (2);
\draw[->] (2) to node[right,la]{$(s,t)$} (4);
\draw[->] (1) to (3);
\draw[->] (3) to node[below,la]{$(HA,HB)$} (4);
\end{tz}
for all maps $A,B\colon \un\to \bA_0$ in $\cV$. Compatibility with identities and compositions follows from the respective properties of $H$.
\end{defn}

The following result appears as \cite[Proposition 4.2]{Extensive}.

\begin{prop} \label{prop:IntUnd}
Suppose that, for every set $I$, the functor $\bigsqcup_{i\in I}\colon \prod_{i\in I} \cV\to \slice{\cV}{\bigsqcup_{i\in I} \un}$ preserves binary products. Then there is an adjunction
\begin{tz}
\node[](1) {$\VCat$}; 
\node[right of=1,xshift=1.5cm](2) {$\CatV$}; 
\punctuation{2}{.};

\draw[->] ($(1.east)+(0,5pt)$) to node[above,la]{$\Int$} ($(2.west)+(0,5pt)$);
\draw[->] ($(2.west)-(0,5pt)$) to node[below,la]{$\Und$} ($(1.east)-(0,5pt)$);

\node[la] at ($(1.east)!0.5!(2.west)$) {$\bot$};
\end{tz}
\end{prop}

\begin{rem} \label{rem:Intffiffunconnected} 
    The functor $\Int\colon \VCat\to \CatV$ is not necessarily fully faithful, as given a $\cV$-category $\cC$, the canonical map $\Ob\cC\to \cV(\un,\bigsqcup_{A\in\cC} \un)$ may not be an isomorphism of sets. In fact, the functor $\Int$ is fully faithful precisely when the terminal object $\un$ is \emph{connected}; i.e., when the functor $\cV(\un,-)\colon \cV\to \Set$ preserves coproducts. This condition is not satisfied, for instance, when we consider a presheaf category over a non-connected category, such as $\cV=\Set^{\mathbbm{1}\sqcup\mathbbm{1}}\cong \Set\times \Set$, where $\mathbbm{1}$ denotes the terminal category.
\end{rem}

\begin{ex} \label{ex:functors}
    Let us describe these functors for certain choices of $\cV$, using the terminology from \cref{ex:enriched,ex:internal}.
\begin{enumerate}
    \item When $\cV=\Set$, the internalization functor $\Int\colon\Cat\to\Cat$ is simply the identity functor. This is also the case for the functor $\Und\colon\Cat\to\Cat$.
    \item When $\cV=\Cat$, the internalization functor $\Int\colon2\Cat\to\Dbl\Cat$ takes a 2-category~$\cC$ to the double category whose objects are the objects of $\cC$, vertical morphisms are identities, horizontal morphisms are the morphisms of $\cC$, and squares are the 2-morphisms of $\cC$. This is also known in the literature as the horizontal full embedding functor. 

    The functor $\Und\colon\Dbl\Cat\to2\Cat$ takes a double category $\bA$ to its underlying horizontal 2-category; that is, to the 2-category whose objects are the objects of $\bA$, morphisms are the horizontal morphisms of $\bA$, and 2-morphisms are given by the squares in $\bA$ with vertical identity boundaries. 
    \item[($n$)] When $\cV=(n-1)\Cat$ for $n\geq 3$, the internalization functor $\Int\colon n\Cat\to \Dbl(n-1)\Cat$ is also a full embedding by \cref{rem:Intffiffunconnected} and identifies $n$-categories with double $(n-1)$-categories with a discrete $(n-1)$-category of objects.
\end{enumerate}
Note that in all these examples the functor $\Int$ preserves binary products, a fact that will be used later.
\end{ex}

 \subsection{Extensivity} \label{subsec:extensive}

 In order for the constructions in this paper to work as expected, we will require the base category $\cV$ to satisfy an extensivity property, which we now recall.

 \begin{defn}\label{extensivity}
    A category $\cV$ is \textbf{extensive} if it has all small coproducts, and for every set $I$ and every family of objects $\{X_i\}_{i\in I}$ in $\cV$, the functor 
    \[ \textstyle \bigsqcup_{i\in I}\colon \prod_{i\in I} \slice{\cV}{X_i}\to \slice{\cV}{\bigsqcup_{i\in I} X_i}  \]
    is an equivalence of categories.
\end{defn}

\begin{ex} 
    The class of extensive categories includes, among others, the categories $\Set$, $\Cat$, and $n\Cat$; the category of topological spaces; any category of presheaves\textemdash e.g.~the category of simplicial sets\textemdash; and  more generally, any Grothendieck topos\textemdash e.g.~a category of sheaves. Moreover, if $\cV$ has pullbacks, then the fact that $\cV$ is extensive implies that $\VCat$ and $\Cat(\cV)$ are as well, by \cite[\S 2]{Extensive}.
\end{ex}

\begin{rem}
     In particular, note that when $\cV$ is extensive, then for every set $I$ the functor $\bigsqcup_{i\in I}\colon \prod_{i\in I} \cV\to \slice{\cV}{\bigsqcup_{i\in I} \un}$ is an equivalence and so it preserves binary products as required for the existence of the adjunction in \cref{prop:IntUnd}. 
\end{rem}

\begin{rem}
 \label{rem:extensive}
 If $\cV$ is a category with pullbacks along coproduct inclusions, then, for every set~$I$ and every family of objects $\{X_i\}_{i\in I}$ in~$\cV$, the functor
    \[ \textstyle \bigsqcup_{i\in I}\colon \prod_{i\in I} \slice{\cV}{X_i}\to \slice{\cV}{\bigsqcup_{i\in I} X_i}  \] 
    admits a right adjoint. This right adjoint sends a map $g\colon Y\to \bigsqcup_{i\in I} X_i$ in $\cV$ to the family of maps $\{g_i\colon Y_i\to X_i\}_{i\in I}$ in $\cV$, where the map $g_i\colon Y_i\to X_i$ is obtained as the following pullback in $\cV$, for every $i\in I$. 
    \begin{diagram} \label{pullbackYi}
\node[](1) {$Y_i$}; 
\node[below of=1](3) {$X_i$};
\node[right of=1,xshift=1.1cm](2) {$Y$}; 
\node[below of=2](4) {$\bigsqcup_{i\in I} X_i$}; 
\pullback{1};

\draw[->] (1) to (2);
\draw[->] (2) to node[right,la]{$g$} (4);
\draw[->] (1) to node[left,la]{$g_i$} (3);
\draw[->] (3) to node[below,la]{$\iota_{i}$} (4);
\end{diagram}
When $\cV$ is extensive, the fact that $\bigsqcup_{i\in I}$ is an equivalence of categories implies that the counit of this adjunction is an isomorphism; i.e., there is a canonical isomorphism $\bigsqcup_{i\in I} Y_i\cong Y$ in $\cV$ over $\bigsqcup_{i\in I} X_i$.
\end{rem}

\begin{notation}\label{defn:oplus}
    Let $\alpha\colon I\to J$ be a map of sets, $\{X_i\}_{i\in I}$ and $\{Y_j\}_{j\in J}$ be families of objects in~$\cV$, and $\{g_i\colon X_i\to Y_{\alpha(i)}\}_{i\in I}$ be a family of maps in $\cV$. We define
\[ \textstyle \bigsqcup_{\alpha}\, g_i\colon \bigsqcup_{i\in I} X_i\to \bigsqcup_{j\in J} Y_j \]
to be the unique map in $\cV$ given by the universal property of the coproduct making the following square in $\cV$ commute, for all $i\in I$. 
\begin{tz}
    \node[](1) {$X_i$}; 
    \node[right of=1,xshift=1.2cm](2) {$\bigsqcup_{i\in I} X_i$}; 
    \node[below of=1](3) {$Y_{\alpha(i)}$}; 
    \node[below of=2](4) {$\bigsqcup_{j\in J} Y_j$}; 

    \draw[->] (1) to node[above,la]{$\iota_i$} (2); 
    \draw[->] (3) to node[below,la]{$\iota_{\alpha(i)}$} (4); 
    \draw[->] (1) to node[left,la]{$g_i$} (3); 
    \draw[->] (2) to node[right,la]{$\bigsqcup_{\alpha}\, g_i$} (4); 
\end{tz}
\end{notation} 

\begin{lemma} \label{extensivelemma}
    Let $\cV$ be an extensive category. Let $\alpha\colon I\to J$ be a map of sets, $\{X_i\}_{i\in I}$ and $\{Y_j\}_{j\in J}$ be families of objects in $\cV$, and $\{g_i\colon X_i\to Y_{\alpha(i)}\}_{i\in I}$ be a family of maps in~$\cV$. Consider a commutative square in $\cV$ as follows.
    \begin{tz}
\node[](1) {$\bigsqcup_{i\in I} W_i$}; 
\node[below of=1](3) {$\bigsqcup_{i\in I} X_i$};
\node[right of=1,xshift=1.6cm](2) {$\bigsqcup_{j\in J} Z_j$}; 
\node[below of=2](4) {$\bigsqcup_{j\in J} Y_j$};

\draw[->] (1) to node[above,la]{$f$} (2);
\draw[->] (2) to node[right,la]{$\bigsqcup_{j\in J} k_j$} (4);
\draw[->] (1) to node[left,la]{$\bigsqcup_{i\in I} h_i$} (3);
\draw[->] (3) to node[below,la]{$\bigsqcup_\alpha g_i$} (4);
\end{tz}
Then there is a unique family $\{f_i\colon W_i\to Z_{\alpha(i)}\}_{i\in I}$ of maps in $\cV$ such that the map $f$ is given by the induced map
\[ \textstyle \bigsqcup_{\alpha} f_i\colon \bigsqcup_{i\in I} W_i\to \bigsqcup_{j\in J} Z_j. \]
and the following square in $\cV$ commutes, for all $i\in I$.
\begin{tz}
\node[](1) {$W_i$}; 
\node[below of=1](3) {$X_i$};
\node[right of=1,xshift=.9cm](2) {$Z_{\alpha(i)}$}; 
\node[below of=2](4) {$Y_{\alpha(i)}$}; 

\draw[->] (1) to node[above,la]{$f_i$} (2);
\draw[->] (2) to node[right,la]{$k_{\alpha(i)}$} (4);
\draw[->] (1) to node[left,la]{$h_i$} (3);
\draw[->] (3) to node[below,la]{$g_i$} (4);
\end{tz}
\end{lemma}

\begin{proof}
    Since $I=\bigsqcup_{j\in J} \alpha^{-1}(j)$, we can rewrite $\bigsqcup_{i\in I} W_i\cong \bigsqcup_{j\in J}\bigsqcup_{i\in \alpha^{-1}(j)} W_i$. Then, by fully faithfulness of the functor $\bigsqcup_{j\in J}\colon \prod_{j\in J} \slice{\cV}{Y_j}\to \slice{\cV}{\bigsqcup_{j\in J} Y_j}$, there is a unique family
    \[ \textstyle\{ \overline{f}_j\colon \bigsqcup_{i\in \alpha^{-1}(j)} W_i\to Z_j\}_{j\in J} \]
    of maps in $\cV$ such that the map $f$ is given by the induced map 
    \[ \textstyle\bigsqcup_{j\in J}\overline{f}_j\colon \bigsqcup_{j\in J}\bigsqcup_{i\in \alpha^{-1}(j)} W_i\to \bigsqcup_{j\in J} Z_j \] 
    and the following square in $\cV$ commutes, for all $j\in J$. 
    \begin{tz}
\node[](1) {$\bigsqcup_{i\in \alpha^{-1}(j)} W_i$}; 
\node[below of=1](3) {$\bigsqcup_{i\in \alpha^{-1}(j)} X_i$};
\node[right of=1,xshift=1.65cm](2) {$Z_{\alpha(i)}$}; 
\node[below of=2](4) {$Y_{\alpha(i)}$}; 

\draw[->] (1) to node[above,la]{$\overline{f}_j$} (2);
\draw[->] (2) to node[right,la]{$k_{\alpha(i)}$} (4);
\draw[->] (1) to node[left,la]{$\bigsqcup_{i\in \alpha^{-1}(j)} h_i$} (3);
\draw[->] (3) to node[below,la]{$\bigsqcup_! g_i$} (4);
\end{tz}
Hence, by the universal property of the coproduct, we get the desired result. 
\end{proof}

\section{Internal Grothendieck construction}\label{sec:groth}

In this section, let $\cV$ denote a cartesian closed category which admits pullbacks and small coproducts. The goal is to construct a generalization of the category of elements that one can apply to $\cV$-functors $\cC^\op\to\cV$. As explained in the introduction, this construction cannot be done entirely within the setting of enriched categories, and instead requires a passage to internal categories. In \cref{subsec:defnGC} we give the desired internal Grothendieck construction for a $\cV$-category $\cC$, and in \cref{subsec:naturality} we show that it is natural in $\cC$. 

\subsection{Definition}  \label{subsec:defnGC}

Let us fix a $\cV$-category $\cC$. We first construct the desired~functor. 

\begin{defn}\label{defn:groth}
The \textbf{internal Grothendieck construction} is the functor 
\[ \textstyle\int_\cC\colon \VCat(\cC^{\op},\cV)\to \slice{\CatV}{\Int\cC} \]
defined as follows. It sends a $\cV$-functor $F\colon \cC^{\op}\to \cV$ to its \textbf{internal category of elements}~$\int_\cC F$ to~$\cV$ whose 
\begin{rome}
\item object of objects and object of morphisms are given by the coproducts in $\cV$ 
\[ \textstyle(\int_\cC F)_0=\bigsqcup_{A\in \cC} FB \quad \text{and} \quad (\int_\cC F)_1=\bigsqcup_{A,B\in \cC} \cC(A,B)\times FB, \]
\item source and target maps $s,t\colon (\int_\cC F)_1\to (\int_\cC F)_0$ are given by the unique maps in~$\cV$ making the following diagrams commute, for all objects $A,B\in \cC$, 
\end{rome}
\begin{tz}
\node[](1) {$\cC(A,B)\times FB$}; 
\node[below of=1](3) {$FA$};
\node[right of=1,xshift=2.8cm](2) {$\bigsqcup_{A,B\in \cC} \cC(A,B)\times FB$};
\node[below of=2](4) {$\bigsqcup_{A\in \cC} FA$}; 

\draw[->] (1) to node[above,la]{$\iota_{A,B}$} (2);
\draw[->] (2) to node[right,la]{$s$} (4);
\draw[->] (1) to node[left,la]{$\ev^F_{A,B}$} (3);
\draw[->] (3) to node[below,la]{$\iota_{A}$} (4);

\node[right of=2,xshift=2.5cm](1) {$\cC(A,B)\times FB$}; 
\node[below of=1](3) {$FB$};
\node[right of=1,xshift=2.8cm](2) {$\bigsqcup_{A,B\in \cC} \cC(A,B)\times FB$};
\node[below of=2](4) {$\bigsqcup_{B\in \cC} FB$}; 

\draw[->] (1) to node[above,la]{$\iota_{A,B}$} (2);
\draw[->] (2) to node[right,la]{$t$} (4);
\draw[->] (1) to node[left,la]{$\pi_1$} (3);
\draw[->] (3) to node[below,la]{$\iota_B$} (4);
\end{tz}
 \hspace{1.12cm} where $\ev^F_{A,B}$ is the map introduced in \cref{notation:evmap},
 \begin{enumerate}[label=(\roman*), start=3]
\item composition map $c\colon (\int_\cC F)_1\times_{(\int_\cC F)_0} (\int_\cC F)_1\to (\int_\cC F)_1$ is given by the unique map in $\cV$ making the following diagram commute, for all objects $A,B,C\in \cC$, 
\begin{tz}
\node[](1) {$\cC(A,B)\times \cC(B,C)\times FC$};
\node[below of=1](3) {$\cC(A,C)\times FC$};
\node[right of=1,xshift=5.2cm](2) {$\bigsqcup_{A,B,C\in \cC} \cC(A,B)\times \cC(B,C)\times FC$};
\node[below of=2](4) {$\bigsqcup_{A,B\in \cC} \cC(A,B)\times FB$}; 

\draw[->] (1) to node[above,la]{$\iota_{A,B,C}$} (2);
\draw[->] (2) to node[right,la]{$c$} (4);
\draw[->] (1) to node[left,la]{$c_{A,B,C}\times\id_{FC}$} (3);
\draw[->] (3) to node[below,la]{$\iota_{A,C}$} (4);
\end{tz}
\item identity map $i\colon (\int_\cC F)_0\to (\int_\cC F)_1$ is given by the unique map in $\cV$ making the following diagram commute, for every object $A\in \cC$.
\begin{tz}
\node[](1) {$FA$}; 
\node[below of=1](3) {$\cC(A,A)\times FA$};
\node[right of=1,xshift=3.2cm](2) {$\bigsqcup_{A\in \cC} FA$};
\node[below of=2](4) {$\bigsqcup_{A,B\in \cC} \cC(A,B)\times FB$}; 

\draw[->] (1) to node[above,la]{$\iota_{A}$} (2);
\draw[->] (2) to node[right,la]{$i$} (4);
\draw[->] (1) to node[left,la]{$i_{A}\times \id_{FA}$} (3);
\draw[->] (3) to node[below,la]{$\iota_{A,A}$} (4);
\end{tz}
\end{enumerate}
Compatibility with target is straightforward, compatibility with source follows from the $\cV$-functoriality of $F$ as described in \cref{lem:VfunctortoV}, and unitality and associativity of composition follows from the respective properties in $\cC$. 

This comes with a projection internal functor $\pi_F\colon \int_\cC F\to \Int\cC$ such that $(\pi_F)_0$ and $(\pi_F)_1$ are given by the unique maps in $\cV$ making the following diagrams commute, for all objects~$A,B\in \cC$.
\begin{tz}
\node[](1) {$FA$}; 
\node[below of=1](3) {$\un$};
\node[right of=1,xshift=1.3cm](2) {$\bigsqcup_{A\in \cC} FA$};
\node[below of=2](4) {$\bigsqcup_{A\in \cC} \un$}; 

\draw[->] (1) to node[above,la]{$\iota_{A}$} (2);
\draw[->] (2) to node[right,la]{$(\pi_F)_0$} (4);
\draw[->] (1) to node[left,la]{$!$} (3);
\draw[->] (3) to node[below,la]{$\iota_{A}$} (4);

\node[right of=2,xshift=2.2cm](1) {$\cC(A,B)\times FB$}; 
\node[below of=1](3) {$\cC(A,B)$};
\node[right of=1,xshift=3.2cm](2) {$\bigsqcup_{A,B\in \cC} \cC(A,B)\times FB$};
\node[below of=2](4) {$\bigsqcup_{A,B\in \cC} \cC(A,B)$}; 

\draw[->] (1) to node[above,la]{$\iota_{A,B}$} (2);
\draw[->] (2) to node[right,la]{$(\pi_F)_1$} (4);
\draw[->] (1) to node[left,la]{$\pi_0$} (3);
\draw[->] (3) to node[below,la]{$\iota_{A,B}$} (4);
\end{tz}

On morphisms, the functor $\int_\cC$ sends a $\cV$-natural transformation $\alpha\colon F\Rightarrow G$ to the internal functor $\int_\cC\alpha\colon \int_\cC F\to \int_\cC G$ such that $(\int_\cC\alpha)_0$ and $(\int_\cC\alpha)_1$ are given by the unique maps in~$\cV$ making the following diagrams commute, for all objects $A,B\in \cC$, 
\begin{tz}
\node[](1) {$FA$}; 
\node[below of=1](3) {$GA$};
\node[right of=1,xshift=1.3cm](2) {$\bigsqcup_{A\in \cC} FA$};
\node[below of=2](4) {$\bigsqcup_{A\in \cC} GA$}; 

\draw[->] (1) to node[above,la]{$\iota_{A}$} (2);
\draw[->] (2) to node[right,la]{$(\int_\cC\alpha)_0$} (4);
\draw[->] (1) to node[left,la]{$\ev^\alpha_A$} (3);
\draw[->] (3) to node[below,la]{$\iota_{A}$} (4);

\node[right of=2,xshift=2.2cm](1) {$\cC(A,B)\times FB$}; 
\node[below of=1](3) {$\cC(A,B)\times GB$};
\node[right of=1,xshift=3.2cm](2) {$\bigsqcup_{A,B\in \cC} \cC(A,B)\times FB$};
\node[below of=2](4) {$\bigsqcup_{A,B\in \cC} \cC(A,B)\times GB$}; 

\draw[->] (1) to node[above,la]{$\iota_{A,B}$} (2);
\draw[->] (2) to node[right,la]{$(\int_\cC\alpha)_1$} (4);
\draw[->] (1) to node[left,la]{$\id_{\cC(A,B)}\times \ev^\alpha_B$} (3);
\draw[->] (3) to node[below,la]{$\iota_{A,B}$} (4);
\end{tz}
where $\ev^\alpha_A$ is the map introduced in \cref{notn:evalpha}. Compatibility with target, composition, and identity is straightforward, and compatibility with source follows from the $\cV$-naturality of $\alpha$ as described in \cref{lem:VnattoV}.

Note that $\int_\cC\alpha$ is such that the following triangle in $\CatV$ commutes.
\begin{tz}
\node[](1) {$\int_\cC F$}; 
\node[below of=1,xshift=2cm,yshift=.2cm](3) {$\Int\cC$};
\node[above of=3,xshift=2cm,yshift=-.2cm](2) {$\int_\cC G$}; 

\draw[->] (1) to node[above,la]{$\int_\cC \alpha$} (2);
\draw[->] (2) to node[right,la,yshift=-4pt]{$\pi_G$} (3);
\draw[->] (1) to node[left,la,yshift=-4pt]{$\pi_F$} (3);
\end{tz}
\end{defn}

\begin{ex} \label{ex:catofelements}
    Going back to our main examples, we have the following interpretation of the internal category of elements. 
    \begin{enumerate}
        \item When $\cV=\Set$, given a functor $F\colon \cC^{\op}\to \Set$, the category $\int_\cC F$ is the usual category of elements considered e.g.~in \cite[Definition 2.4.2]{Riehlcontext}. It is such that
        \begin{rome}
            \item its objects are pairs $(A,x)$ of an object $A\in \cC$ and an element $x\in FA$,
            \item its morphisms $(A,x)\to(B,y)$ are morphisms $f\colon A\to B$ in $\cC$ such that $Ff(y)=x$.
        \end{rome}
        \item When $\cV=\Cat$, given a $2$-functor $F\colon \cC^{\op}\to \Cat$, the double category $\int_\cC F$ coincides with the double category of elements considered in \cite[\S 1.2]{GraParPersistentII}. It is such that 
        \begin{rome}
            \item its objects are pairs $(A,x)$ of objects $A\in \cC$ and $x\in FA$, and its vertical morphisms are pairs $(A,u)$ of an object $A\in \cC$ and a morphism $u\colon x\to x'$ in $FA$, 
            \item its horizontal morphisms $(A,x)\to(B,y)$ are morphisms $f\colon A\to B$ in $\cC$ such that $Ff(y)=x$, and its squares from $(A,u)$ to $(B,v)$ are $2$-morphisms $\alpha\colon f\Rightarrow g\colon A\to B$ in $\cC$ such that $Fg(v)\circ (F\alpha)_y=u$, where $y$ denotes the source of $v$.
        \end{rome}
        \item[($n$)] When $\cV=(n-1)\Cat$ for $n\geq 3$, given an $n$-functor $F\colon \cC^{\op}\to (n-1)\Cat$, the double $(n-1)$-category $\int_\cC F$ is such that
        \begin{rome}
            \item its $(n-1)$-category of objects has as objects pairs $(A,x)$ of objects $A\in \cC$ and $x\in FA$, and as $k$-morphisms pairs $(A,\varphi)$ of an object $A\in \cC$ and a $k$-morphism $\varphi$ in $FA$, for all $1\leq k\leq n-1$,
            \item its $(n-1)$-category of morphisms has as objects from $(A,x)$ to $(B,y)$ morphisms $f\colon A\to B$ in $\cC$ such that $Ff(y)=x$, and as $k$-morphisms from $(A,\varphi)$ to $(B,\psi)$ pairs of a $k$-morphism $\mu$ in $\cC(A,B)$---i.e., a $(k+1)$-morphism in $\cC$ whose source and target objects are $A$ and $B$---such that $Fg(\varphi)\circ_0 (F\mu)_y=\psi$, where $y$ denotes the source object of the $k$-morphism $\psi$ and $\circ_0$ denotes the composition along objects of $k$-morphisms in $\cC$, for all $1\leq k\leq n-1$.
        \end{rome}
    \end{enumerate}
\end{ex}

\subsection{Naturality} \label{subsec:naturality}

We further give the following change of base construction.

\begin{defn} \label{constr:changeofbase}
    Given $\cV$-functors $F\colon \cC\to \cD$ and $G\colon \cD^{\op}\to \cV$, using \cref{defn:oplus}, we define an internal functor $\int_F G\colon \int_\cC GF\to \int_\cD G$ such that $(\int_F G)_0$ is given by the map in $\cV$
    \[ \textstyle\bigsqcup_{\Ob F} \id_{GFA}\colon \bigsqcup_{A\in \cC} GFA\to \bigsqcup_{D\in \cD} GD \]
    and $(\int_F G)_1$ is given by the map in $\cV$
    \[ \textstyle\bigsqcup_{\Ob F\times \Ob F} F_{A,B}\times \id_{GFB}\colon  \bigsqcup_{A,B\in \cC} \cC(A,B)\times GFB\to \bigsqcup_{C,D\in \cD} \cD(C,D)\times GD. \]
    Note that $\int_F G$ is such that the following square in $\CatV$ commutes. 
    \begin{tz}
\node[](1) {$\int_\cC GF$}; 
\node[below of=1](3) {$\Int\cC$};
\node[right of=1,xshift=1.2cm](2) {$\int_\cD G$};
\node[below of=2](4) {$\Int\cD$}; 

\draw[->] (1) to node[above,la]{$\int_F G$} (2);
\draw[->] (2) to node[right,la]{$\pi_G$} (4);
\draw[->] (1) to node[left,la]{$\pi_{GF}$} (3);
\draw[->] (3) to node[below,la]{$\Int F$} (4);
\end{tz}
\end{defn}

In fact, the above square is a pullback, which yields the following naturality of the internal Grothendieck construction. 

\begin{prop}
For every $\cV$-functor $F\colon\cC\to \cD$, the following diagram in $\Cat$ commutes,
    \begin{tz}
\node[](1) {$\VCat(\cD^{\op}, \cV)$}; 
\node[below of=1](3) {$\VCat(\cC^{\op}, \cV)$};
\node[right of=1,xshift=2.5cm](2) {$\slice{\CatV}{\Int\cD}$};
\node[below of=2](4) {$\slice{\CatV}{\Int\cC}$}; 

\draw[->] (1) to node[above,la]{$\int_\cD$} (2);
\draw[->] (2) to node[right,la]{$\Int F^*$} (4);
\draw[->] (1) to node[left,la]{$F^*$} (3);
\draw[->] (3) to node[below,la]{$\int_\cC$} (4);
\end{tz}
where the left-hand functor is given by pre-composing with $F$, and the right-hand functor by taking pullbacks along $\Int F$.
\end{prop}

\section{Internal discrete fibrations}\label{sec:discfib}

In this section, let $\cV$ be a category with pullbacks. The internal Grothendieck construction is closely related to internal discrete fibrations, whose definition we recall in \cref{subsec:defndiscfib}. A strong understanding of the structure enforced by an internal discrete fibration on its domain, and of the behavior of internal functors between internal discrete fibrations, will be crucial for our goals; we study this in \cref{subsec:discfiboverC,subsec:mapofdiscfib} in the case where $\cV$ is extensive. We also show in \cref{subsec:props} that internal discrete fibrations are stable under pullbacks, and how it is possible to characterize an isomorphism between internal discrete fibrations simply in terms of its level $0$ or of its values on fibers. Finally, \cref{subsec:examples} contains three examples of internal discrete fibrations, given by slice internal categories, comma internal categories, and our internal category of elements.

\subsection{Definition} \label{subsec:defndiscfib}

 We recall the notion of internal discrete fibrations in $\CatV$.

\begin{defn} \label{defn:discfib}
    An internal functor $P\colon\bA\to\bB$ to $\cV$ is an \textbf{internal discrete fibration} if the following square in $\cV$ is a pullback.
\begin{tz}
\node[](1) {$\bA_1$}; 
\node[below of=1](3) {$\bB_1$};
\node[right of=1,xshift=.7cm](2) {$\bA_0$}; 
\node[below of=2](4) {$\bB_0$};
\pullback{1};

\draw[->] (1) to node[above,la]{$t$} (2);
\draw[->] (2) to node[right,la]{$P_0$} (4);
\draw[->] (1) to node[left,la]{$P_1$} (3);
\draw[->] (3) to node[below,la]{$t$} (4);
\end{tz}
\end{defn}

\begin{ex} \label{ex:discfib} 
    In the examples of interest, we have the following. 
    \begin{enumerate}
        \item When $\cV=\Set$, internal discrete fibrations coincide with the usual notion of discrete fibrations for categories; see e.g. \cite[Definition 2.1.3]{LoregianRiehl}.
        \item When $\cV=\Cat$, internal discrete fibrations coincide with the \emph{double discrete fibrations} considered in \cite[Definition 2.12]{Lambert}.
        \item[($n$)] When $\cV=(n-1)\Cat$ for $n\geq 3$, we refer to internal discrete fibrations as \emph{double $(n-1)$-discrete fibrations}.
    \end{enumerate}
\end{ex}

\begin{lemma} \label{lem:pullbackcomp}
    Let $P\colon \bA\to \bB$ be an internal discrete fibration in $\CatV$. Then the following square in $\cV$ is a pullback, 
    \begin{tz}
\node[](1) {$\bA_1\times_{\bA_0} \bA_1$}; 
\node[below of=1](3) {$\bB_1\times_{\bB_0}\bB_1$};
\node[right of=1,xshift=1.3cm](2) {$\bA_0$}; 
\node[below of=2](4) {$\bB_0$};
\pullback{1};

\draw[->] (1) to node[above,la]{$t\circ \pi_1$} (2);
\draw[->] (2) to node[right,la]{$P_0$} (4);
\draw[->] (1) to node[left,la]{$P_1\times_{P_0} P_1$} (3);
\draw[->] (3) to node[below,la]{$t\circ \pi_1$} (4);
\end{tz}
where the maps $\pi_1\colon \bA_1\times_{\bA_0} \bA_1\to \bA_1$ and $\pi_1\colon \bB_1\times_{\bB_0} \bB_1\to \bB_1$ denote the canonical projections onto the second factor. 
\end{lemma}

\begin{proof}
    By definition of $\bA_1\times_{\bA_0}\bA_1$ and $\bB_1\times_{\bB_0}\bB_1$ and the fact that $P$ is an internal discrete fibration, we have the following pullbacks in $\cV$.
    \begin{tz}
\node[](1) {$\bA_1\times_{\bA_0} \bA_1$}; 
\node[below of=1](3) {$\bA_1$};
\node[below of=3](5) {$\bB_1$};
\node[right of=1,xshift=1.3cm](2) {$\bA_1$}; 
\node[below of=2](4) {$\bA_0$};
\node[below of=4](6) {$\bB_0$};
\pullback{1};
\pullback{3};

\draw[->] (1) to node[above,la]{$\pi_1$} (2);
\draw[->] (2) to node[right,la]{$s$} (4);
\draw[->] (4) to node[right,la]{$P_0$} (6);
\draw[->] (1) to node[left,la]{$\pi_0$} (3);
\draw[->] (3) to node[left,la]{$P_1$} (5);
\draw[->] (3) to node[below,la]{$t$} (4);
\draw[->] (5) to node[below,la]{$t$} (6);

\node[right of=2,xshift=2cm](1) {$\bA_1\times_{\bA_0} \bA_1$}; 
\node[below of=1](3) {$\bB_1\times_{\bB_0}\bB_1$};
\node[below of=3](5) {$\bB_1$};
\node[right of=1,xshift=1.3cm](2) {$\bA_1$}; 
\node[right of=2,xshift=1cm](2') {$\bA_0$}; 
\node[below of=2'](4') {$\bB_0$};
\node[below of=2](4) {$\bB_1$};
\node[below of=4](6) {$\bB_0$};
\pullback{3};
\pullback{2};

\draw[->] (1) to node[above,la]{$\pi_1$} (2);
\draw[->] (2) to node[above,la]{$t$} (2');
\draw[->] (2) to node[left,la]{$P_1$} (4);
\draw[->] (2') to node[right,la]{$P_0$} (4');
\draw[->] (4) to node[right,la]{$s$} (6);
\draw[->] (1) to node[left,la]{$P_1\times_{P_0} P_1$} (3);
\draw[->] (3) to node[left,la]{$\pi_0$} (5);
\draw[->] (3) to node[below,la]{$\pi_1$} (4);
\draw[->] (4) to node[below,la]{$t$} (4');
\draw[->] (5) to node[below,la]{$t$} (6);
\end{tz}
Note that the maps in the left-hand outer rectangle agree with those in the vertical rectangle in the right-hand diagram above. Hence, since the rectangle on the left is a pullback, so is the one on the right. By cancellation of pullbacks, this implies that the top left square in the right-hand diagram is a pullback as well. Hence, the horizontal rectangle is also a pullback as desired. 
\end{proof}

\begin{notation}
    We denote by $\slice{\Dfib}{\bB}$ the full subcategory of $\slice{\CatV}{\bB}$ spanned by the internal discrete fibrations over $\bB$.
\end{notation}

\subsection{Structure of an internal discrete fibration} \label{subsec:discfiboverC}

In this subsection, suppose that $\cV$ is extensive. Of particular interest to us are internal discrete fibrations over $\Int\cC$, where $\cC$ is a $\cV$-category. In what follows, we consider an internal discrete fibration $P\colon\bA\to\Int\cC$ in $\CatV$, and investigate what this implies for the structure of $\bA$. We begin by studying the fibers of $P$. 

\begin{defn} 
    The \textbf{fiber of $P$} at an object $A\in \cC$ is defined as the following pullback in~$\Cat(\cV)$.
    \begin{tz}
\node[](1) {$P^{-1}A$}; 
\node[below of=1](3) {$\un$};
\node[right of=1,xshift=.9cm](2) {$\bA$}; 
\node[below of=2](4) {$\Int\cC$};
\pullback{1};

\draw[->] (1) to node[above,la]{} (2);
\draw[->] (2) to node[right,la]{$P$} (4);
\draw[->] (1) to node[left,la]{} (3);
\draw[->] (3) to node[below,la]{$A$} (4);
\end{tz}
\end{defn}

\begin{rem} \label{pullbacksinCatV}
    Recall that pullbacks in $\CatV$ are computed levelwise in $\cV$, by taking the pullbacks of the objects of objects, and of the objects of morphisms.
\end{rem}

    In analogy with the classical case for $\cV=\Set$, we expect our internal Grothendieck construction to exhibit an equivalence between internal discrete fibrations over $\Int\cC$ and $\cV$-functors $\cC^\op\to\cV$, where the inverse correspondence is given by taking fibers (this is the content of \cref{section:equiv}). One may think that the fiber of an internal discrete fibration contains too much information for this to hold, since a priori it has relevant data both at levels 0 and 1, but the following result shows this is not the case.

\begin{prop}
    The fiber $P^{-1} A$ at an object $A\in \cC$ is a constant internal category. 
\end{prop}

\begin{proof}
    This is a direct consequence of the fact that the maps in $\cV$ 
    \[ P_1\colon \bA_1\to \Int\cC_1 \quad \text{and} \quad P_1\times_{P_0} P_1\colon \bA_1\times_{\bA_0}\bA_1 \to \Int\cC_1\times_{\Int\cC_0} \Int\cC_1 \]
    are pullbacks of the map $P_0\colon \bA_0\to \Int\cC_0$ by definition of an internal discrete fibration and by \cref{lem:pullbackcomp}. 
\end{proof}

    For an object $A\in \cC$, since $P^{-1}A$ is a constant internal category, we also write $P^{-1} A$ for the following pullback in $\cV$, which is precisely its object of objects.
    \begin{diagram} \label{defn:fiber}
\node[](1) {$P^{-1}A$}; 
\node[below of=1](3) {$\un$};
\node[right of=1,xshift=.9cm](2) {$\bA_0$}; 
\node[below of=2](4) {$\bigsqcup_{A\in \cC} \un$};
\pullback{1};

\draw[->] (1) to node[above,la]{} (2);
\draw[->] (2) to node[right,la]{$P_0$} (4);
\draw[->] (1) to node[left,la]{} (3);
\draw[->] (3) to node[below,la]{$\iota_A$} (4);
    \end{diagram}

Then, by \cref{rem:extensive}, the facts that $\cV$ is extensive and that the bottom map in this pullback is a coproduct inclusion, for every object $A\in \cC$, imply that there is an isomorphism over $\bigsqcup_{A\in \cC} \un$ \[ \textstyle \varphi_0\colon \bigsqcup_{A\in\cC} P^{-1}A\cong \bA_0. \] 

We can also express $\bA_1$ in terms of the fibers of $P$, as we now explain. Since $P$ is an internal discrete fibration, we have the following commutative diagrams in $\cV$,
\begin{tz}
    \node[](1') {$\cC(A,B)\times P^{-1}B$};
    \node[below of=1'](3') {$\cC(A,B)$};
\node[right of=1',xshift=1.65cm](1) {$\bA_1$}; 
\node[below of=1](3) {$\bigsqcup_{A,B\in\cC}\cC(A,B)$};
\node[right of=1,xshift=1.55cm](2) {$\bA_0$}; 
\node[below of=2](4) {$\bigsqcup_{B\in\cC} \un$};
\pullback{1};

\draw[->] (1) to node[above,la]{$t$} (2);
\draw[->] (2) to node[right,la]{$P_0$} (4);
\draw[->] (1) to node[left,la]{$P_1$} (3);
\draw[->] (3) to node[below,la]{$t$} (4);
\draw[->,dashed] (1') to (1);
\draw[->] (1') to node[left,la]{$\pi_0$} (3');
\draw[->] (3') to node[below,la]{$\iota_{A,B}$} (3);

\node[right of=2,xshift=1cm](1') {$\cC(A,B)\times P^{-1}B$};
\node at ($(4)!0.5!(1')$) {$=$};
    \node[below of=1'](3') {$\cC(A,B)$};
\node[right of=1',xshift=1.55cm](1) {$P^{-1}B$}; 
\node[below of=1](3) {$\un$};
\node[right of=1,xshift=.45cm](2) {$\bA_0$}; 
\node[below of=2](4) {$\bigsqcup_{B\in\cC} \un$};
\pullback{1};
\pullback{1'};

\draw[->] (1) to (2);
\draw[->] (2) to node[right,la]{$P_0$} (4);
\draw[->] (1) to (3);
\draw[->] (3) to node[below,la]{$\iota_B$} (4);
\draw[->] (1') to node[above,la]{$\pi_1$} (1);
\draw[->] (1') to node[left,la]{$\pi_0$} (3');
\draw[->] (3') to (3);
\end{tz}
where three of the squares are pullbacks. By  cancellation of pullbacks, we get that the left-most square is also a pullback. As before, by \cref{rem:extensive} we get an isomorphism over~$\bigsqcup_{A,B\in \cC} \cC(A,B)$
\[  \textstyle \varphi_1\colon \bigsqcup_{A,B\in\cC} \cC(A,B)\times P^{-1}B\cong \bA_1. \]

A similar argument, using the pullback square from \cref{lem:pullbackcomp}, yields an isomorphism over~$\bigsqcup_{A,B,C\in \cC} \cC(A,B)\times \cC(B,C)$
\[  \textstyle \varphi_2\colon \bigsqcup_{A,B,C\in\cC} \cC(A,B)\times \cC(B,C)\times P^{-1}C\cong \bA_1\times_{\bA_0}\bA_1. \]

In particular, using the internal categorical structure of $\bA$, we get a unique internal category $\FibA$ to $\cV$
\begin{tz}
    \node[](1) {$\bigsqcup_{A\in \cC} P^{-1}A$}; 
    \node[right of=1,xshift=3.2cm](2) {$\bigsqcup_{A,B\in \cC} \cC(A,B)\times P^{-1}B$}; 
    \node[right of=2,xshift=5.3cm](3) {$\bigsqcup_{A,B,C\in \cC} \cC(A,B)\times \cC(B,C)\times P^{-1}C$};

    \draw[->] (1) to node[over,la]{$i$} (2);
    \draw[->] ($(2.west)+(0,5pt)$) to node[above,la]{$s$} ($(1.east)+(0,5pt)$); 
    \draw[->] ($(2.west)-(0,5pt)$) to node[below,la]{$t$} ($(1.east)-(0,5pt)$);
    \draw[->] (3) to node[above,la]{$c$} (2);
\end{tz}
such that the isomorphisms $\varphi_i$ with $i=0,1,2$ assemble into an invertible internal functor $\varphi\colon \FibA\cong \bA$. Moreover, we get an internal functor $\FibP\coloneqq P\circ \varphi \colon \FibA\to \Int\cC$ whose components are induced by the canonical projections.

We now study how the source, target, identity, and composition maps in $\FibA$ can be understood in terms of their relation with those of $\Int\cC$.

\subsubsection*{Target map in \texorpdfstring{$\FibA$}{Fib(A)}} 

Consider the following diagram in $\cV$,
\begin{tz}
\node[](5) {$\cC(A,B)\times P^{-1}B$}; 
\node[below of=5](6) {$P^{-1}B$};
\node[right of=5,xshift=3.7cm](1) {$\bigsqcup_{A,B\in\cC} \cC(A,B)\times P^{-1}B$}; 
\node[below of=1](3) {$\bigsqcup_{B\in\cC} P^{-1}B$};
\node[right of=1,xshift=2.5cm](2) {$\bA_1$};
\node[below of=2](4) {$\bA_0$};

\draw[->] (5) to node[above,la]{$\iota_{A,B}$} (1);
\draw[->] (5) to node[left,la]{$\pi_1$} (6);
\draw[->] (6) to node[below,la]{$\iota_B$} (3);
\draw[->] (1) to node[below,la]{$\varphi_1$} node[above,la]{$\cong$} (2);
\draw[->] (2) to node[right,la]{$t$} (4);
\draw[->] (1) to node[left,la]{$t$} (3);
\draw[->] (3) to node[below,la]{$\varphi_0$} node[above,la]{$\cong$} (4);
\end{tz} 
where the outer square commutes by definition of $\varphi_1$ and the right-hand square commutes since~$\varphi$ is an internal functor. Note that the left-hand square must also commute, since $\varphi_0$ is an isomorphism. Hence, using \cref{defn:oplus}, by the universal property of the coproduct, we get that the target map $t\colon \FibA_1\to \FibA_0$ must be given by the map  
\[ \textstyle \bigsqcup_{\mathrm{pr}_1} \pi_1\colon\bigsqcup_{A,B\in \cC} \cC(A,B)\times P^{-1} B\to \bigsqcup_{B\in \cC} P^{-1}B \]
induced by the projection $\mathrm{pr}_1\colon \Ob\cC\times \Ob\cC\to \Ob\cC$, $(A,B)\mapsto B$, and the canonical projections $\pi_1\colon \cC(A,B)\times P^{-1} B\to P^{-1}B$, for all objects $A,B\in \cC$.

\subsubsection*{Source map in \texorpdfstring{$\FibA$}{Fib(A)}}

Let $\mathrm{pr}_0\colon \Ob\cC\times \Ob\cC\to \Ob\cC$ denote the projection $(A,B)\mapsto A$ and consider the family of maps $\{!\colon \cC(A,B)\to \un\}_{A,B\in \cC}$ in $\cV$. Then, using \cref{defn:oplus}, the source map $s\colon (\Int\cC)_1\to (\Int\cC)_0$ is given by
        \[ \textstyle s=\bigsqcup_{\mathrm{pr}_0} !\colon \bigsqcup_{A,B\in \cC} \cC(A,B)\to \bigsqcup_{A\in \cC} \un. \]
From the commutativity of the following square in~$\cV$
\begin{tz}
\node[](1) {$\bigsqcup_{A,B\in\cC} \cC(A,B)\times P^{-1}B$}; 
\node[below of=1](3) {$\bigsqcup_{A,B\in\cC}\cC(A,B)$};
\node[right of=1,xshift=3.2cm](2) {$\bigsqcup_{A\in\cC} P^{-1}A$};
\node[below of=2](4) {$\bigsqcup_{A\in\cC} \un$}; 

\draw[->] (1) to node[above,la]{$s$} (2);
\draw[->] (2) to node[right,la]{$(\FibP)_0$} (4);
\draw[->] (1) to node[left,la]{$(\FibP)_1$} (3);
\draw[->] (3) to node[below,la]{$s=\bigsqcup_{\mathrm{pr}_0}!$} (4);
\end{tz}
 and using \cref{extensivelemma}, we see that there is a unique family of maps in $\cV$
 \setcounter{equation}{\value{theorem}}
\refstepcounter{theorem}\refstepcounter{equation}
\begin{equation} \label{evP-1}
    \{ \ev^{P^{-1}}_{A,B}\colon \cC(A,B)\times P^{-1}B\to P^{-1}A \}_{A,B\in \cC}
\end{equation} 
such that the source map $s\colon \FibA_1\to \FibA_0$ is given by the induced map 
\[ \textstyle \bigsqcup_{\mathrm{pr}_0} \ev^{P^{-1}}_{A,B}\colon \bigsqcup_{A,B\in\cC}\cC(A,B)\times P^{-1}B\to \bigsqcup_{A\in\cC}P^{-1}A. \]

\subsubsection*{Identity map in \texorpdfstring{$\FibA$}{Fib(A)}} 

Let $\mathrm{diag}\colon \Ob\cC\to \Ob\cC\times \Ob\cC$ denote the diagonal map $A\mapsto (A,A)$ and consider the family of maps $\{i_A\colon \un\to \cC(A,A)\}_{A\in \cC}$ in $\cV$. Then, using \cref{defn:oplus}, the identity map $i\colon (\Int\cC)_0\to (\Int\cC)_1$ is given by 
        \[ \textstyle i=\bigsqcup_{\mathrm{diag}} i_A\colon \bigsqcup_{A\in \cC} \un \to \bigsqcup_{A,B\in \cC} \cC(A,B). \]
        From the commutativity of the following diagram in $\cV$
\begin{tz}
\node[](1) {$\bigsqcup_{B\in\cC} P^{-1}B$}; 
\node[below of=1](3) {$\bigsqcup_{B\in\cC} \un$};
\node[right of=1, xshift=3.2cm](2) {$\bigsqcup_{A,B\in\cC} \cC(A,B)\times P^{-1}B$}; 
\node[below of=2](4) {$\bigsqcup_{A,B\in\cC}\cC(A,B)$}; 
\node[right of=2, xshift=3.2cm](2') {$\bigsqcup_{B\in\cC} P^{-1}B$};
\node[below of=2'](4') {$\bigsqcup_{B\in\cC} \un$};

\draw[->] (1) to node[above,la]{$i$} (2);
\draw[->] (2) to node[above,la]{$t$} (2');
\draw[->,bend left=15] (1) to node[above,la]{$\id$} (2');
\draw[->] (2) to node[right,la]{$(\FibP)_1$} (4);
\draw[->] (2') to node[right,la]{$(\FibP)_0$} (4');
\draw[->] (1) to node[left,la]{$(\FibP)_0$} (3);
\draw[->] (3) to node[below,la]{$i=\bigsqcup_{\mathrm{diag}} i_B$} (4);
\draw[->] (4) to node[below,la]{$t$} (4');
\end{tz}
and using \cref{extensivelemma}, there is a unique family of maps in $\cV$
\[ \{ \psi_{B}\colon P^{-1}B\to \cC(B,B)\times P^{-1} B\}_{B\in \cC} \]
such that the identity map $i\colon \FibA_0\to \FibA_1$ is given by the induced map 
\[ \textstyle \bigsqcup_{\mathrm{diag}} \psi_B\colon \bigsqcup_{B\in \cC} P^{-1} B\to \bigsqcup_{A,B\in \cC} \cC(A,B)\times P^{-1} B, \]
 and such that the following diagram in $\cV$ commutes, for all objects $B\in \cC$.
\begin{tz}
\node[](1) {$P^{-1}B$}; 
\node[below of=1](3) {$\un$};
\node[right of=1, xshift=2.2cm](2) {$\cC(B,B)\times P^{-1}B$}; 
\node[below of=2](4) {$\cC(B,B)$}; 
\node[right of=2, xshift=2.2cm](2') {$P^{-1}B$};
\node[below of=2'](4') {$\un$};

\draw[->] (1) to node[above,la]{$\psi_B$} (2);
\draw[->] (2) to node[above,la]{$\pi_1$} (2');
\draw[->,bend left=20] (1) to node[above,la]{$\id_{P^{-1}B}$} (2');
\draw[->] (2) to node[right,la]{$\pi_0$} (4);
\draw[->] (2') to (4');
\draw[->] (1) to (3);
\draw[->] (3) to node[below,la]{$i_B$} (4);
\draw[->] (4) to (4');
\end{tz}
In particular, from the commutativity of this diagram and the universal property of the product, we deduce that $\psi_B$ must be given by the map 
\[ i_B\times \id_{P^{-1}B}\colon P^{-1}B\to \cC(B,B)\times P^{-1} B, \]
and so the identity map $i\colon \FibA_0\to \FibA_1$ must be the induced map  
\[ \textstyle\bigsqcup_{\mathrm{diag}} i_B\times \id_{P^{-1}B}\colon \bigsqcup_{B\in \cC} P^{-1} B\to \bigsqcup_{A,B\in \cC} \cC(A,B)\times P^{-1} B. \]

\subsubsection*{Composition map in \texorpdfstring{$\FibA$}{Fib(A)}}

First consider the following diagram in $\cV$,
\begin{tz}
\node[](5) {$\cC(A,B)\times \cC(B,C)\times P^{-1}C$}; 
\node[below of=5](6) {$P^{-1}C$};
\node[right of=5,xshift=5.5cm](1) {$\bigsqcup_{A,B,C\in\cC} \cC(A,B)\times \cC(B,C)\times P^{-1}C$}; 
\node[below of=1](3) {$\bigsqcup_{C\in\cC} P^{-1}C$};
\node[right of=1,xshift=4cm](2) {$\bA_1\times_{\bA_0}\bA_1$};
\node[below of=2](4) {$\bA_0$};

\draw[->] (5) to node[above,la]{$\iota_{A,B,C}$} (1);
\draw[->] (5) to node[left,la]{$\pi_2$} (6);
\draw[->] (6) to node[below,la]{$\iota_C$} (3);
\draw[->] (1) to node[below,la]{$\varphi_2$} node[above,la]{$\cong$} (2);
\draw[->] (2) to node[right,la]{$t\circ \pi_1$} (4);
\draw[->] (1) to node[left,la]{$t\circ \pi_1$} (3);
\draw[->] (3) to node[below,la]{$\varphi_0$} node[above,la]{$\cong$} (4);
\end{tz}
where the outer square commutes by definition of $\varphi_2$ and the right-hand square commutes since $\varphi$ is an internal functor. Note that the left-hand square must also commute, since $\varphi_0$ is an isomorphism. Hence, using \cref{defn:oplus}, by the universal property of the coproduct, we get that the map $t\circ \pi_1\colon \FibA_1\to \FibA_0$ must be given by the map  
\[ \textstyle \bigsqcup_{\mathrm{pr}_2} \pi_2\colon\bigsqcup_{A,B,C\in \cC} \cC(A,B)\times \cC(B,C)\times P^{-1} C\to \bigsqcup_{C\in \cC} P^{-1}C \]
induced by the projection $\mathrm{pr}_2\colon \Ob\cC\times \Ob\cC\times \Ob\cC\to \Ob\cC$, $(A,B,C)\mapsto C$, and the canonical projection $\pi_2\colon \cC(A,B)\times \cC(B,C)\times P^{-1} C\to P^{-1}C$, for all objects $A,B,C\in \cC$.

Next, let $\mathrm{pr}_{0,2}\colon \Ob\cC\times \Ob\cC\times \Ob\cC\to \Ob\cC\times \Ob\cC$ denote the projection $(A,B,C)\mapsto (A,C)$ and consider the family of maps $\{c_{A,B,C}\colon \cC(A,B)\times \cC(B,C)\to \cC(A,C)\}_{A,B,C\in \cC}$ in $\cV$. Then, using \cref{defn:oplus}, the composition map $c\colon (\Int\cC)_1\times_{(\Int\cC)_0}(\Int\cC)_1\to (\Int\cC)_1$ is given by
        \[ \textstyle c=\bigsqcup_{\mathrm{pr}_{0,2}} c_{A,B,C}\colon \bigsqcup_{A,B,C\in \cC}\cC(A,B)\times \cC(B,C)\to \bigsqcup_{A,C\in \cC}\cC(A,C). \]
From the commutativity of the following diagram in $\cV$
\begin{tz}
\node[](1) {$\bigsqcup_{A,B,C\in\cC} \cC(A,B)\times\cC(B,C)\times P^{-1}C$}; 
\node[below of=1](3) {$\bigsqcup_{A,B,C\in\cC} \cC(A,B)\times\cC(B,C)$};
\node[right of=1,xshift=5.3cm](2) {$\bigsqcup_{A,C\in\cC} \cC(A,C)\times P^{-1}C$}; 
\node[below of=2](4) {$\bigsqcup_{A,C\in\cC}\cC(A,C)$}; 
\node[right of=2, xshift=3.2cm](2') {$\bigsqcup_{C\in\cC} P^{-1}C$};
\node[below of=2'](4') {$\bigsqcup_{C\in\cC} \un$};

\draw[->] (1) to node[above,la]{$c$} (2);
\draw[->] (2) to node[above,la]{$t$} (2');
\draw[->,bend left=15] (1) to node[above,la]{$t\circ \pi_1$} (2');
\draw[->] (2) to node[right,la]{$(\FibP)_1$} (4);
\draw[->] (2') to node[right,la]{$(\FibP)_0$} (4');
\draw[->] (1) to node[left,la]{$(\FibP)_1\times_{(\FibP)_0} (\FibP)_1$} (3);
\draw[->] (3) to node[below,la]{$c=\bigsqcup_{\mathrm{pr}_{0,2}} c_{A,B,C}$} (4);
\draw[->] (4) to node[below,la]{$t$} (4');
\end{tz}
 and using \cref{extensivelemma}, there is a unique family of maps in $\cV$ 
\[ \{\gamma_{A,B,C}\colon \cC(A,B)\times\cC(B,C)\times P^{-1}C\to \cC(A,C)\times P^{-1}C\}_{A,B,C\in \cC} \]
such that the composition map $c\colon \FibA_1\times_{\FibA_0} \FibA_1\to \FibA_1$
is given by the induced map 
\[ \textstyle \bigsqcup_{\mathrm{pr}_{0,2}} \gamma_{A,B,C} \colon \bigsqcup_{A,B,C\in \cC}\cC(A,B)\times\cC(B,C)\times P^{-1}C \to \bigsqcup_{A,C\in \cC}\cC(A,C)\times P^{-1}C, \]
 and such that the following diagram in $\cV$ commutes, for all objects $A,B,C\in \cC$.
\begin{tz}
\node[](1) {$\cC(A,B)\times\cC(B,C)\times P^{-1}C$}; 
\node[below of=1](3) {$\cC(A,B)\times\cC(B,C)$};
\node[right of=1,xshift=3.9cm](2) {$\cC(A,C)\times P^{-1}C$}; 
\node[below of=2](4) {$\cC(A,C)$}; 
\node[right of=2, xshift=2.1cm](2') {$P^{-1}C$};
\node[below of=2'](4') {$\un$};

\draw[->] (1) to node[above,la]{$\gamma_{A,B,C}$} (2);
\draw[->] (2) to node[above,la]{$\pi_1$} (2');
\draw[->,bend left=20] (1) to node[above,la]{$\pi_1$} (2');
\draw[->] (2) to node[right,la]{$(\FibP)_1$} (4);
\draw[->] (2') to (4');
\draw[->] (1) to node[left,la]{$\pi_0$} (3);
\draw[->] (3) to node[below,la]{$c_{A,B,C}$} (4);
\draw[->] (4) to (4');
\end{tz}
In particular, from the commutativity of this diagram and the universal property of the product, we deduce  that $\gamma_{A,B,C}$ must be given by the map 
\[ c_{A,B,C}\times \id_{P^{-1}C}\colon \cC(A,B)\times \cC(B,C)\times P^{-1}C\to \cC(A,C)\times P^{-1}C\]
and so the composition map $c\colon \FibA_1\times_{\FibA_0} \FibA_1\to \FibA_1$ must be the induced map 
\[ \textstyle\bigsqcup_{\mathrm{pr}_{0,2}}c_{A,B,C}\times \id_{P^{-1}C}\colon \bigsqcup_{A,B,C\in \cC}\cC(A,B)\times \cC(B,C)\times P^{-1}C\to \bigsqcup_{A,C\in \cC}\cC(A,C)\times P^{-1}C. \]

\subsection{Structure of an internal functor between internal discrete fibrations} \label{subsec:mapofdiscfib}

In this subsection, suppose that $\cV$ is extensive. Having studied the structure of internal discrete fibrations over $\Int\cC$, we now focus on the internal functors between these internal discrete fibrations. In what follows, we consider a commutative diagram in $\CatV$
\begin{tz}
\node[](1) {$\bA$}; 
\node[below of=1,xshift=1.8cm,yshift=.2cm](3) {$\Int\cC$};
\node[above of=3,xshift=1.8cm,yshift=-.2cm](2) {$\bB$}; 

\draw[->] (1) to node[above,la]{$H$} (2);
\draw[->] (2) to node[right,la,yshift=-4pt]{$P$} (3);
\draw[->] (1) to node[left,la,yshift=-4pt]{$Q$} (3);
\end{tz}
    with $P$ and $Q$ internal discrete fibrations over $\Int\cC$, and investigate what this implies for the structure of the internal functor $H$. 

Recall the isomorphisms $\varphi\colon \FibA\cong \bA$ and $\psi\colon \FibB\cong \bB$ constructed in \cref{subsec:discfiboverC}. We write $\FibF\colon\FibA\to \FibB$ for the composite
\[ \FibA\xrightarrow{\varphi} \bA\xrightarrow{H}\bB\xrightarrow{\psi^{-1}} \FibB. \]
We denote by $\ev^H_A\colon P^{-1} A\to Q^{-1} A$ the unique map given by the universal property of the pullback making the following diagram commute.
\begin{diagram}\label{defn:mapoffibers}
        \node[](0) {$P^{-1}A$};
        \node[right of=0,xshift=1cm](0'){$\bA_0$};
\node[below right of=0,xshift=1cm](1) {$Q^{-1}A$}; 
\node[below of=1](3) {$\un$};
\node[right of=1,xshift=1cm](2) {$\bB_0$}; 
\node[below of=2](4) {$\bigsqcup_{A\in \cC} \un$}; 
\pullback{1};

\draw[->,dashed] (0) to node[below,la,xshift=-8pt]{$\ev^H_A$} (1);
\draw[->] (0) to (0');
\draw[->] (0') to node[above,la]{$H_0$} (2);
\draw[->,bend right=30](0) to (3);

\draw[->] (1) to (2);
\draw[->] (2) to node[right,la]{$Q_0$} (4);
\draw[->] (1) to (3);
\draw[->] (3) to node[below,la]{$A$} (4);
\end{diagram}

Note that since the maps $P^{-1}A\to \bA_0$ and $Q^{-1} A\to \bB_0$ determine the isomorphisms $\varphi_0$ and~$\psi_0$, the maps $\ev_A^H$ correspond to the components of the map $\FibF_0$. By the universal property of the coproduct, the map $\FibF_0\colon\FibA_0\to \FibB_0$ must be given by the induced map over $\bigsqcup_{A\in \cC}\un$
\[ \textstyle\bigsqcup_{A\in \cC} \ev^H_A\colon \bigsqcup_{A\in \cC} P^{-1} A\to\bigsqcup_{A\in \cC} Q^{-1} A. \]

Now, since the functor $\bigsqcup_{A,B\in\cC}\colon \prod_{A,B\in\cC}\slice{\cV}{\cC(A,B)}\to \slice{\cV}{\bigsqcup_{A,B\in \cC} \cC(A,B)}$ is fully faithful, there is a unique family of maps in $\cV$ over $\cC(A,B)$
\[ \{\phi_{A,B}\colon \cC(A,B)\times P^{-1} B\to \cC(A,B)\times Q^{-1} B\}_{A,B\in \cC} \]
such that the map $\FibF_1\colon \FibA_1\to \FibB_1$ is the induced map over $\bigsqcup_{A,B\in\cC}\cC(A,B)$
\[ \textstyle\bigsqcup_{A,B\in\cC}\phi_{A,B}\colon\bigsqcup_{A,B\in\cC}\cC(A,B)\times P^{-1} B\to \bigsqcup_{A,B\in\cC}\cC(A,B)\times Q^{-1} B. \]
Then, by commutativity of the following diagram in $\cV$, 
\begin{tz}
\node[](1) {$\bigsqcup_{A,B\in\cC} \cC(A,B)\times P^{-1}B$}; 
\node[below of=1](3) {$\bigsqcup_{A,B\in\cC}\cC(A,B)\times Q^{-1}B$};
\node[right of=1, xshift=3.6cm](2) {$\bigsqcup_{B\in\cC} P^{-1}B$};
\node[below of=2](4) {$\bigsqcup_{B\in\cC} Q^{-1}B$}; 

\draw[->] (1) to node[above,la]{$t=\bigsqcup_{\mathrm{pr}_1} \pi_1$} (2);
\draw[->] (2) to node[right,la]{$\FibF_0$} (4);
\draw[->] (1) to node[left,la]{$\FibF_1$} (3);
\draw[->] (3) to node[below,la]{$t=\bigsqcup_{\mathrm{pr}_1} \pi_1$} (4);
\end{tz}
and the description of the target maps given in \cref{subsec:discfiboverC}, we get that the following diagram in $\cV$ commutes, for all objects $A,B\in \cC$.
\begin{tz}
\node[](1) {$\cC(A,B)\times P^{-1}B$}; 
\node[below of=1](3) {$\cC(A,B)\times Q^{-1}B$};
\node[right of=1, xshift=2.2cm](2) {$P^{-1}B$};
\node[below of=2](4) {$Q^{-1}B$}; 

\draw[->] (1) to node[above,la]{$\pi_1$} (2);
\draw[->] (2) to node[right,la]{$\ev^H_B$} (4);
\draw[->] (1) to node[left,la]{$\phi_{A,B}$} (3);
\draw[->] (3) to node[below,la]{$\pi_1$} (4);
\end{tz}
In particular, from the commutativity of this diagram, together with the fact that the maps $\phi_{A,B}$ are over $\cC(A,B)$ and the universal property of the product, we can deduce that $\phi_{A,B}$ must be given by the map 
\[ \id_{\cC(A,B)}\times \ev^H_B\colon \cC(A,B)\times P^{-1} B\to \cC(A,B)\times Q^{-1} B. \]
Then, the map $\FibF_1\colon \FibA_1\to \FibB_1$ must be the induced map 
\[ \textstyle\bigsqcup_{A,B\in\cC}\id_{\cC(A,B)}\times \ev^H_B\colon\bigsqcup_{A,B\in\cC}\cC(A,B)\times P^{-1} B\to \bigsqcup_{A,B\in\cC}\cC(A,B)\times Q^{-1} B. \]

\subsection{Properties}\label{subsec:props}

We now collect some useful properties of internal discrete fibrations.

\begin{prop} \label{prop:pullbackofdiscfib}
 The pullback of an internal discrete fibration in $\CatV$ is also an internal discrete fibration. 
\end{prop}

\begin{proof}
    Consider a pullback square in $\CatV$
    \begin{tz}
\node[](1) {$\bA'$}; 
\node[below of=1](3) {$\bB'$};
\node[right of=1,xshift=.5cm](2) {$\bA$}; 
\node[below of=2](4) {$\bB$};
\pullback{1};

\draw[->] (1) to (2);
\draw[->] (2) to node[right,la]{$P$} (4);
\draw[->] (1) to node[left,la]{$P'$} (3);
\draw[->] (3) to (4);
\end{tz}
where $P$ is an internal discrete fibration. We show that $P'$ is an internal discrete fibration as well. For this, consider the following diagram in $\cV$
\begin{tz}
    \node[](1') {$\bA'_1$};
    \node[below of=1'](3') {$\bB'_1$};
\node[right of=1',xshift=.7cm](1) {$\bA'_0$}; 
\node[below of=1](3) {$\bB'_0$};
\node[right of=1,xshift=.7cm](2) {$\bA_0$}; 
\node[below of=2](4) {$\bB_0$};
\pullback{1};

\draw[->] (1) to (2);
\draw[->] (2) to node[right,la]{$P_0$} (4);
\draw[->] (1) to node[left,la]{$P'_0$} (3);
\draw[->] (3) to node[below,la]{} (4);
\draw[->] (1') to node[above,la]{$t$} (1);
\draw[->] (1') to node[left,la]{$P'_1$} (3');
\draw[->] (3') to node[below,la]{$t$} (3);

\node[right of=2,xshift=.7cm](1') {$\bA'_1$};
\node at ($(4)!0.5!(1')$) {$=$};
    \node[below of=1'](3') {$\bB'_1$};
\node[right of=1',xshift=.7cm](1) {$\bA_1$}; 
\node[below of=1](3) {$\bB_1$};
\node[right of=1,xshift=.7cm](2) {$\bA_0$}; 
\node[below of=2](4) {$\bB_0$};
\pullback{1};
\pullback{1'};

\draw[->] (1) to node[above,la]{$t$} (2);
\draw[->] (2) to node[right,la]{$P_0$} (4);
\draw[->] (1) to node[left,la]{$P_1$} (3);
\draw[->] (3) to node[below,la]{$t$} (4);
\draw[->] (1') to (1);
\draw[->] (1') to node[left,la]{$P'_1$} (3');
\draw[->] (3') to (3);
\end{tz}
where the two middle squares are pullbacks as $P'$ is a pullback of $P$, and the right-hand square is a pullback by definition of $P$ being an internal discrete fibration. By cancellation of pullbacks, the left-hand square is also a pullback which shows that $P'$ is an internal discrete fibration.
\end{proof}

The following result is particularly useful, as it significantly reduces the difficulty of checking whether an internal functor between internal discrete fibrations is an isomorphism.

\begin{prop}\label{prop:isoonfibers}
    Consider a commutative diagram in $\CatV$
\begin{tz}
\node[](1) {$\bA$}; 
\node[below of=1,xshift=1.8cm,yshift=.2cm](3) {$\bC$};
\node[above of=3,xshift=1.8cm,yshift=-.2cm](2) {$\bB$}; 

\draw[->] (1) to node[above,la]{$H$} (2);
\draw[->] (2) to node[right,la,yshift=-4pt]{$P$} (3);
\draw[->] (1) to node[left,la,yshift=-4pt]{$Q$} (3);
\end{tz}
    with $P$ and $Q$ internal discrete fibrations over $\bC$. Then the following are equivalent:
    \begin{rome}
        \item the internal functor $H\colon\bA\to\bB$ is an isomorphism in $\CatV$,
        \item the map $H_0\colon\bA_0\to\bB_0$ is an isomorphism in $\cV$.
    \end{rome}
    Moreover, if $\cV$ is extensive and $\bC=\Int\cC$ for some $\cV$-category $\cC$, the above conditions are further equivalent to:
    \begin{enumerate}[label=(\roman*)]
    \setcounter{enumi}{2}
        \item the map $\ev^H_A\colon P^{-1}A\to Q^{-1}A$ from \eqref{defn:mapoffibers} is an isomorphism in $\cV$, for every object $A\in\cC$.
    \end{enumerate}
\end{prop}

\begin{proof}    
    Since $H$ is an isomorphism in $\CatV$ precisely when $H_0$ and $H_1$ are isomorphisms in~$\cV$, the fact that (i) implies (ii) is immediate, and to prove that (ii) implies (i) it is enough to show that $H_1$ is an isomorphism in $\cV$. This follows from 2-out-of-3 in the following commutative diagram in $\cV$, where the vertical isomorphisms come from the definition of $P$ and $Q$ being internal discrete fibrations.
    \begin{tz}
\node[](1) {$\bA_1$}; 
\node[below of=1](3) {$\bA_0\times_{\bC_0} \bC_1$};
\node[right of=1,xshift=2.3cm](2) {$\bB_1$}; 
\node[below of=2](4) {$\bB_0\times_{\bC_0} \bC_1$};

\draw[->] (1) to node[above,la]{$H_1$} (2);
\draw[->] (2) to node[right,la]{$\cong$} (4);
\draw[->] (1) to node[left,la]{$\cong$}
 (3);
\draw[->] (3) to node[above,la]{$\cong$} node[below,la]{$H_0\times_{\bC_0} \bC_1$} (4);
\end{tz}
 
    Now assume that $\bC=\Int\cC$ for some $\cV$-category $\cC$.
    To show that (ii) and (iii) are equivalent, consider the following commutative diagram in $\cV$, where the vertical isomorphisms are the ones from \cref{subsec:discfiboverC}.
\begin{tz}
\node[](1) {$\bA_0$}; 
\node[below of=1](3) {$\bigsqcup_{A\in \cC} P^{-1}A$};
\node[right of=1,xshift=2.3cm](2) {$\bB_0$}; 
\node[below of=2](4) {$\bigsqcup_{A\in \cC} Q^{-1}A$};

\draw[->] (1) to node[above,la]{$H_0$} (2);
\draw[->] (2) to node[right,la]{$\cong$} (4);
\draw[->] (1) to node[left,la]{$\cong$}
 (3);
\draw[->] (3) to node[above,la]{$\cong$} node[below,la]{$\bigsqcup_{A\in \cC} \ev^H_A$} (4);
    \end{tz}
    Then by $2$-out-of-$3$, the top map is an isomorphism if and only if the bottom one is. As the functor $\bigsqcup_{A\in \cC}\colon \prod_{A\in \cC}\slice{\cV}{\un}\to \slice{\cV}{\bigsqcup_{A\in \cC}\un}$ creates isomorphisms, the bottom map is an isomorphism if and only if each map $\ev^H_A\colon P^{-1}A\to Q^{-1}A$ is an isomorphism, for every object $A\in \cC$. 
\end{proof}

\subsection{Examples}\label{subsec:examples}

We conclude this section with some examples of internal discrete fibrations.

\subsubsection*{Slice internal categories} We first define slices in the internal setting, and show that the projection from the slice is an internal discrete fibration.

\begin{notation}
    We write $\mathbbm{1}$ for the terminal $\cV$-category with one object $0$ and hom-object $\mathbbm{1}(0,0)=\un$. Note that $\Int\mathbbm{1}=\cst\un=\un$. For convenience, we also write $\mathbbm{2}$ for the internal category $\Int\mathbbm{2}$ where $\mathbbm{2}$ is the $\cV$-category with object set $\{0,1\}$ and hom-objects in $\cV$ 
    \[ \mathbbm{2}(0,0)=\mathbbm{2}(0,1)=\mathbbm{2}(1,1)=\un \quad\text{and}\quad \mathbbm{2}(1,0)=\emptyset,  \]
    where $\emptyset$ denotes the initial object of $\cV$. We write $\langle 0\rangle\colon \un\to \mathbbm{2}$, $\langle 1\rangle\colon \un\to \mathbbm{2}$, and $\langle 0,1\rangle\colon \un\sqcup\un\to \mathbbm{2}$ for the canonical internal functors to $\cV$ induced by coproduct inclusions.
\end{notation}

\begin{rem} \label{rem:hom2A}
    Given an internal category $\bA$, note that the internal category $\llbracket\mathbbm{2},\bA\rrbracket$ has level $0$ given by $\llbracket\mathbbm{2},\bA\rrbracket_0=\bA_1$.
\end{rem}

\begin{defn} \label{defn:slice}
Let $\bA$ be an internal category to $\cV$, and $T\colon \un\to \bA$ be an element. The \textbf{slice internal category} $\slice{\bA}{T}$ is defined as the following pullback in $\CatV$.
\begin{tz}
\node[](1) {$\slice{\bA}{T}$}; 
\node[below of=1](3) {$\bA$};
\node[right of=1,xshift=1cm](2) {$\llbracket\mathbbm{2},\bA\rrbracket$}; 
\node[below of=2](4) {$\bA\times\bA$};
\pullback{1};

\draw[->] (1) to node[above,la]{} (2);
\draw[->] (2) to node[right,la]{$\langle 0,1\rangle^*$} (4);
\draw[->] (1) to node[left,la]{} (3);
\draw[->] (3) to node[below,la]{$\id_\bA\times T$} (4);
\end{tz}
\end{defn}

\begin{rem} \label{rem:pullbackslice}
    Since pullbacks compose, the slice internal category $\slice{\bA}{T}$ can equivalently be seen as the following pullback in $\CatV$.
    \begin{tz}
\node[](1) {$\slice{\bA}{T}$}; 
\node[below of=1](3) {$\un$};
\node[right of=1,xshift=1cm](2) {$\llbracket\mathbbm{2},\bA\rrbracket$}; 
\node[below of=2](4) {$\bA$};
\pullback{1};

\draw[->] (1) to node[above,la]{} (2);
\draw[->] (2) to node[right,la]{$\langle 1\rangle^*$} (4);
\draw[->] (1) to node[left,la]{} (3);
\draw[->] (3) to node[below,la]{$T$} (4);
\end{tz}
\end{rem}

\begin{prop} \label{prop:sliceisdiscfib}
    Let $\bA$ be an internal category to $\cV$, and $T\colon \un\to \bA$ be an element. Then the projection $\slice{\bA}{T}\to\bA$ is an internal discrete fibration.
\end{prop}

\begin{proof}
    Consider the following pushout in $\Cat$.
    \begin{tz}
        \node[](1) {$\mathbbm{2}$}; 
\node[below of=1](3) {$\mathbbm{1}$};
\node[right of=1,xshift=.9cm](2) {$\mathbbm{2}\times \mathbbm{2}$}; 
\node[below of=2](4) {$\mathbbm{2}\sqcup_{\mathbbm{1}} \mathbbm{2}$};
\pushout{4};

\draw[->] (1) to node[above,la]{$\id_\mathbbm{2}\times \langle 1\rangle$} (2);
\draw[->] (2) to (4);
\draw[->] (1) to node[left,la]{$!$} (3);
\draw[->] (3) to node[below,la]{$\langle 2\rangle$} (4);
    \end{tz}
    By applying the functor $\llbracket -,\bA\rrbracket$, we get a pullback in $\CatV$ as below left, whose level 0  can be seen to be the pullback in $\cV$ depicted below right using \cref{rem:hom2A}.
    \begin{tz}
        \node[](1) {$\llbracket\mathbbm{2}\sqcup_{\un} \mathbbm{2},\bA\rrbracket$}; 
\node[below of=1](3) {$\bA$};
\node[right of=1,xshift=2cm](2) {$\llbracket\mathbbm{2}\times \mathbbm{2},\bA\rrbracket$}; 
\node[below of=2](4) {$\llbracket\mathbbm{2},\bA\rrbracket$};
\pullback{1};

\draw[->] (1) to (2);
\draw[->] (2) to node[right,la]{$(\id_\mathbbm{2}\times \langle 1\rangle)^*$} (4);
\draw[->] (1) to node[left,la]{$\langle 2\rangle^*$} (3);
\draw[->] (3) to node[below,la]{$!^*$} (4);

    \node[right of=2,xshift=2cm](1) {$\bA_1\times_{\bA_0}\bA_1$}; 
\node[below of=1](3) {$\bA_0$};
\node[right of=1,xshift=1.6cm](2) {$\llbracket\mathbbm{2},\bA\rrbracket_1$}; 
\node[below of=2](4) {$\bA_1$};
\pullback{1};

\draw[->] (1) to (2);
\draw[->] (2) to node[right,la]{$\langle 1\rangle^*_1$} (4);
\draw[->] (1) to node[left,la]{$t\circ \pi_1$} (3);
\draw[->] (3) to node[below,la]{$i$} (4);
    \end{tz}
    
Now consider the following commutative diagram in $\cV$.
    \begin{tz}
    \node[](1') {$(\slice{\bA}{T})_1$};
    \node[below of=1'](3') {$\un$};
\node[right of=1',xshift=1.7cm](1) {$\bA_1\times_{\bA_0} \bA_1$}; 
\node[below of=1](3) {$\bA_0$};
\node[right of=1,xshift=1.7cm](2) {$\llbracket\mathbbm{2},\bA\rrbracket_1$}; 
\node[below of=2](4) {$\bA_1$};
\pullback{1};

\draw[->] (1) to (2);
\draw[->] (2) to node[right,la]{$\langle 1\rangle^*_1$} (4);
\draw[->] (1) to node[left,la]{$t\circ \pi_1$} (3);
\draw[->] (3) to node[below,la]{$i$} (4);
\draw[->,dashed] (1') to (1);
\draw[->] (1') to (3');
\draw[->] (3') to node[below,la]{$T$} (3);
\draw[->,bend left=20] (1') to (2);
\end{tz}
By the universal property of the pullback, there is a unique dashed map as depicted. Moreover,  as the outer square is a pullback by \cref{rem:pullbackslice}, by cancellation of pullbacks the left-hand square is also a pullback. Finally, consider the following diagram in $\cV$. 
\begin{tz}
\node[](1) {$(\slice{\bA}{T})_1$}; 
\node[below of=1](3) {$\bA_1\times_{\bA_0}\bA_1$};
\node[below of=3](5) {$\bA_1$};
\node[right of=1,xshift=1.3cm](2) {$(\slice{\bA}{T})_0$}; 
\node[right of=2,xshift=1cm](2') {$\un$}; 
\node[below of=2'](4') {$\bA_0$};
\node[below of=2](4) {$\bA_1$};
\node[below of=4](6) {$\bA_0$};
\pullback{3};
\pullback{2};

\draw[->] (1) to node[above,la]{$t$} (2);
\draw[->] (2) to (2');
\draw[->] (2) to (4);
\draw[->] (2') to node[right,la]{$T$} (4');
\draw[->] (4) to node[right,la]{$s$} (6);
\draw[->] (1) to node[left,la]{} (3);
\draw[->] (3) to node[left,la]{$\pi_0$} (5);
\draw[->] (3) to node[below,la]{$\pi_1$} (4);
\draw[->] (4) to node[below,la]{$t$} (4');
\draw[->] (5) to node[below,la]{$t$} (6);
\end{tz}
The top right-hand square is a pullback by \cref{rem:pullbackslice}, the horizontal rectangle is a pullback as explained above, and the bottom left square is a pullback by definition of $\bA_1\times_{\bA_0}\bA_1$. Hence, by cancellation of pullbacks, we get that the top left square is a pullback, and hence so is the vertical rectangle. This shows that $\slice{\bA}{T}\to \bA$ is an internal discrete fibration. 
\end{proof}

\subsubsection*{Comma internal categories} Next, we define comma internal categories, and show that they provide another source of internal discrete fibrations.

\begin{defn}\label{defn:internalcomma}
    Let $H\colon \bA\to \bB$ be an internal functor to $\cV$, and $B\colon \un\to \bB$ be an element. The \textbf{comma internal category} $H\downarrow B$ is defined as the following pullback in $\CatV$.
\begin{tz}
\node[](1) {$H\downarrow B$}; 
\node[below of=1](3) {$\bA$};
\node[right of=1,xshift=1.2cm](2) {$\llbracket\mathbbm{2},\bB\rrbracket$}; 
\node[below of=2](4) {$\bB\times\bB$};
\pullback{1};

\draw[->] (1) to node[above,la]{} (2);
\draw[->] (2) to node[right,la]{$\langle 0,1\rangle^*$} (4);
\draw[->] (1) to node[left,la]{} (3);
\draw[->] (3) to node[below,la]{$H\times B$} (4);
\end{tz}
\end{defn}

\begin{prop} \label{prop:commaisdiscfib}
    Let $H\colon \bA\to \bB$ be an internal functor to $\cV$, and $B\colon \un\to \bB$ be an element. Then the projection $H\downarrow B\to \bA$ is an internal discrete fibration. 
\end{prop}

\begin{proof}
    Consider the following diagram in $\CatV$.
    \begin{tz}
    \node[](1') {$H\downarrow B$};
    \node[below of=1'](3') {$\bA$};
\node[right of=1',xshift=1.1cm](1) {$\slice{\bB}{B}$}; 
\node[below of=1](3) {$\bB$};
\node[right of=1,xshift=1.1cm](2) {$\llbracket\mathbbm{2},\bB\rrbracket$}; 
\node[below of=2](4) {$\bB\times\bB$};
\pullback{1};

\draw[->] (1) to (2);
\draw[->,bend left=20] (1') to (2);
\draw[->] (2) to node[right,la]{$\langle 0,1\rangle^*$} (4);
\draw[->] (1) to node[left,la]{} (3);
\draw[->] (3) to node[below,la]{$\id_\bB\times B$} (4);
\draw[->,dashed] (1') to (1);
\draw[->] (1') to (3');
\draw[->] (3') to node[below,la]{$H$} (3);
\end{tz}
By the universal property of the pullback, there is a unique dashed internal functor as depicted. Moreover, note that the right-hand square and the outer square are pullbacks by definition of the slice and comma internal categories, respectively. Hence, by cancellation of pullbacks, the left-hand square is also a pullback. Since we know by \cref{prop:sliceisdiscfib} that $\slice{\bB}{B}\to \bB$ is an internal discrete fibration, \cref{prop:pullbackofdiscfib} then ensures that its pullback $H\downarrow B\to \bA$ is also an internal discrete fibration. 
\end{proof}

\subsubsection*{Internal Grothendieck construction} 
The projection from the internal category of elements from \cref{defn:groth} is another example of an internal discrete fibration, whenever $\cV$ is cartesian closed and extensive.

\begin{prop}\label{grothdiscfib}
    Let $F\colon\cC^{\op}\to\cV$ be a $\cV$-functor. Then the projection from its internal category of elements $\pi_F\colon\int_\cC F\to\Int \cC$ is an internal discrete fibration. 
\end{prop}

\begin{proof}
    To see this, we need to check that the following diagram in $\cV$ is a pullback.
\begin{tz}
\node[](1) {$\bigsqcup_{A,B\in \cC} \cC(A,B)\times FB$}; 
\node[below of=1](3) {$\bigsqcup_{A,B\in\cC}\cC(A,B)$};
\node[right of=1, xshift=2.8cm](2) {$\bigsqcup_{B\in\cC} FB$}; 
\node[below of=2](4) {$\bigsqcup_{B\in\cC} \un$}; 

\draw[->] (1) to node[above,la]{$t$} (2);
\draw[->] (2) to node[right,la]{$(\pi_F)_1$} (4);
\draw[->] (1) to node[left,la]{$(\pi_F)_0$} (3);
\draw[->] (3) to node[below,la]{$t$} (4);
\end{tz}
This is immediate from the fact that coproducts in $\cV$ preserve pullbacks, as $\cV$ is extensive. 
\end{proof}

\section{Equivalence between enriched functors and internal discrete fibrations}\label{section:equiv}

In this section, let $\cV$ be an extensive and cartesian closed category with pullbacks. In light of \cref{grothdiscfib}, we know that the internal Grothendieck construction takes values in internal discrete fibrations. The aim of this section is to prove the following.

\begin{theorem}\label{thm:equiv}
    The internal Grothendieck construction gives an equivalence of categories
    \[ \textstyle\int_\cC\colon \VCat(\cC^{\op},\cV)\to \slice{\Dfib}{\Int\cC} \]
    between the category of $\cV$-functors $\cC^{\op}\to\cV$, and the category of internal discrete fibrations over~$\Int\cC$. 
\end{theorem}

This proof is done in two stages: first, in \cref{subsec:inversegroth} we construct the inverse functor, and prove that it is well-defined; then, in \cref{subsec:proofequiv} we prove that this is indeed an inverse by constructing the required natural isomorphisms. Moreover, once the equivalence is settled, in \cref{subsec:enrichedequiv} we show that it can be promoted to an equivalence of $\cV$-categories.

\subsection{The inverse of the internal Grothendieck construction}\label{subsec:inversegroth}

From the data of an internal discrete fibration $P\colon \bA\to\Int\cC$, we construct a $\cV$-functor $\cC^{\op}\to \cV$ by using the description in \cref{lem:VfunctortoV}.

\begin{constr}
Let $P\colon \bA\to\Int\cC$ be an internal discrete fibration in $\CatV$. We define a $\cV$-functor $P^{-1}\colon\cC^{\op}\to\cV$ such that
\begin{rome}
    \item it sends an object $A\in \cC$ to the object $P^{-1}A$ of $\cV$  defined in \eqref{defn:fiber},
    \item for objects $A,B\in \cC$, the map $\ev^{P^{-1}}_{A,B}$ is given by the map in $\cV$ defined in \eqref{evP-1} \[ \ev^{P^{-1}}_{A,B}\colon \cC(A,B)\times P^{-1}B\to P^{-1}A. \]
\end{rome}
\end{constr}

\begin{prop}\label{prop:phiP}
The correspondence $P^{-1}\colon\cC^{\op}\to\cV$ is a $\cV$-functor.
\end{prop}

\begin{proof}
 To show that $P^{-1}$ is indeed a $\cV$-functor, we must verify the conditions in \cref{lem:VfunctortoV}. For this, first note that since $\FibA$ is an internal category, the following diagrams in $\CatV$ commute.
    \begin{tz}
       \node[](1) {$\FibA_0$}; 
\node[right of=1,xshift=1.6cm](2) {$\FibA_1$}; 
\node[below of=2](4) {$\FibA_0$}; 

\draw[->] (1) to node[above,la]{$i$} (2);
\draw[->] (2) to node[right,la]{$s$} (4);
\draw[->] (1) to node[below,la,xshift=-6pt]{$\id_{\FibA_0}$} (4);

\node[right of=2,xshift=2.7cm](1) {$\FibA_1\times_{\FibA_0}\FibA_1$}; 
\node[below of=1](3) {$\FibA_1$};
\node[right of=1, xshift=3cm](2) {$\FibA_1$}; 
\node[below of=2](4) {$\FibA_0$}; 

\draw[->] (1) to node[above,la]{$c$} (2);
\draw[->] (2) to node[right,la]{$s$} (4);
\draw[->] (1) to node[left,la]{$\pi_0$} (3);
\draw[->] (3) to node[below,la]{$s$} (4);
\end{tz}
   Then, as explained in \cref{subsec:discfiboverC}, the components of the above diagrams at objects $A,B,C\in \cC$ consist of the following commutative diagrams in $\cV$,
    \begin{tz}
 
\node[](1) {$P^{-1}B$}; 
\node[right of=1,xshift=2.1cm](2) {$\cC(B,B)\times P^{-1} B$}; 
\node[below of=2](4) {$P^{-1}B$}; 

\draw[->] (1) to node[above,la]{$i_B\times \id_{P^{-1}B}$} (2);
\draw[->] (2) to node[right,la]{$\ev^{P^{-1}}_{B,B}$} (4);
\draw[->] (1) to node[below,la,xshift=-6pt]{$\id_{P^{-1} B}$} (4);

\node[right of=2,xshift=2.9cm](1) {$\cC(A,B)\times \cC(B,C)\times P^{-1}C$}; 
\node[below of=1](3) {$\cC(A,B)\times P^{-1}B$};
\node[right of=1, xshift=4.3cm](2) {$\cC(A,C)\times P^{-1}C$}; 
\node[below of=2](4) {$P^{-1}A$}; 

\draw[->] (1) to node[above,la]{$c_{A,B,C}\times \id_{P^{-1}C}$} (2);
\draw[->] (2) to node[right,la]{$\ev^{P^{-1}}_{A,C}$} (4);
\draw[->] (1) to node[left,la]{$\id_{\cC(A,B)}\times \ev^{P^{-1}}_{B,C}$} (3);
\draw[->] (3) to node[below,la]{$\ev^{P^{-1}}_{A,B}$} (4);
\end{tz}
   which express the desired compatibility of $P^{-1}$ with identities and compositions.
\end{proof} 

We now describe a way to produce a $\cV$-natural transformation between $\cV$-functors $\cC^{\op}\to \cV$ from the data of an internal functor $H\colon \bA\to \bB$ between internal discrete fibrations over $\Int\cC$, using  \cref{lem:VnattoV}.

\begin{constr}
Consider a commutative diagram in $\CatV$
\begin{tz}
\node[](1) {$\bA$}; 
\node[below of=1,xshift=1.8cm,yshift=.2cm](3) {$\Int\cC$};
\node[above of=3,xshift=1.8cm,yshift=-.2cm](2) {$\bB$}; 

\draw[->] (1) to node[above,la]{$H$} (2);
\draw[->] (2) to node[right,la,yshift=-4pt]{$P$} (3);
\draw[->] (1) to node[left,la,yshift=-4pt]{$Q$} (3);
\end{tz}
    with $P$ and $Q$ internal discrete fibrations over $\Int\cC$. We define a $\cV$-natural transformation $H\colon P^{-1}\Rightarrow Q^{-1}$ whose component $\ev^{H}_A$ at an object $A\in \cC$ is given by the map in $\cV$ defined in~\eqref{defn:mapoffibers} 
\[ \ev^H_A\colon P^{-1} A\to Q^{-1} A. \]
\end{constr}

\begin{prop}
    The correspondence $H\colon P^{-1}\Rightarrow Q^{-1}$ is a $\cV$-natural transformation.
\end{prop}

\begin{proof}
    To show that $H$ is indeed a $\cV$-natural transformation, we must verify the $\cV$-naturality condition in \cref{lem:VnattoV}. For this, first recall that since $\FibF\colon \FibA\to \FibB$ is an internal functor, the following diagram in $\CatV$ commutes. 
    \begin{tz}
\node[](1) {$\FibA_1$}; 
\node[below of=1](3) {$\FibB_1$};
\node[right of=1,xshift=1.6cm](2) {$\FibA_0$}; 
\node[below of=2](4) {$\FibB_0$}; 

\draw[->] (1) to node[above,la]{$s$} (2);
\draw[->] (2) to node[right,la]{$\FibF_0$} (4);
\draw[->] (1) to node[left,la]{$\FibF_1$} (3);
\draw[->] (3) to node[below,la]{$s$} (4);
\end{tz}
Then, as explained in \cref{subsec:discfiboverC,subsec:mapofdiscfib}, the component of the above diagram at objects $A,B\in \cC$ consists of the following commutative diagram in $\cV$,
\begin{tz}
\node[](1) {$\cC(A,B)\times P^{-1}B$}; 
\node[below of=1](3) {$\cC(A,B)\times Q^{-1}B$};
\node[right of=1,xshift=2.2cm](2) {$P^{-1}A$}; 
\node[below of=2](4) {$Q^{-1}A$}; 

\draw[->] (1) to node[above,la]{$\ev^{P^{-1}}_{A,B}$} (2);
\draw[->] (2) to node[right,la]{$\ev^H_A$} (4);
\draw[->] (1) to node[left,la]{$\id_{\cC(A,B)}\times \ev^H_B$} (3);
\draw[->] (3) to node[below,la]{$\ev^{Q^{-1}}_{A,B}$} (4);
\end{tz}
which expresses the desired $\cV$-naturality condition for $H$.
\end{proof}

We can use the above constructions to define the inverse to the internal Grothendieck construction.

\begin{constr}\label{constr:phi}
    We define a functor $\Phi\colon \slice{\Dfib}{\Int\cC}\to\VCat(\cC^{\op},\cV)$, which sends
    \begin{rome}
        \item an internal discrete fibration $P\colon\bA\to\Int\cC$ to the $\cV$-functor $P^{-1}\colon\cC^{\op}\to\cV$,
        \item an internal functor $H\colon \bA\to \bB$ over $\Int\cC$ between internal discrete fibrations $P$ and $Q$ to the $\cV$-natural transformation $H\colon P^{-1}\Rightarrow Q^{-1}$.
    \end{rome}
\end{constr}

\begin{prop}
    The correspondence $\Phi\colon \slice{\Dfib}{\Int\cC}\to\VCat(\cC^{\op},\cV)$ is a functor.
\end{prop}

\begin{proof}
The fact that $\Phi$ preserves compositions and identities follows from the uniqueness in the universal property of the pullbacks \eqref{defn:mapoffibers} defining the components of the $\cV$-natural transformations.
\end{proof}

\subsection{Proof of the equivalence}\label{subsec:proofequiv} 

Our strategy to prove \cref{thm:equiv} is to show that the internal Grothendieck construction $\int_\cC$ of \cref{sec:groth} and the functor $\Phi$ of \cref{constr:phi} are inverse equivalences. This requires us to construct natural isomorphisms \[  \textstyle\eta\colon \Phi\int_\cC\cong \id_{\VCat(\cC^{\op},\cV)} \quad \text{and} \quad \varepsilon\colon \int_\cC \Phi\cong \id_{\slice{\Dfib}{\Int\cC}}. \] 

We first define the components of $\eta$.

\begin{constr}\label{constr:eta}
    Let $F\colon\cC^{\op}\to\cV$ be a $\cV$-functor. We define a $\cV$-natural transformation $\eta_F\colon \Phi\int_\cC F\Rightarrow F$ whose component $\ev_A^{\eta_F}$ at an object $A\in \cC$ is given by the unique isomorphism 
    \[ \textstyle  \ev_A^{\eta_F}\colon (\int_\cC F)^{-1} A\xrightarrow{\cong} FA \]
    induced by the fact that both objects are pullbacks of the same diagram in $\cV$ as depicted below. 
    \begin{tz}
\node[](1) {$(\int_\cC F)^{-1} A$}; 
\node[below of=1](3) {$\un$};
\node[right of=1,xshift=1.8cm](2) {$\bigsqcup_{A\in \cC} FA$}; 
\node[below of=2](4) {$\bigsqcup_{A\in \cC} \un$}; 
\pullback{1};

\draw[->] (1) to node[above,la]{} (2);
\draw[->] (2) to node[right,la]{$\ev^F_A$} (4);
\draw[->] (1) to (3);
\draw[->] (3) to node[below,la]{$\iota_A$} (4);

\node[right of=2,xshift=1.5cm](1) {$FA$}; 
\node[below of=1](3) {$\un$};
\node[right of=1,xshift=1.3cm](2) {$\bigsqcup_{A\in \cC} FA$}; 
\node[below of=2](4) {$\bigsqcup_{A\in \cC} \un$}; 
\pullback{1};

\draw[->] (1) to node[above,la]{$\iota_A$} (2);
\draw[->] (2) to node[right,la]{$\ev^F_A$} (4);
\draw[->] (1) to (3);
\draw[->] (3) to node[below,la]{$\iota_A$} (4);
\end{tz}
It is straightforward to check that $\eta_F$ is $\cV$-natural, and that the $\eta_F$'s assemble into a natural isomorphism $\eta\colon \Phi \int_\cC \cong \id_{\VCat(\cC^{\op},\cV)}$.
\end{constr}

Next, we define the components of $\varepsilon\colon \int_\cC \Phi\cong \id_{\slice{\Dfib}{\Int\cC}}$. Given an internal discrete fibration $P\colon\bA\to\Int\cC$, one can unpack the definitions to see that $\int_\cC P^{-1}$ is equal to the internal category~$\FibA$ constructed in \cref{subsec:discfiboverC}. Then, considering a commutative diagram in~$\CatV$
\begin{tz}
\node[](1) {$\bA$}; 
\node[below of=1,xshift=1.8cm,yshift=.2cm](3) {$\Int\cC$};
\node[above of=3,xshift=1.8cm,yshift=-.2cm](2) {$\bB$}; 

\draw[->] (1) to node[above,la]{$H$} (2);
\draw[->] (2) to node[right,la,yshift=-4pt]{$P$} (3);
\draw[->] (1) to node[left,la,yshift=-4pt]{$Q$} (3);
\end{tz}
with $P$ and $Q$ internal discrete fibrations over $\Int\cC$, one can check that the internal functor $\int_\cC H\colon \int_\cC P^{-1}\to \int_\cC Q^{-1}$ agrees with the description of $\FibF\colon \FibA\to \FibB$ given in \cref{subsec:mapofdiscfib}. Moreover, recall the $\FibF$ was defined in such a way that the following diagram in~$\CatV$ commutes. 
\begin{tz}
\node[](1) {$\FibA$}; 
\node[below of=1](3) {$\FibB$};
\node[right of=1,xshift=1.1cm](2) {$\bA$}; 
\node[below of=2](4) {$\bB$};

\draw[->] (1) to node[above,la]{$\varepsilon_P\coloneqq \varphi$} (2);
\draw[->] (2) to node[right,la]{$H$} (4);
\draw[->] (1) to node[left,la]{$\FibF$} (3);
\draw[->] (3) to node[below,la]{$\varepsilon_Q\coloneqq \psi$} (4);
    \end{tz}

\begin{constr}\label{constr:epsilon}
    Let $P\colon \bA\to \Int\cC$ be an internal discrete fibration in $\CatV$. We define an internal functor $\varepsilon_P\colon \int_\cC P^{-1}\to \bA$ to be the isomorphism $\varphi\colon\FibA\cong\bA$ constructed in \cref{subsec:discfiboverC}. Moreover, the $\varepsilon_P$'s assemble into a natural isomorphism $\varepsilon\colon \int_\cC \Phi\cong \id_{\slice{\Dfib}{\Int\cC}}$.
\end{constr}

\begin{proof}[Proof of \cref{thm:equiv}]
    The data $(\Phi,\eta,\varepsilon)$ from \cref{constr:phi,constr:eta,constr:epsilon} makes \[\textstyle\int_\cC\colon \VCat(\cC^{\op},\cV)\to \slice{\Dfib}{\Int\cC}\] into an equivalence of categories. 
\end{proof}

\subsection{Enrichment of the equivalence}\label{subsec:enrichedequiv}

The categories involved in the equivalence of \cref{thm:equiv} can be promoted to $\cV$-categories, and we now explain how the equivalence from \cref{thm:equiv} can in turn be promoted to an equivalence of $\cV$-categories.

\begin{rem}
    The category $\VCat(\cC^{\op},\cV)$ is $\cV$-enriched with hom-objects given by $\cV^{\cC^{\op}}(F,G)$, for all $\cV$-functors $F,G\colon \cC^{\op}\to \cV$, as defined in \cref{rem:internalhomVcat}. We write $\cV^{\cC^{\op}}$ for this $\cV$-category. 
\end{rem}

\begin{rem}
    The category $\CatV$ is $\cV$-enriched with hom-objects given by $\llbracket \bA,\bB\rrbracket_0$, for internal categories $\bA$ and $\bB$ to $\cV$. Hence, the category $\slice{\CatV}{\Int\cC}$ is also $\cV$-enriched with hom-objects given by the following pullbacks in $\cV$,
\begin{tz}
\node[](1) {$\llbracket P,Q\rrbracket_0$}; 
\node[below of=1](3) {$\un$};
\node[right of=1,xshift=1.5cm](2) {$\llbracket \bA,\bB\rrbracket_0$}; 
\node[below of=2](4) {$\llbracket \bA,\Int\cC\rrbracket_0$};
\pullback{1};

\draw[->] (1) to (2);
\draw[->] (2) to node[right,la]{$(Q_*)_0$} (4);
\draw[->] (1) to (3);
\draw[->] (3) to node[below,la]{$P$} (4);
    \end{tz}
    for all internal functors $P\colon \bA\to \Int\cC$ and $Q\colon \bB\to \Int\cC$ to $\cV$. In particular, as a full subcategory of $\slice{\CatV}{\Int\cC}$, the category $\slice{\Dfib}{\Int\cC}$ inherits this $\cV$-enrichment. We write $\slice{\underline{\Dfib}}{\Int\cC}$ for this $\cV$-category. 
\end{rem}

\begin{defn}
    A $\cV$-functor $F\colon \cC\to \cD$ is an \textbf{equivalence of $\cV$-categories} if there is a $\cV$-functor $G\colon \cD\to \cC$ and $\cV$-natural isomorphisms $\eta\colon \id_\cC\cong GF$ and $\varepsilon\colon FG\cong \id_\cD$.
\end{defn}

\begin{theorem}\label{thm:enrichedequiv}
    The equivalence of categories $\int_\cC\colon \VCat(\cC^{\op},\cV)\to \slice{\Dfib}{\Int\cC}$ can be promoted to an equivalence of $\cV$-categories 
    \[ \textstyle\int_\cC \colon \cV^{\cC^{\op}}\to \slice{\underline{\Dfib}}{\Int\cC} \]
    between the $\cV$-category of $\cV$-functors $\cC^{\op}\to \cV$ and the $\cV$-category of internal discrete fibrations over $\Int\cC$. 
\end{theorem}

\begin{proof}
    By \cite[Proposition 2.1.12]{RiehlVerity}, the equivalence data $(\int_\cC, \Phi, \eta, \varepsilon)$ can be promoted to an adjoint equivalence. Then, since $\int_\cC$ preserves products (as a right adjoint) and $\cV$-tensors in~$\cV^{\cC^{\op}}$ are given by products with constant $\cV$-functors at objects of $\cV$, it also preserves $\cV$-tensors. Hence, by \cite[Proposition 3.7.10]{Riehlcathtpyth} we get a $\cV$-adjoint equivalence, as desired. 
\end{proof}

\begin{ex} \label{ex:equiv}
    In the usual examples, using the interpretation and terminology from \cref{ex:catofelements,ex:discfib}, \cref{thm:enrichedequiv} translates as follows.
    \begin{enumerate}
        \item When $\cV=\Set$, it retrieves the usual equivalence between the categories of functors $\cC^{\op}\to \Set$ and  of discrete fibrations over $\cC$ from e.g.~\cite[Theorem 2.1.2]{LoregianRiehl}.
        \item When $\cV=\Cat$, even though the double Grothendiceck construction has been considered in several contexts \cite{GraParPersistentII,cM2}, and double discrete fibrations have been compared in \cite{Lambert} to double functors valued in the double category of spans, we could not find the comparison as stated above in the literature. In this case, the theorem gives a $2$-equivalence between the $2$-categories of $2$-functors $\cC^{\op}\to \Cat$ and of double discrete fibrations over~$\Int\cC$.
        \item[($n$)] When $\cV=(n-1)\Cat$ for $n\geq 3$, the theorem gives an $n$-equivalence between the $n$-categories of $n$-functors $\cC^{\op}\to (n-1)\Cat$ and of double $(n-1)$-discrete fibrations over~$\Int\cC$. 
    \end{enumerate}
\end{ex}

\section{The representation theorem for enriched functors}

In this section, let $\cV$ be an extensive and cartesian closed category with pullbacks. We now turn our attention to representable $\cV$-functors, and show how these can be characterized by the existence of an internal terminal object in their internal category of elements. To this end, in \cref{subsec:internalterminal}, we first introduce internal terminal objects and prove that slice internal categories have an internal terminal object given by the identity. Then, in \cref{subsec:repfib}, given a $\cV$-category~$\cC$ and an object $C\in \cC$, we identify the internal category of elements of the representable $\cV$-functor $\cC(-,C)$ with the slice internal category~$\slice{\Int\cC}{C}$. We use this in \cref{subsec:proofofrepthm} to prove the representation theorem stated as \cref{thm:representation}. Finally, in \cref{subsec:reformulate}, we give a fully enriched statement of the representation theorem in terms of $\cV$-terminal objects in $\cV$-categories of generalized elements, which can be found as \cref{enrichedstatement}.

\subsection{Internal terminal objects} \label{subsec:internalterminal}

In order to state the theorem, we must first understand what it means to have a terminal object in the internal setting. We start by introducing this notion, and exploring some examples.

\begin{defn}\label{defn:internalterminal}
Let $\bA$ be an internal category to $\cV$, and $T\colon \un\to \bA$ be an element. We say that $T$ is an \textbf{internal terminal object} in $\bA$ if the canonical projection $\slice{\bA}{T}\to\bA$ is an isomorphism.   
\end{defn}

\begin{rem}\label{terminalinvariance}
    It is immediate from the definition that internal terminal objects are created by isomorphisms of internal categories.
\end{rem}

\begin{ex} \label{ex:doubleterminal}
    For our usual examples, we have the following. 
    \begin{enumerate}
        \item When $\cV=\Set$, an internal terminal object in a category $\bA$ coincides with the usual notion of a terminal object in $\bA$. 
        \item When $\cV=\Cat$, an internal terminal object in a double category $\bA$ coincides with the notion of a \emph{double terminal object} as introduced in \cite[\S 1.8]{GrandisPare}. This amounts to an object $T\in \bA$ such that for every object $A\in \bA$ there is a unique horizontal morphism from $A\to T$, and for every vertical morphism $u$ in $\bA$, there is a unique square from $u$ to the vertical identity at $T$. 
        \item[($n$)] When $\cV=(n-1)\Cat$ for $n\geq 3$, we call an internal terminal object in a double $(n-1)$-category~$\bA$ a \emph{double $(n-1)$-terminal object}. It consists of an object $T\in \bA$ such that, for every object $A$ in $\bA_0$, there is a unique object in $\bA_1$ from $A$ to $T$ and, for every $k$-morphism~$\varphi$ in $\bA_0$, there is a unique $k$-morphism in $\bA_1$ from $\varphi$ to the identity $k$-morphism at $T$, for all $1\leq k\leq n-1$.
    \end{enumerate}
\end{ex}

A concrete case where we can always find an internal terminal object is given by considering slice internal categories. Given an element $A\colon \un\to \bA$ in an internal category $\bA$ to $\cV$, we write $i_A\colon \un \to \llbracket\mathbbm{2},\bA\rrbracket$ for the element given by the composite of internal functors
\[ \un\xrightarrow{A}\bA\xrightarrow{!^*} \llbracket\mathbbm{2},\bA\rrbracket. \]
Note that it induces a unique element $i_A\colon \un\to \slice{\bA}{A}$, by the universal property of the pullback defining the slice internal category $\slice{\bA}{A}$ from \cref{rem:pullbackslice}. We show that $i_A$ is in fact an internal terminal object in $\slice{\bA}{A}$.

\begin{prop}\label{terminalobjslice}
    Let $\bA$ be an internal category to $\cV$, and $A\colon \un\to \bA$ be an element. Then the element $i_A\colon \un\to \slice{\bA}{A}$ is an internal terminal object in $\slice{\bA}{A}$. 
\end{prop}

\begin{proof}
Let $\rhd$ denote the pushout in $\Cat$ depicted below left, whose terminal object is denoted by $\top$. Then, it is not hard to check that the square below right is also a pushout in $\Cat$. 
\begin{tz}
\node[](1) {$\mathbbm{2}$}; 
\node[below of=1](3) {$\mathbbm{1}$};
\node[right of=1,xshift=.9cm](2) {$\mathbbm{2}\times \mathbbm{2}$}; 
\node[below of=2](4) {$\rhd$}; 
\pushout{4};

\draw[->] (1) to node[above,la]{$\langle 1\rangle\times \id_{\mathbbm{2}}$} (2);
\draw[->] (2) to (4);
\draw[->] (1) to node[left,la]{$!$} (3);
\draw[->] (3) to node[below,la]{$\langle \top\rangle$} (4);

\node[right of=2,xshift=1cm](1) {$\mathbbm{2}$}; 
\node[below of=1](3) {$\mathbbm{1}$};
\node[right of=1,xshift=.9cm](2') {$\mathbbm{2}\times \mathbbm{2}$}; 
\node[right of=2',xshift=.9cm](2) {$\rhd$}; 
\node[below of=2](4) {$\mathbbm{2}$}; 
\pushout{4};

\draw[->] (1) to node[above,la]{$\id_{\mathbbm{2}}\times \langle 1\rangle$} (2');
\draw[->] (2') to (2);
\draw[->] (2) to (4);
\draw[->] (1) to node[left,la]{$!$} (3);
\draw[->] (3) to node[below,la]{$\langle 1\rangle$} (4);
\end{tz}
By applying $\llbracket -,\bA\rrbracket$, we then get the following pullbacks in $\CatV$. 
\begin{tz}
\node[](1) {$\llbracket\rhd,\bA\rrbracket$}; 
\node[below of=1](3) {$\bA$};
\node[right of=1,xshift=1.6cm](2) {$\llbracket\mathbbm{2}\times \mathbbm{2},\bA\rrbracket$}; 
\node[below of=2](4) {$\llbracket\mathbbm{2},\bA\rrbracket$}; 
\pullback{1};

\draw[->] (1) to (2);
\draw[->] (2) to node[right,la]{$(\langle 1\rangle\times \id_{\mathbbm{2}})^*$} (4);
\draw[->] (1) to node[left,la]{$\langle 
\top\rangle^*$} (3);
\draw[->] (3) to node[below,la]{$!^*$} (4);

\node[right of=2,xshift=2cm](1) {$\llbracket\mathbbm{2},\bA\rrbracket$}; 
\node[below of=1](3) {$\bA$};
\node[right of=1,xshift=1.3cm](2) {$\llbracket\rhd,\bA\rrbracket$}; 
\node[below of=2](4) {$\llbracket\mathbbm{2},\bA\rrbracket$}; 
\pullback{1};

\draw[->] (1) to (2);
\draw[->] (2) to node[right,la]{$(\id_{\mathbbm{2}}\times \langle 1\rangle)^*$} (4);
\draw[->] (1) to node[left,la]{$\langle 1\rangle^*$} (3);
\draw[->] (3) to node[below,la]{$!^*$} (4);
\end{tz}
Hence the right-hand square in the following commutative diagram in $\CatV$ is a pullback.
    \begin{tz}
    \node[](1') {$\llbracket\mathbbm{2},\slice{\bA}{A}\rrbracket$};
    \node[below of=1'](3') {$\un$};
\node[right of=1',xshift=1.4cm](1) {$\llbracket\rhd,\bA\rrbracket$}; 
\node[below of=1](3) {$\bA$};
\node[right of=1,xshift=1.6cm](2) {$\llbracket\mathbbm{2}\times\mathbbm{2},\bA\rrbracket$}; 
\node[below of=2](4) {$\llbracket\mathbbm{2},\bA\rrbracket$};
\pullback{1};

\draw[->] (1) to (2);
\draw[->] (2) to node[right,la]{$(\langle 1\rangle\times \id_{\mathbbm{2}})^*$} (4);
\draw[->] (1) to node[left,la]{$\langle \top\rangle^*$} (3);
\draw[->] (3) to node[below,la]{$!^*$} (4);
\draw[->,dashed] (1') to (1);
\draw[->] (1') to (3');
\draw[->] (3') to node[below,la]{$A$} (3);
\draw[->,bend left=20] (1') to (2);
\end{tz}
By the universal property of the pullback, there is a unique dashed map as depicted. Moreover, the outer square can be seen to be a pullback by applying $\llbracket\mathbbm{2},-\rrbracket$ to the pullback from \cref{rem:pullbackslice}. Hence, by cancellation of pullbacks, the left-hand square is also a pullback.

Now consider the following commutative diagram in $\CatV$.
\begin{tz}
    \node[](1') {$\llbracket\rhd,\bA\rrbracket$};
    \node[below of=1'](3') {$\bA$};
\node[right of=1',xshift=1.2cm](1) {$\llbracket\mathbbm{2},\bA\rrbracket$}; 
\node[below of=1](3) {$\bA$};
\node[right of=1,xshift=1.2cm](2) {$\bA$}; 
\node[below of=2](4) {$\bA$};
    \node[below of=3'](5') {$\un$};
\node[below of=3](5) {$\un$};
\node[below of=4](6) {$\un$};

\draw[->] (2) to node[above,la]{$!^*$} (1);
\draw[d] (2) to (4);
\draw[->] (1) to node[right,la]{$\langle 1\rangle^*$} (3);
\draw[d] (4) to (3);
\draw[->] (1') to node[above,la]{$(\id_{\mathbbm{2}}\times \langle 1\rangle)^*$} (1);
\draw[->] (1') to node[left,la]{$\langle \top\rangle^*$} (3');
\draw[d] (3') to (3);

\draw[->] (5') to node[left,la]{$A$} (3');
\draw[->] (5) to node[left,la]{$A$} (3);
\draw[->] (6) to node[right,la]{$A$} (4);
\draw[d] (5') to (5);
\draw[d] (6) to (5);
\end{tz}
By first computing the pullback of each row, we get the following span
\begin{tz}
    \node[](1') {$\llbracket\mathbbm{2},\bA\rrbracket$};
\node[right of=1',xshift=1.2cm](1) {$\bA$}; 
\node[right of=1,xshift=1.2cm](2) {$\un$}; 

\draw[->] (2) to node[above,la]{$A$} (1);
\draw[->] (1') to node[above,la]{$\langle 1\rangle^*$} (1);
\end{tz}
whose pullback is given by $\slice{\bA}{A}$ using \cref{rem:pullbackslice}. Next, by computing the pullback of each column, we get the following span 
\begin{tz}
    \node[](1') {$\llbracket\mathbbm{2},\slice{\bA}{A}\rrbracket$};
\node[right of=1',xshift=1.2cm](1) {$\slice{\bA}{A}$}; 
\node[right of=1,xshift=1.2cm](2) {$\un$}; 

\draw[->] (2) to node[above,la]{$i_A$} (1);
\draw[->] (1') to node[above,la]{$\langle 1\rangle^*$} (1);
\end{tz}
whose pullback is given by $\slice{(\slice{\bA}{A})}{i_A}$. We therefore have an isomorphism 
\[ \slice{(\slice{\bA}{A})}{i_A}\cong \slice{\bA}{A}\]
which is given by the desired projection; this can be seen by direct inspection of the canonical functor $\rhd\to \mathbbm{2}$.
\end{proof}

\subsection{Representable internal discrete fibrations} \label{subsec:repfib}

 Let us now recall the notion of $\cV$-repre\-sen\-table functors.

\begin{rem}\label{rem:xnattr}
    Let $F\colon \cC^\op\to \cV$ be a $\cV$-functor and $C\in\cC$ be an object. By the enriched Yoneda lemma (see \cite[\S 1.9]{Kelly}), a $\cV$-natural transformation
    \[ \alpha\colon \cC(-,C)\Rightarrow F \]
    is determined by a unique element $x\colon \un\to FC$ given by the composite of maps in $\cV$
    \[ \un\xrightarrow{i_C} \cC(C,C)\xrightarrow{\ev^\alpha_C} FC. \]
    We write $\alpha_x\colon \cC(-,C)\Rightarrow F$ for the $\cV$-natural transformation corresponding to $x$.
\end{rem}

\begin{defn}
    Let $F\colon \cC^\op\to \cV$ be a $\cV$-functor. We say that $F$ is \textbf{$\cV$-representable} by a pair~$(C,x)$ of an object $C\in\cC$ and an element $x\colon \un \to FC$ if the $\cV$-natural transformation
    \[ \alpha_x\colon \cC(-,C)\Rightarrow F\]
    is a $\cV$-natural isomorphism.
\end{defn}

Since the result we are after seeks to characterize $\cV$-representable $\cV$-functors in terms of their internal category of elements, it is crucial to understand what this construction does to a $\cV$-functor of the form $\cC(-,C)\colon\cC^\op\to\cV$ for an object $C\in \cC$. As we now show, this construction recovers the slice internal category $\slice{\Int\cC}{C}$.

\begin{constr}\label{constr:psi}
    Let $\cC$ be a $\cV$-category and $C\in \cC$ be an object. We construct an internal functor 
    \begin{tz}
\node[](1) {$\int_\cC \cC(-,C)$}; 
\node[below of=1,xshift=2.2cm,yshift=.2cm](3) {$\Int\cC$};
\node[above of=3,xshift=2.2cm,yshift=-.2cm](2) {$\slice{\Int\cC}{C}$}; 

\draw[->] (1) to node[above,la]{$\Psi_C$} (2);
\draw[->] (2) to (3);
\draw[->] (1) to (3);
\end{tz}
    For this, we first construct an internal functor 
    \[ \textstyle\int_\cC \cC(-,C)\to \llbracket\mathbbm{2},\Int\cC\rrbracket. \]
    At level $0$, it is given by the canonical inclusion 
    \[ \textstyle\bigsqcup_{A\in \cC} \cC(A,C) \to \bigsqcup_{A,B\in \cC} \cC(A,B), \]
    and at level $1$, it is given by the map 
    \[ \textstyle\bigsqcup_{A,B\in \cC} \cC(A,B)\times \cC(B,C) \to \bigsqcup_{A,B,C,D\in \cC} (\cC(A,B)\times \cC(B,C))\times_{\cC(A,C)} (\cC(A,C)\times \cC(C,D)) \]
    induced by the maps in $\cV$
    \[ \id\times (c_{A,B,C},i_C)\colon \cC(A,B)\times \cC(B,C)\to (\cC(A,B)\times \cC(B,C))\times_{\cC(A,C)} (\cC(A,C)\times \cC(C,C)). \]
    By the universal property of the pullback defining $\slice{\Int\cC}{C}$ from \cref{defn:slice}, this induces a unique internal functor $\int_\cC \cC(-,C)\to \slice{\Int\cC}{C}$ over $\Int\cC$, as desired.
\end{constr}

\begin{theorem}\label{cor:slicehom}
Let $\cC$ be a $\cV$-category and $C\in \cC$ be an object. Then the internal functor $\Psi_C$ of \cref{constr:psi} is an isomorphism in $\slice{\CatV}{\Int\cC}$.
\begin{tz}
\node[](1) {$\int_\cC \cC(-,C)$}; 
\node[below of=1,xshift=2.2cm,yshift=.2cm](3) {$\Int\cC$};
\node[above of=3,xshift=2.2cm,yshift=-.2cm](2) {$\slice{\Int\cC}{C}$}; 

\draw[->] (1) to node[above,la]{$\cong$} node[below,la]{$\Psi_C$} (2);
\draw[->] (2) to (3);
\draw[->] (1) to (3);
\end{tz}
\end{theorem}

\begin{proof}
    Recall from \cref{prop:sliceisdiscfib,grothdiscfib} that the internal functors $\slice{\Int\cC}{C}\to \Int\cC$ and $\int_\cC \cC(-,C)\to \Int\cC$ are internal discrete fibrations. Hence, by \cref{prop:isoonfibers}, to show that $\int_\cC \cC(-,C)$ and $\slice{\Int\cC}{C}$ are isomorphic over $\Int\cC$ it suffices to show that their fibers are isomorphic. By direct inspection, we see that the fibers at an object $A\in \cC$ are given in both cases by the hom-object $\cC(A,C)$.
\end{proof}

\begin{rem}\label{rem:functorsfromint}
   Let $F\colon \cC^{\op}\to \cV$ be a $\cV$-functor and $C\in \cC$ be an object. We can use the above result to construct internal functors $\slice{\Int\cC}{C}\to\int_\cC F$ to $\cV$. Indeed, given an element $x\colon \un\to FC$, the $\cV$-natural transformation $\alpha_x\colon\cC(-,C)\Rightarrow F$ from \cref{rem:xnattr} induces an internal functor $\int_\cC\alpha_x\colon\int_\cC \cC(-,C)\to\int_\cC F$. We then get an internal functor in $\slice{\CatV}{\Int\cC}$ \[ \textstyle\int_\cC\alpha_x \circ \Psi_C^{-1}\colon\slice{\Int\cC}{C}\to\int_\cC F \] 
   such that its pre-composition with $i_C\colon \un \to \slice{\Int\cC}{C}$ is precisely $(C,x)\colon \un\to \int_\cC F$,
   where $\Psi_C$ is the isomorphism from \cref{cor:slicehom}.

    Moreover, every internal functor $\slice{\Int\cC}{C}\to\int_\cC F$ to $\cV$ must arise in this manner in a unique way, since by pre-composing with the isomorphism $\Psi_C\colon\int_\cC \cC(-,C)\to\slice{\Int\cC}{C}$ we obtain an internal functor $\int_\cC \cC(-,C)\to\int_\cC F$, and \cref{thm:equiv} ensures that the functor $\int_\cC$ is fully faithful.
\end{rem}

\subsection{The representation theorem} \label{subsec:proofofrepthm}

We now establish some technical results that will allow us to prove our desired characterization of $\cV$-representable functors.

\begin{constr}\label{constr:slicemap}
    Let $P\colon\bA\to\Int\cC$ be an internal functor, and consider a commutative diagram in $\CatV$.
    \begin{tz}
\node[](1) {$\slice{\Int\cC}{C}$}; 
\node[below of=1,xshift=2cm,yshift=.2cm](3) {$\Int\cC$};
\node[above of=3,xshift=2cm,yshift=-.2cm](2) {$\bA$}; 

\draw[->] (1) to node[above,la,xshift=-5pt]{$H$} (2);
\draw[->] (2) to node[right,la,yshift=-4pt]{$P$} (3);
\draw[->] (1) to (3);
\end{tz}
For every object $C\in \cC$, we get an element of $\bA$ given by the composite of internal functors
\[ \un\xrightarrow{i_C} \slice{\Int\cC}{C}\xrightarrow{H} \bA \]
whose post-composition with $P$ is the canonical inclusion $\iota_C\colon \un\to \Int\cC$. In particular, by the universal property of the pullback defining $P^{-1}C$, the above map induces a unique map in $\cV$
\[ H\circ i_C\colon \un\to P^{-1} C. \]
Moreover, we get an internal functor $(\slice{\Int\cC}{C})_{/i_C}\to \slice{\bA}{H\circ i_C}$ induced by the universal property of the pullback as depicted below.
\begin{tz}
\node[](1) {$\slice{(\slice{\Int\cC}{C})}{i_C}$}; 
\node[below of=1,yshift=-.5cm](3) {$\slice{\Int\cC}{C}$};
\node[right of=1,xshift=3.5cm](2) {$\llbracket\mathbbm{2},\slice{\Int\cC}{C}\rrbracket$}; 
\node[below of=2,yshift=-.5cm](4) {$\slice{\Int\cC}{C}\times \slice{\Int\cC}{C}$}; 
\pullback{1};

\draw[->] (1) to (2);
\draw[->] (2) to node[right,la,pos=0.3]{$\langle 0,1\rangle^*$} (4);
\draw[->] (1) to (3);
\draw[->] (3) to node[above,la,pos=0.32,yshift=-2pt]{$\id_{\slice{\Int\cC}{C}}\times i_C$} (4);

\node[right of=1,yshift=-1cm,xshift=1.2cm](1') {$\slice{\bA}{H\circ i_C}$}; 
\node[below of=1',yshift=-.5cm](3') {$\bA$};
\node[right of=1',xshift=3.5cm](2') {$\llbracket\mathbbm{2},\bA\rrbracket$}; 
\node[below of=2',yshift=-.5cm](4') {$\bA\times \bA$}; 
\pullback{1'};

\draw[->,w] (1') to (2');
\draw[->] (2') to node[right,la]{$\langle 0,1\rangle^*$} (4');
\draw[->,w] (1') to (3');
\draw[->] (3') to node[below,la]{$\id_{\bA}\times (H\circ i_C)$} (4');

\draw[->,dashed] (1) to (1');
\draw[->] (2) to node[above,la,pos=0.6,xshift=10pt]{$\llbracket\mathbbm{2},H\rrbracket$} (2');
\draw[->] (3) to node[below,la,pos=0.4,xshift=-3pt]{$H$} (3');
\draw[->] (4) to node[above,la,pos=0.6,xshift=10pt]{$H\times H$} (4');
\end{tz}
\end{constr}

\begin{prop}\label{prop:isoslices}
Let $P\colon\bA\to\Int\cC$ be an internal discrete fibration in $\CatV$, $C\in \cC$ be an object, and $H\colon \slice{\Int\cC}{C}\to \bA$ be an 
internal functor over $\Int\cC$. Then the induced internal functor from  \cref{constr:slicemap} is an isomorphism in $\CatV$ 
\[ (\slice{\Int\cC}{C})_{/i_C}\xrightarrow{\cong} \slice{\bA}{H\circ i_C}. \]
\end{prop}

\begin{proof}
First recall that $i_C$ is an internal terminal object in $\slice{\Int\cC}{C}$ by \cref{terminalobjslice}; hence, the projection $(\slice{\Int\cC}{C})_{/i_C}\to\slice{\Int\cC}{C}$ is an isomorphism. We can use its inverse to get an internal functor \[\slice{\Int\cC}{C}\to \slice{\bA}{H\circ i_C},\] and by 2-out-of-3, it is enough to prove that this is an isomorphism to conclude our result.

 Recall from \cref{prop:sliceisdiscfib} that the internal functors $\slice{\Int\cC}{C}\to \Int\cC$ and $\slice{\bA}{H\circ i_C}\to \bA$ are internal discrete fibrations, and so the composite $\slice{\bA}{H\circ i_C}\to \bA\to \Int\cC$ is also an internal discrete fibration. Hence, by \cref{prop:isoonfibers}, to show that $\slice{\Int\cC}{C}$ and $\slice{\bA}{H\circ i_C}$ are isomorphic over $\Int\cC$, it suffices to show that their levels $0$ are isomorphic. For this, first note that we have the following pullback in $\cV$.
    \begin{tz}
\node[](1) {$\cC(A,C)$}; 
\node[below of=1](3) {$P^{-1}A$};
\node[right of=1,xshift=3.3cm](2) {$\cC(A,C)\times P^{-1} C$}; 
\node[below of=2](4) {$P^{-1} A\times P^{-1} C$}; 
\pullback{1};

\draw[->] (1) to node[above,la]{$\id_{\cC(A,C)}\times (H\circ i_C)$} (2);
\draw[->] (2) to node[right,la]{$(\ev^{P^{-1}}_{A,C},\pi_1)$} (4);
\draw[->] (1) to (3);
\draw[->] (3) to node[below,la]{$\id_{P^{-1} A}\times (H\circ i_C)$} (4);
\end{tz}
Hence, by extensivity, this induces a pullback in $\cV$ as below.
    \begin{tz}
\node[](1) {$\bigsqcup_{A\in \cC} \cC(A,C)$}; 
\node[below of=1](3) {$\bA_0=\bigsqcup_{A\in \cC} P^{-1}A$};
\node[right of=1,xshift=6.7cm](2) {$\bigsqcup_{A,B\in \cC} \cC(A,B)\times P^{-1} B=\bA_1$}; 
\node[below of=2](4) {$\bigsqcup_{A\in \cC} P^{-1} A\times \bigsqcup_{B\in \cC} P^{-1} B=\bA_0\times \bA_0$}; 
\pullback{1};

\draw[->] (1) to (2);
\draw[->] (2) to node[right,la]{$\bigsqcup_{A\in\cC}(\ev^{P^{-1}}_{A,C},\pi_1)$} (4);
\draw[->] (1) to (3);
\draw[->] (3) to node[below,la]{$\bigsqcup_{A\in \cC}\id_{P^{-1} A}\times (H\circ i_C)$} (4);
\end{tz}
By \cref{defn:slice}, this pullback is precisely $(\slice{\bA}{H\circ i_C})_0$, and so we have a canonical isomorphism $(\slice{\Int\cC}{C})_0\cong (\slice{\bA}{H\circ i_C})_0$. Moreover, this map is level $0$ of the internal functor $\slice{\Int\cC}{C}\cong \slice{\bA}{H\circ i_C}$, as it is the unique internal functor induced by the universal property of the pullback.
\end{proof}

We are finally equipped to prove one of the main results of this paper.

\begin{theorem}\label{thm:representation}
Let $F\colon \cC^{\op}\to \cV$ be a $\cV$-functor. Given a pair $(C,x)$ of an object $C\in \cC$ and an element $x\colon \un\to FC$ in $\cV$, the following are equivalent: 
\begin{rome}
    \item the $\cV$-functor $F$ is $\cV$-representable by $(C,x)$, 
    \item the pair $(C,x)$ is an internal terminal object in the internal category of elements $\int_\cC F$. 
\end{rome}
\end{theorem}

\begin{proof}
Consider the composite of internal functors $\int_\cC\alpha_x \circ \Psi_C\colon\slice{\Int\cC}{C}\to\int_\cC F$ in $\slice{\CatV}{\Int\cC}$ as in \cref{rem:functorsfromint}. Recall that, by \cref{constr:slicemap}, we have a commutative diagram in $\CatV$
\begin{tz}
\node[](1) {$\slice{(\slice{\Int\cC}{C})}{i_C}$}; 
\node[below of=1](3) {$\slice{\Int\cC}{C}$};
\node[right of=1,xshift=2cm](2) {$\slice{\int_\cC F}{(C,x)}$}; 
\node[below of=2](4) {$\int_\cC F$}; 

\draw[->] (1) to node[above,la]{$\cong$} (2);
\draw[->] (2) to node[right,la]{} (4);
\draw[->] (1) to node[left,la]{$\cong$} (3);
\draw[->] (3) to node[below,la]{$\int_\cC\alpha_x \circ \Psi_C$} (4);
\end{tz}
where the top internal functor is an isomorphism by \cref{prop:isoslices}, and the left-hand internal functor is an isomorphism by \cref{terminalobjslice}.

By definition, we have that $(C,x)$ is an internal terminal object in $\int_\cC F$ if and only if the projection from the slice internal category $\slice{\int_\cC F}{(C,x)}\to\int_\cC F$ is an isomorphism. By $2$-out-of-$3$ in the above diagram, this is the case if and only if the internal functor $\int_\cC\alpha_x \circ \Psi_C\colon \slice{\Int\cC}{C}\to\int_\cC F$ is an isomorphism, and since $\Psi_C$ is an isomorphism by \cref{cor:slicehom}, this happens if and only if $\int_\cC\alpha_x \colon\int_\cC \cC(-,C)\to\int_\cC F$ is an isomorphism in $\slice{\CatV}{\Int\cC}$. In turn, this happens if and only if  $\alpha_x\colon\cC(-,C)\Rightarrow F$ is a $\cV$-natural isomorphism, as the internal Grothendieck construction $\int_\cC\colon \VCat(\cC^{\op},\cV)\to \slice{\Dfib}{\Int\cC}$ is an equivalence of categories by \cref{thm:equiv}, and as such it creates isomorphisms. Finally, by definition, this is equivalent to the $\cV$-functor $F\colon\cC^{\op}\to\cV$ being $\cV$-representable by $(C,x)$.
\end{proof}

\begin{ex} \label{ex:repthm} 
    Going back to our examples, using the interpretation and terminology from \cref{ex:catofelements,ex:doubleterminal}, \cref{thm:representation} translates as follows. 
    \begin{enumerate}
        \item When $\cV=\Set$, it retrieves the classical result which states that a functor $F\colon \cC^{\op}\to \Set$ is representable if and only if its category of elements $\int_\cC F$ has a terminal object; see e.g.~\cite[Proposition 2.4.8]{Riehlcontext}. 
        \item When $\cV=\Cat$, we retrieve a stricter analogue of the equivalence between (i) and (ii) in \cite[Theorem 6.8]{cM2}. In our case, the theorem proves that a $2$-functor $F\colon \cC^{\op}\to \Cat$ is $2$-representable if and only if its double category of elements $\int_\cC F$ has a double terminal object. 
        \item[($n$)] When $\cV=(n-1)\Cat$ for $n\geq 3$, the theorem says that an $n$-functor $F\colon \cC^{\op}\to (n-1)\Cat$ is $n$-representable if and only if its double $(n-1)$-category of elements $\int_\cC F$ has a double $(n-1)$-terminal object. 
    \end{enumerate}
\end{ex}

\subsection{An equivalent formulation in the language of enriched categories} \label{subsec:reformulate}

There is also a notion of terminal object in the setting of $\cV$-categories, which we now recall.

\begin{defn}\label{defn:Vterminal}
    Let $\cC$ be a $\cV$-category, and $T\in \cC$ be an object. We say that $T$ is a \textbf{$\cV$-terminal object} in $\cC$ if the unique map $\cC(A,T)\to \un$ in $\cV$ is an isomorphism, for every object~$A\in \cC$. 
\end{defn} 

It is natural to try to compare \cref{defn:internalterminal} and \cref{defn:Vterminal}; that is, to wonder whether the data of an internal terminal object in an internal category is the same as the data of a $\cV$-terminal object in its underlying $\cV$-category. The following shows that the notion of internal terminal object is stronger than that of $\cV$-terminal object.

\begin{prop} \label{internalterminalimpliesVterminal}
    Let $\bA$ be an internal category to $\cV$, and $T\colon \un\to \bA$ be an element. If $T$ is an internal terminal object in $\bA$, then $T$ is a $\cV$-terminal object in the underlying $\cV$-category~$\Und\bA$.
\end{prop}

\begin{proof}
Let $A\colon \un\to \bA_0$ be an object in $\Und\bA$, and consider the following diagram in~$\cV$.
\begin{tz}
\node[](5) {$\Und\bA(A,T)$}; 
\node[below of=5](6) {$ \un$};
\node[right of=5, xshift=1.8cm](1) {$(\slice{\bA}{T})_0$}; 
\node[below of=1](3) {$\bA_0$};
\node[right of=1, xshift=1.3cm](2) {$\bA_1$};
\node[below of=2](4) {$\bA_0\times\bA_0$};
\pullback{1};

\draw[->,dashed] (5) to (1);
\draw[->] (5) to (6);
\draw[->] (6) to node[below,la]{$A$} (3);
\draw[->] (1) to (2);
\draw[->] (2) to node[right,la]{$(s,t)$} (4);
\draw[->] (1) to node[left,la]{$\cong$} (3);
\draw[->] (3) to node[below,la]{$\id_{\bA_0}\times T$} (4);
\draw[->,bend left=20] (5) to (2);
\end{tz}
Here, the middle vertical map is an isomorphism since $T$ is internal terminal, and the right-hand square and the outer square are pullbacks by definition. Hence, by pullback cancellation the left-hand square is also a pullback, and so $\Und\bA(A,T)\to \un$ is an isomorphism. This shows that $T$ is $\cV$-terminal in $\Und\bA$.
\end{proof}

While true for $\cV=\Set$, the converse does not hold in general. Notably, this explains why the stronger notion of an internal terminal object is needed in \cref{thm:representation}, and in general we cannot simply encode the data of the representability of a $\cV$-functor in terms of a $\cV$-terminal object in the underlying $\cV$-category $\Und\int_\cC F$, as was already understood in the case of $\cV=\Cat$ \cite{cM1}. However, it is possible to have an equivalent statement to that of \cref{thm:representation} solely in the language of enriched categories, as long as we are willing to consider ``shifted'' underlying $\cV$-categories of generalized elements, as we now explain.

\begin{defn}\label{defn:conservative}
    A family of objects $\cG$ is \textbf{conservative} if it satisfies the following condition: a map $f\colon Y\to Z$ in $\cV$ is an isomorphism in $\cV$ if and only if, for every object $X\in \cG$, the induced~map 
    \[ f_*\colon \cV(X,Y)\to \cV(X,Z) \]
    is an isomorphism of sets. 
\end{defn}

\begin{ex}\label{ex:conservatives}
 For any category $\cV$, the Yoneda lemma ensures that there exists at least one conservative family: the one consisting of all objects in $\cV$. In practice, we can often do better, as illustrated by the following examples.
    \begin{enumerate}
        \item When $\cV=\Set$, the family containing just the one-point set $\{*\}$ is conservative.
        \item When $\cV=\Cat$, the family containing just the free-living morphism $\mathbbm{2}$ is conservative. To see this, note that a functor is an isomorphism if and only if it is an isomorphism on objects and on morphisms; however, the map on objects is a retract of the map on morphisms and so it is an isomorphism if the map on morphisms is so.         
        \item[($n$)] When $\cV=(n-1)\Cat$ for $n\geq 3$, the family containing just the free-living $(n-1)$-morphism $C_{n-1}$ is conservative, by a similar argument to the one in (2).
    \end{enumerate}
\end{ex}

\begin{rem} \label{rem:homcstX}
    Given an object $X\in \cV$ and an internal category to $\cV$, note that we have the following computations in $\cV$
    \[ \llbracket \cst X,\bA\rrbracket_0=[X,\bA_0] \quad \text{and} \quad \llbracket \mathbbm{2},\llbracket\cst X,\bA\rrbracket\rrbracket_0=\llbracket\cst X,\bA\rrbracket_1=[X,\bA_1]. \]
\end{rem}

\begin{theorem}\label{thm:thenewexample} 
    Let $\cG$ be a conservative family of objects in $\cV$, and $\bA$ be an internal category to $\cV$. Given an element $T\colon \un\to \bA$, the following are equivalent: 
    \begin{rome}
        \item the element $T$ is an internal terminal object in $\bA$, 
        \item for every object $X\in \cG$, the element $T\colon \un\xrightarrow{T} \bA\xrightarrow{!^*} \llbracket \cst X,\bA\rrbracket$ is a $\cV$-terminal object in the $\cV$-category $\Und \llbracket\cst X,\bA\rrbracket$.
    \end{rome}
\end{theorem}

\begin{proof}
    We first prove that (i) implies (ii). By definition, if $T$ is an internal terminal object in~$\bA$, we have that the projection $\slice{\bA}{T}\to \bA$ is an isomorphism of internal categories. Applying the functor $\llbracket \cst X,-\rrbracket$ to this isomorphism, we obtain that the projection 
    \[ \slice{\llbracket \cst X,\bA\rrbracket}{T}\cong \llbracket \cst X,\slice{\bA}{T}\rrbracket\to \llbracket \cst X,\bA\rrbracket\]
    is also an isomorphism, where $T$ is seen as an element of $\llbracket \cst X,\bA\rrbracket$ by post-composing $T\colon \un\to \bA$ with $\bA\xrightarrow{!^*}\llbracket \cst X,\bA\rrbracket$. Hence,  \cref{internalterminalimpliesVterminal} ensures that $T$ is $\cV$-terminal in the underlying $\cV$-category $\Und \llbracket \cst X,\bA\rrbracket$, as desired. 

    Next, we prove that (ii) implies (i). Suppose that, for every object $X\in \cG$, the element $T\colon \un\xrightarrow{T} \bA\xrightarrow{!^*} \llbracket\cst X,\bA\rrbracket$ is $\cV$-terminal in $\Und\llbracket\cst X,\bA\rrbracket$; we need to show that the internal functor $\slice{\bA}{T}\to \bA$ is an isomorphism. Since the identity at $\bA$ is an internal discrete fibration, as is the projection $\slice{\bA}{T}\to \bA$ by \cref{prop:sliceisdiscfib}, it then follows from \cref{prop:isoonfibers} that this is equivalent to showing that the induced map $(\slice{\bA}{T})_0\to \bA_0$ is an isomorphism in $\cV$. By definition of the conservative family $\cG$, this is equivalent to showing that, for every object $X\in \cG$, the induced map 
\[ \cV(X,(\slice{\bA}{T})_0)\to \cV(X,\bA_0) \]
is an isomorphism of sets. 

Consider an element $G$ in the right-hand set; i.e., a map $G\colon X\to \bA_0$ in $\cV$. Then $G$ can be regarded as an internal functor $G\colon \cst X\to \bA$, or equivalently, as an object $G\colon \un\to \llbracket\cst X,\bA\rrbracket$. Since $T$ is $\cV$-terminal in $\Und\llbracket\cst X,\bA\rrbracket$, there is an isomorphism $\Und \llbracket\cst X,\bA\rrbracket(G,T)\cong \un$ in $\cV$ and hence an isomorphism of sets 
\[ \cV(\un, \Und\llbracket\cst X,\bA\rrbracket(G,T))\cong \cV(\un,\un)\cong \{*\}.\]
Therefore, there is a unique map $f\colon \un \to \Und \llbracket\cst X,\bA\rrbracket(G,T)$ in $\cV$. The map $f$ can equivalently be seen, by definition of $\Und \llbracket\cst X,\bA\rrbracket$ and using \cref{rem:homcstX}, as a map $f\colon \un\to [X,\bA_1]$ in $\cV$ making the diagram  below left commute, and then, by the universal property of the adjunction $X\times (-)\dashv [X,-]$, as a map $X\to \bA_1$ in $\cV$ making the  diagram below right commute.
\begin{tz}
    \node[](1) {$\un$}; 
        \node[right of=1,xshift=1cm](2) {$[X,\bA_1]$}; 
        \node[below of=2](3) {$[X,\bA_0]\times [X,\bA_0]$}; 

        \draw[->] (1) to node[above,la]{$f$} (2); 
        \draw[->] (2) to node[right,la] {$(s,t)$} (3); 
        \draw[->] (1) to node[left,la,yshift=-6pt]{$G\times T$} (3);

        \node[right of=2,xshift=1.5cm](1) {$X$}; 
        \node[right of=1,xshift=1cm](2) {$\bA_1$}; 
        \node[below of=2](3) {$\bA_0\times \bA_0$}; 

        \draw[->] (1) to node[above,la]{$f$} (2); 
        \draw[->] (2) to node[right,la] {$(s,t)$} (3); 
        \draw[->] (1) to node[left,la,yshift=-6pt]{$(G,T)$} (3);
        \end{tz}
        Finally, by the pullback description of $\slice{\bA}{T}$ from \cref{defn:slice}, the map $f\colon X\to \bA_1$ defines a unique map $X\to (\slice{\bA}{T})_0$ whose post-composition with $(\slice{\bA}{T})_0\to \bA_0$ is precisely $G$. This shows the desired bijection. 
\end{proof}

Applying this perspective to \cref{thm:representation}, we get the following.

\begin{theorem} \label{enrichedstatement}
    Let $\cG$ be a conservative family of objects in $\cV$, and $F\colon \cC^{\op}\to \cV$ be a $\cV$-functor. Given a pair $(C,x)$ of an object $C\in \cC$ and an element $x\colon \un\to FC$ in $\cV$, the following are equivalent: 
\begin{rome}
    \item the $\cV$-functor $F$ is $\cV$-representable by $(C,x)$, 
    \item for every object $X\in \cG$, the pair $(C,x)$ is a $\cV$-terminal object in the $\cV$-category of generalized elements $\Und\llbracket \cst X,\int_\cC F\rrbracket$. 
\end{rome}
\end{theorem}

\begin{proof}
    This is obtained by combining \cref{thm:representation,thm:thenewexample}.
\end{proof}

\begin{ex} \label{ex:enrichedrepthm} 
Going back to our examples, using the interpretation from \cref{ex:functors} and the conservative families from \cref{ex:conservatives}, \cref{enrichedstatement} translates as follows. 
    \begin{enumerate}
        \item When $\cV=\Set$, it is identical to the one in \cref{thm:representation}, specialized to $\cV=\Set$ in \cref{ex:repthm}(1), since we can choose the conservative family consisting of~$\{*\}$ and we have that the functor $\Und$ is simply the identity.
        \item When $\cV=\Cat$, we get a stricter analogue of the equivalence between (i) and (iv) in \cite[Theorem 6.8]{cM2}. In our case, the theorem tells us that to test whether a $2$-functor $F\colon \cC^{\op}\to \Cat$ is $2$-representable by a pair $(C,x)$ of objects $C\in \cC$ and $x\in FC$, it is enough to test whether the pair $(C,\id_x)$ is $2$-terminal in the \emph{$2$-category of morphisms} $\Und \llbracket \cst\mathbbm{2},\int_\cC F\rrbracket$, whose objects are pairs $(A,u)$ of an object $A\in \cC$ and a morphism $\varphi$ in~$FA$. 
        \item[($n$)] When $\cV=(n-1)\Cat$ for $n\geq 3$, the theorem tells us that to test whether an $n$-functor $F\colon \cC^{\op}\to (n-1)\Cat$ is $n$-representable by a pair $(C,x)$ of objects $C\in \cC$ and $x\in FC$, it is enough to test whether the pair $(C,\id^{n-1}_x)$---where $\id^{n-1}_x$ denotes the identity $(n-1)$-morphism at the object $x$---is $n$-terminal in the \emph{$n$-category of $(n-1)$-morphisms} $\Und \llbracket \cst C_{n-1},\int_\cC F\rrbracket$, whose objects are pairs $(A,\varphi)$ of an object $A\in \cC$ and an $(n-1)$-morphism $\varphi$ in $FA$. 
    \end{enumerate}
\end{ex}

\section{The representation theorem in the presence of tensors}

In this section, let $\cV$ be an extensive and cartesian closed category with pullbacks. The aim of this section is to study when \cref{thm:representation}, which characterizes the representability of a $\cV$-functor $F\colon\cC^\op\to\cV$ in terms of the existence of an internal terminal object in its internal category of elements $\int_\cC F$, can be expressed solely in the language of $\cV$-categories by looking for $\cV$-terminal objects in the underlying $\cV$-category $\Und\int_\cC F$.

As a first step, we investigate when internal terminal objects to an internal category $\bA$ agree with $\cV$-terminal objects in $\Und\bA$. In \cref{subsec:tensors} we show that this is the case when $\bA$ has certain \emph{internal tensors}; this is the content of \cref{thm:enrvsintterminal}. 

Since we are interested in internal terminal objects in $\int_\cC F$, the natural next step is to study when this specific internal category has  internal tensors. Unfortunately, this is seldom the case, and we are required to modify our approach. In \cref{subsec:Ctensors}, we start from an internal discrete fibration $\bA\to\Int\cC$ and define a more restricted notion of \emph{$\cC$-internal tensors} in $\bA$. \cref{thm:ctensorsterminal} then shows that if $\bA$ admits certain $\cC$-internal tensors, we can still freely move between internal and enriched terminal objects.

Finally, \cref{subsec:reptensors} focuses on the internal discrete fibration $\int_\cC F\to\Int\cC$. \cref{prop:grothtensors} shows that $\int_\cC F$ has the required $\cC$-internal tensors when $\cC$ has the corresponding tensors in the enriched sense and $F\colon \cC^{\op}\to \cV$ preserves them, and \cref{thm:enrichedrepthm} achieves our goal of a representation theorem that only involves the underlying $\cV$-category of $\int_\cC F$. 

\subsection{Internal terminal objects in the presence of internal tensors}\label{subsec:tensors}

Recall from \cref{internalterminalimpliesVterminal} that an internal terminal object is in particular $\cV$-terminal, but that the converse is generally false. Here we show that, in the presence of \emph{internal tensors}, the two notions are in fact equivalent.

\begin{notation}
    Given an internal category $\bA$ to $\cV$ and an element $I\colon \un\to \bA$, we denote by~$\sliceunder{\bA}{I}$ the dual version of \cref{defn:slice} obtained by pulling back along the internal functor $I\times \id_{\bA}\colon \bA\to \bA\times \bA$.
\end{notation}

\begin{defn}
    Let $\bA$ be an internal category to $\cV$, and $I\colon \un\to \bA$ be an element. We say that $I$ is an \textbf{internal initial object} in $\bA$ if the canonical projection $\sliceunder{\bA}{I}\to \bA$ is an isomorphism. 
\end{defn}

\begin{notation} \label{not:catofcocones}
    Given an internal functor $G\colon \bI\to \bA$ to $\cV$, we denote by $G\downarrow \Delta$ the dual version of the comma internal category from \cref{defn:internalcomma} obtained by pulling back along the internal functor $G\times \Delta\colon \bA\to \llbracket\bI,\bA\rrbracket\times\llbracket\bI,\bA\rrbracket$. We refer to $G\downarrow \Delta$ as the \textbf{internal category of cocones under $G$}.
\end{notation}

\begin{defn} \label{def:internalcolimit}
    Let $G\colon \bI\to \bA$ be an internal functor to $\cV$. An \textbf{internal colimit} of $G$ is an internal initial object in the internal category $G\downarrow \Delta$ of cocones under $G$. 
\end{defn}

We are interested in internal colimits of a particular type of internal functors: those of the form $\cst X\to \bA$. Recall that these correspond under the adjunction $\cst\dashv (-)_0$ to maps $X\to \bA_0$ in $\cV$. We introduce the following terminology.

\begin{defn}
    Given an object $X\in \cV$, we say that an internal category $\bA$ has all \textbf{internal tensors by $X$}\footnote{Our choice of terminology is due to the fact that internal tensors are closely related to $\cV$-tensors; for a detailed explanation, see \cref{rem:justifytensors}.} if, for every internal functor $G\colon \cst X\to \bA$, the internal colimit of $G$ exists. 
\end{defn}

\begin{ex} \label{ex:tensors}
    In particular, we will only be interested in internal tensors by a conservative family of objects in $\cV$; see \cref{defn:conservative}. For our usual examples, families of conservative objects can be chosen as in \cref{ex:conservatives} and we get the following. 
    \begin{enumerate}
        \item When $\cV=\Set$, internal tensors by the point $\{*\}$ always exist in any category $\bA$ and, for a given functor $A\colon \{*\}\to \bA$, it is given by the object $A$ itself.
        \item When $\cV=\Cat$, internal tensors by the free-living morphism $\mathbbm{2}$ are called \emph{cotabulators}; see \cite[\S 5.3]{GrandisPare}.
        \item[($n$)] When $\cV=(n-1)\Cat$ for $n\geq 3$, we refer to internal tensors by the free-living $(n-1)$-morphism $C_{n-1}$ as \emph{$n$-cotabulators}.
    \end{enumerate}
\end{ex}

In order to study internal tensors, it will be useful to have an explicit description of the internal category $G\downarrow \Delta$ for a given internal functor $G\colon\cst X\to\bA$.

\begin{prop} \label{descrslice}
    Let $X\in \cV$ be an object, and $G\colon \cst X\to \bA$ be an internal functor to $\cV$, or equivalently a map $G\colon X\to \bA_0$ in $\cV$. Given an object $Y\in \cV$, we have that
    \begin{rome}
        \item a map $Y\to (G\downarrow \Delta)_0$ consists of a pair $(y,\alpha)$ of maps $y\colon Y\to \bA_0$ and $\alpha\colon X\times Y\to \bA_1$ in $\cV$ such that the following diagram commutes,
    \begin{tz}
        \node[](1) {$X\times Y$}; 
        \node[right of=1,xshift=1cm](2) {$\bA_1$}; 
        \node[below of=2](3) {$\bA_0\times \bA_0$}; 

        \draw[->] (1) to node[above,la]{$\alpha$} (2); 
        \draw[->] (2) to node[right,la] {$(s,t)$} (3); 
        \draw[->] (1) to node[left,la,yshift=-5pt]{$G\times y$} (3);
    \end{tz}
    \item a map $Y\to (G\downarrow \Delta)_1$ whose post-compositions with source and target are given by pairs $(y,\alpha)$ and $(z,\beta)$ consists of a map $f\colon Y\to \bA_1$ in $\cV$ such that the following diagrams commute.
    \begin{tz}
    \node[](1) {$Y$}; 
        \node[right of=1,xshift=1cm](2) {$\bA_1$}; 
        \node[below of=2](3) {$\bA_0\times \bA_0$}; 

        \draw[->] (1) to node[above,la]{$f$} (2); 
        \draw[->] (2) to node[right,la] {$(s,t)$} (3); 
        \draw[->] (1) to node[left,la,yshift=-6pt]{$(y,z)$} (3);
        
        \node[right of=2,xshift=1.5cm](1) {$X\times Y$}; 
        \node[right of=1,xshift=1.7cm](2) {$\bA_1\times_{\bA_0}\bA_1$}; 
        \node[below of=2](3) {$\bA_1$}; 

        \draw[->] (1) to node[above,la]{$(\alpha,f\pi_1)$} (2); 
        \draw[->] (2) to node[right,la] {$c$} (3); 
        \draw[->] (1) to node[left,la,yshift=-6pt]{$\beta$} (3);
    \end{tz}
    \end{rome}
\end{prop}

\begin{proof}
    As pullbacks in $\CatV$ are computed levelwise by \cref{pullbacksinCatV}, we can unpack \cref{not:catofcocones} using \cref{rem:homcstX} to get that levels $0$ and $1$ of the internal category $G\downarrow \Delta$ of cocones under $G$ are given by the following pullbacks in $\cV$,
    \begin{tz}
        \node[](1) {$(G\downarrow \Delta)_0$}; 
\node[below of=1](3) {$\bA_0$};
\node[right of=1,xshift=1.8cm](2) {$[X,\bA_1]$}; 
\node[below of=2](4) {$[X,\bA_0]\times[X,\bA_0]$};
\pullback{1};

\draw[->] (1) to node[above,la]{} (2);
\draw[->] (2) to node[right,la]{$(s,t)$} (4);
\draw[->] (1) to node[left,la]{} (3);
\draw[->] (3) to node[below,la]{$G\times \Delta$} (4);

\node[right of=2,xshift=1.8cm](1) {$(G\downarrow \Delta)_1$}; 
\node[below of=1](3) {$\bA_1$};
\node[right of=1,xshift=3.8cm](2) {$[X,(\bA_1\times_{\bA_0} \bA_1)\times_{\bA_1}(\bA_1\times_{\bA_0} \bA_1)]$}; 
\node[below of=2](4) {$[X,\bA_1]\times [X,\bA_1]$};
\pullback{1};

\draw[->] (1) to node[above,la]{} (2);
\draw[->] (2) to node[right,la]{$(s,t)_*$} (4);
\draw[->] (1) to node[left,la]{} (3);
\draw[->] (3) to node[below,la]{$G\times \Delta$} (4);
    \end{tz}
    where the pullback $(\bA_1\times_{\bA_0} \bA_1)\times_{\bA_1}(\bA_1\times_{\bA_0} \bA_1)$ is taken on both sides over the composition map $c\colon \bA_1\times_{\bA_0}\bA_1\to \bA_1$. Then, using the universal property of the above pullbacks and the adjunction $X\times (-)\dashv [X,-]$, we obtain the desired result.
\end{proof}

\begin{rem}\label{elementsincatofcones}
    In particular, we can use \cref{descrslice} to describe a generic element $\un\to G\downarrow \Delta$. Under the adjunction $\cst\dashv (-)_0$, such an element corresponds to a map $\un\to (G\downarrow\Delta)_0$. Taking $Y=\un$ in \cref{descrslice}(i), we get that it consists of a pair $(y,\alpha)$ of maps $y\colon \un \to \bA_0$ and $\alpha\colon X\to \bA_1$ in $\cV$ such that the following diagram in $\cV$ commutes.
    
    \begin{tz}
        \node[](1) {$X$}; 
        \node[right of=1,xshift=1cm](2) {$\bA_1$}; 
        \node[below of=2](3) {$\bA_0\times \bA_0$}; 

        \draw[->] (1) to node[above,la]{$\alpha$} (2); 
        \draw[->] (2) to node[right,la] {$(s,t)$} (3); 
        \draw[->] (1) to node[left,la,yshift=-5pt]{$G\times y$} (3);
    \end{tz}
\end{rem}

We can also make explicit what it means for an element $\un\to G\downarrow \Delta$ to be an internal colimit of $G$.

\begin{prop} \label{univpropcolim}
    Let $X\in \cV$ be an object and $G\colon \cst X\to \bA$ be an internal functor. An element $(G\otimes X,\gamma)\colon \un\to G\downarrow \Delta$ is an internal colimit of $G$ if and only if, for every map $(y,\alpha)\colon Y\to (G\downarrow \Delta)_0$ in $\cV$, there is a unique map $f\colon Y\to (G\downarrow \Delta)_1$ from $(G\otimes X,\gamma)$ to $(y,\alpha)$; i.e., making the following diagram in $\cV$ commute. 
        \begin{tz}
        \node[](1) {$Y$}; 
        \node[right of=1,xshift=1.3cm](2) {$(G\downarrow \Delta)_1$}; 
        \node[below of=2](3) {$(G\downarrow \Delta)_0\times (G\downarrow \Delta)_0$}; 
        \draw[->,dashed] (1) to node[above,la]{$\exists ! f$} (2); 
        \draw[->] (2) to node[right,la] {$(s,t)$} (3); 
        \draw[->] (1) to node[left,la,yshift=-6pt]{$(G\otimes X,\gamma)\times  (y,\alpha)$} (3);
    \end{tz}
\end{prop}

\begin{proof}
By definition, an object $(G\otimes X,\gamma)$ in $G\downarrow \Delta$ is an internal colimit of $G$ if it is an initial object in $G\downarrow \Delta$; i.e., if the internal functor
\[ P\colon \sliceunder{G\downarrow \Delta}{(G\otimes X,\gamma)}\to G\downarrow \Delta \]
is an isomorphism. Now, the identity at $G\downarrow \Delta$ is an internal discrete opfibration, as is the projection $P$ by the dual of \cref{prop:sliceisdiscfib}; it then follows from the dual of \cref{prop:isoonfibers} that this is equivalent to requiring that the induced map $P_0$ is an isomorphism in $\cV$. By the Yoneda lemma, this is equivalent to saying that, for every $Y\in \cV$, the induced map
\[ (P_0)_*\colon \cV(Y,(\sliceunder{G\downarrow \Delta}{(G\otimes X,\gamma)})_0)\to \cV(Y,(G\downarrow \Delta)_0) \]
is an isomorphism of sets. 

By direct inspection, an element in the left-hand set is a map $Y\to (G\downarrow \Delta)_1$ whose source is $(G\otimes X,\gamma)$, an element in the right-hand set is a map $Y\to (G\downarrow \Delta)_0$, and the assignment is given by sending each $Y\to (G\downarrow \Delta)_1$ to its target. Hence, we get the desired result.
\end{proof}

We can now use the notion of internal tensors to obtain a converse to \cref{internalterminalimpliesVterminal}. In particular, for this, we only need to require the existence of internal tensors by a conservative family of objects in $\cV$; see \cref{defn:conservative}. 

\begin{theorem}\label{thm:enrvsintterminal} 
Let $\cG$ be a conservative family of objects in $\cV$, and $\bA$ be an internal category to~$\cV$ that has all internal tensors by objects $X\in \cG$. Given an element $T\colon \un\to \bA$, the following are equivalent: 
\begin{rome}
        \item the element $T$ is an internal terminal object in~$\bA$,
        \item the element $T$ is a $\cV$-terminal object in the underlying $\cV$-category $\Und\bA$.
\end{rome}
\end{theorem}

\begin{proof} 
The fact that (i) implies (ii) is \cref{internalterminalimpliesVterminal}. To show that (ii) implies (i), suppose that $T$ is $\cV$-terminal in $\Und\bA$; we need to show that the internal functor $\slice{\bA}{T}\to \bA$ is an isomorphism. Since the identity at $\bA$ is an internal discrete fibration, as is the projection $\slice{\bA}{T}\to \bA$ by \cref{prop:sliceisdiscfib}, it then follows from \cref{prop:isoonfibers} that this is equivalent to showing that the induced map $(\slice{\bA}{T})_0\to \bA_0$ is an isomorphism in $\cV$. By definition of the conservative family $\cG$, this is equivalent to showing that, for every object $X\in \cG$, the induced map 
\[ \cV(X,(\slice{\bA}{T})_0)\to \cV(X,\bA_0) \]
is an isomorphism of sets. 

Consider an element $G$ in the right-hand set; i.e., a map $G\colon X\to \bA_0$ in $\cV$. As $\bA$ has all internal tensors by objects in $\cG$, there is an internal colimit $(G\otimes X,\gamma)$ for the corresponding internal functor $G\colon \cst X\to \bA$, where we recall that $G\otimes X$ is a map $\un\to \bA_0$ and $\gamma$ is a map $X\to \bA_1$ in $\cV$ making the diagram below left commute.
\begin{tz}
    \node[](1) {$X$}; 
        \node[right of=1,xshift=.9cm](2) {$\bA_1$}; 
        \node[below of=2](3) {$\bA_0\times \bA_0$}; 

        \draw[->] (1) to node[above,la]{$\gamma$} (2); 
        \draw[->] (2) to node[right,la] {$(s,t)$} (3); 
        \draw[->] (1) to node[left,la,yshift=-6pt]{$G\times (G\otimes X)$} (3);

\node[right of=2,xshift=.5cm](1) {$\un$}; 
        \node[right of=1,xshift=.9cm](2) {$\bA_1$}; 
        \node[below of=2](3) {$\bA_0\times \bA_0$}; 

        \draw[->] (1) to node[above,la]{$f$} (2); 
        \draw[->] (2) to node[right,la] {$(s,t)$} (3); 
        \draw[->] (1) to node[left,la,yshift=-6pt]{$(G\otimes X)\times T$} (3);

        \node[right of=2,xshift=.5cm](1) {$X$}; 
        \node[right of=1,xshift=1.2cm](2) {$\bA_1\times_{\bA_0}\bA_1$}; 
        \node[right of=2,xshift=1.2cm](3) {$\bA_1$}; 
        \node[below of=3](4) {$\bA_0\times \bA_0$};

        \draw[->] (1) to node[above,la]{$(\gamma,f\circ !)$} (2); 
        \draw[->] (2) to node[above,la] {$c$} (3); 
        \draw[->] (3) to node[right,la]{$(s,t)$} (4);
        \draw[->] (1) to node[left,la,yshift=-8pt]{$G\times T$} (4);
\end{tz}
In particular, the element $G\otimes X\colon \un\to \bA_0$ defines an object of $\Und\bA$. Since $T$ is $\cV$-terminal in $\Und\bA$, there is an isomorphism $\Und\bA(G\otimes X,T)\cong \un$ in $\cV$ and hence an isomorphism of sets 
\[ \cV(\un, \Und\bA(G\otimes X,T))\cong \cV(\un,\un)\cong \{*\}.  \]
Therefore, there is a unique map $f\colon \un\to \Und\bA(G\otimes X,T)$ in $\cV$ which, by definition of $\Und\bA$, can equivalently be seen as a map $\un\to \bA_1$ in $\cV$ making the middle diagram above commute. 

Now consider the composite $c(\gamma, f\circ !)\colon X\to \bA_1$ in $\cV$, which in particular makes the diagram above right commute. Using the universal property of the internal tensor $(G\otimes X,\gamma)$ as stated in \cref{univpropcolim}, and the description in \cref{descrslice}, both for the case $Y=\un$, we see that the map $c(\gamma, f\circ !)$ is uniquely determined by $f$. Finally, by the pullback description of $\slice{\bA}{T}$ from \cref{defn:slice}, the map $c(\gamma, f\circ !)\colon X\to \bA_1$ defines a unique map $X\to (\slice{\bA}{T})_0$ whose post-composition with $(\slice{\bA}{T})_0\to \bA_0$ is precisely $G$. This shows the desired bijection. 
\end{proof}

\begin{ex} \label{ex:doublevsnterminal}
       In our examples, using the facts and terminology from \cref{ex:functors,ex:doubleterminal,ex:tensors} and the conservative families from \cref{ex:conservatives}, \cref{thm:enrvsintterminal} translates as follows. 
       \begin{rome}
           \item When $\cV=\Set$, it is tautological. Indeed, in this case, we have that $\Und$ is the identity functor, internal and $\cV$-terminal objects in $\CatV=\VCat$ agree by \cref{ex:doubleterminal}(1), and tensors with $\{*\}$ always exist by \cref{ex:tensors}(1).
           \item When $\cV=\Cat$, we retrieve a stricter analogue of the dual of \cite[Theorem 5.11]{cM2}. In our case, the theorem tells us that in the presence of cotabulators in a double category~$\bA$, a double terminal object in $\bA$ corresponds to a $2$-terminal object in the underlying (horizontal) $2$-category $\Und\bA$. 
           \item[($n$)] When $\cV=(n-1)\Cat$ for $n\geq 3$, the theorem tells us that in the presence of $n$-cotabulators in a double $(n-1)$-category $\bA$, a double $(n-1)$-terminal object in $\bA$ corresponds to an $n$-terminal object in the underlying $n$-category $\Und\bA$. 
       \end{rome}
\end{ex}

\subsection{Specializing to internal discrete fibrations}\label{subsec:Ctensors}

We now study internal terminal objects in an internal category living over a $\cV$-category. Let us fix a $\cV$-category $\cC$. We first introduce the notion of a $\cC$-internal tensor.

\begin{defn}
    Let $P\colon \bA\to \Int\cC$ be an internal discrete fibration. Given an object $X\in \cV$, we say that the internal category $\bA$ has all \textbf{$\cC$-internal tensors by $X$} if, for every object $A\in \cC$ and every map $G\colon X\to P^{-1}A$ in $\cV$, the internal colimit of the corresponding internal functor $G\colon \cst X\to \bA$ exists. 
\end{defn}

\begin{ex} \label{ex:C-tensors}
    In \cref{ex:tensors}, all internal tensors by objects in the given conservative families coincide with $\cC$-internal tensors, as internal functors out of the set $\{*\}$, the category $\mathbbm{2}$, or the $(n-1)$-category $C_{n-1}$ for $n\geq 3$ must factor through the fiber at some object $C\in \cC$. 
\end{ex}

We show that in this context, $\cC$-internal tensors are enough to imply that a $\cV$-terminal object in the underlying $\cV$-category $\Und\bA$ is also an internal terminal object in $\bA$.

\begin{theorem}\label{thm:ctensorsterminal}
    Let $\cG$ be a conservative family of objects in $\cV$, and $P\colon \bA\to \Int\cC$ be an internal discrete fibration such that $\bA$ has all $\cC$-internal tensors by objects $X\in \cG$. Given a pair~$(C,x)$ of an object $C\in \cC$ and an element $x\colon \un\to P^{-1}C$, the following are equivalent: 
    \begin{rome}
        \item the pair $(C,x)$ is an internal terminal object in~$\bA$,
        \item the pair $(C,x)$ is a $\cV$-terminal object in the underlying $\cV$-category $\Und\bA$.
    \end{rome}
\end{theorem}

\begin{proof}
The fact that (i) implies (ii) is \cref{internalterminalimpliesVterminal}. We show that (ii) implies (i). Suppose that $(C,x)$ is $\cV$-terminal in $\Und\bA$; we need to show that the internal functor $\slice{\bA}{(C,x)}\to \bA$ is an isomorphism. Since $P$ is an internal discrete fibration, as is the projection $\slice{\bA}{(C,x)}\to \bA$ by \cref{prop:sliceisdiscfib}, we also have that the composite $\slice{\bA}{(C,x)}\to \bA\xrightarrow{P} \Int\cC$ is an internal discrete fibration. It then follows from \cref{prop:isoonfibers} that this is equivalent to showing that the induced maps between fibers are isomorphisms in $\cV$. 

Given an object $A\in \cC$, recall that we denote by $P^{-1}A$ the fiber of $P$ at $A$. Then, by \cref{defn:slice}, level $0$ of $\slice{\bA}{(C,x)}$ is given by the pullback in $\cV$ as below left, which can be computed using \cref{subsec:discfiboverC} as the pullback in $\cV$ as below right.
    \begin{tz}
\node[](5) {$(\slice{\bA}{(C,x)})_0$}; 
\node[below of=5](6) {$\un$};
\node[right of=5, xshift=1.2cm](1) {$\bA_1$}; 
\node[below of=1](3) {$\bA_0$};
\pullback{5};

\draw[->] (5) to (1);
\draw[->] (5) to (6);
\draw[->] (6) to node[below,la]{$(C,x)$} (3);
\draw[->] (1) to node[right,la]{$t$} (3);

\node[right of=1,xshift=2cm](5) {$\bigsqcup_{A\in \cC} \cC(A,C)$}; 
\node[below of=5](6) {$\un$};
\node[right of=5, xshift=3.2cm](1) {$\bigsqcup_{A\in \cC} \cC(A,C)\times P^{-1} C$}; 
\node[below of=1](3) {$P^{-1}C$};
\pullback{5};

\draw[->] (5) to (1);
\draw[->] (5) to (6);
\draw[->] (6) to node[below,la]{$x$} (3);
\draw[->] (1) to node[right,la]{$\bigsqcup_! \pi_1$} (3);
\end{tz}
    This shows that the fiber of the composite $\slice{\bA}{(C,x)}\to\bA\xrightarrow{P} \Int\cC$ at an object $A\in \cC$ is given by the hom-object $\cC(A,C)$. Hence, we wish to show that the induced map between fibers 
\[ \cC(A,C)\xrightarrow{\id_{\cC(A,C)}\times x} \cC(A,C)\times P^{-1}C\xrightarrow{\ev^P_{A,C}} P^{-1}A \]
is an isomorphism in $\cV$. By definition of the conservative family $\cG$, this is equivalent to showing that, for every object $X\in \cG$, the induced map 
\[ \cV(X,\cC(A,C))\to \cV(X,P^{-1}A) \]
is an isomorphism of sets. 

Consider an element $G$ in the right-hand set; i.e., a map $G\colon X\to P^{-1} A$ in $\cV$. Since $\bA$ has all $\cC$-internal tensors by objects in $\cG$, there is an internal colimit $(G\otimes X,\gamma)$ for the corresponding internal functor $G\colon \cst X\to \bA$. As in the proof of \cref{thm:enrvsintterminal}, we obtain a unique map $X\to (\slice{\bA}{(C,x)})_0$ whose post-composition with $(\slice{\bA}{(C,x)})_0\to \bA_0$ is precisely~$G$. Now, note that the following diagram in $\cV$ commutes
\begin{tz}
    \node[](0) {$X$};
    \node[right of=0,xshift=1.3cm](1) {$(\slice{\bA}{(C,x)})_0$}; 
    \node[right of=1,xshift=1.3cm](2) {$\bA_0$}; 
    \node[below of=2](3) {$\bigsqcup_{A\in \cC} \un$}; 
    \node[below of=0](1') {$\un$};

    \draw[->,bend left=20] (0) to node[above,la]{$G$} (2);
    \draw[->] (0) to (1);
    \draw[->] (1) to (2);
    \draw[->] (0) to node[left,la]{$!$} (1');
    \draw[->] (1') to node[below,la]{$\iota_A$} (3);
    \draw[->] (2) to node[right,la]{$P_0$} (3);
\end{tz}
and so, by the universal property of the fiber of $\slice{\bA}{(C,x)}\to \Int\cC$ at $A$, the map $X\to (\slice{\bA}{(C,x)})_0$ induces a unique map $X\to \cC(A,C)$ into the fiber whose post-composition with $\cC(A,C)\to P^{-1}A$ is $G$. This shows the desired bijection.
\end{proof}

\subsection{The representation theorem in the presence of \texorpdfstring{$\cV$}{V}-tensors}\label{subsec:reptensors}

We now study the special case of the internal discrete fibration $\int_\cC F\to \Int\cC$ for a $\cV$-functor $F\colon \cC^\op\to \cV$. In particular, we show that the internal category $\int_\cC F$ admits $\cC$-internal tensors when $\cC$ has tensors in the enriched sense and $F$ preserves them. Let us first recall these enriched notions; see \cite[\S 3.7]{Kelly} for more details.

\begin{defn} \label{def:Vtensors}
    Let $\cC$ be a $\cV$-category and $X\in \cV$ be an object. We say that $\cC$ has all \textbf{$\cV$-tensors by $X$} if, for every object $C\in \cC$, there is an object $C\otimes X\in \cC$ together with a map $\gamma\colon X\to \cC(C,C\otimes X)$ in $\cV$ such that, for every object $A\in \cC$, the induced map 
    \[ [X,\cC(C,\gamma)]\colon \cC(C\otimes X,A)\to [X,\cC(C,A)] \]
    is an isomorphism in $\cV$. 
\end{defn}

Using the Yoneda lemma, this can be expressed by the following universal property. 

\begin{lemma} \label{univproptensor}
    Let $\cC$ be a $\cV$-category, and $C\in \cC$ and $X\in \cV$ be objects. A pair $(C\otimes X,\gamma)$ of an object $C\otimes X\in \cC$ and a map $\gamma\colon X\to \cC(C,C\otimes X)$ in $\cV$ is a $\cV$-tensor if and only if, for all objects $Y\in \cV$ and $A\in \cC$ and every map $\alpha\colon X\times Y\to \cC(C,A)$ in $\cV$, there is a unique map $g\colon Y\to \cC(C\otimes X,A)$ in $\cV$ making the following diagram commute.
    \begin{tz}
        \node[](1) {$X\times Y$}; 
        \node[right of=1,xshift=3.1cm](2) {$\cC(C,C\otimes X)\times \cC(C\otimes X,A)$}; 
        \node[below of=2](3) {$\cC(C,A)$}; 

        \draw[->] (1) to node[above,la]{$\gamma\times g$} (2); 
        \draw[->] (2) to node[right,la] {$c_{C,C\otimes X,A}$} (3); 
        \draw[->] (1) to node[left,la,yshift=-6pt]{$\alpha$} (3);
    \end{tz}
\end{lemma}

\begin{proof}
    By definition, the pair $(C\otimes X,\gamma)$ is a $\cV$-tensor if, for every object $A\in \cC$, the induced map 
    \[ [X,\cC(C,\gamma)]\colon \cC(C\otimes X,A)\to [X,\cC(C,A)] \]
    is an isomorphism in $\cV$. By the Yoneda lemma, this is equivalent to saying that, for all objects $Y\in \cV$ and $A\in \cC$, the induced map 
    \[ \cV(Y,[X,\cC(C,\gamma)])\colon \cV(Y,\cC(C\otimes X,A))\to \cV(Y,[X,\cC(C,A)])\cong \cV(X\times Y,\cC(C,A)) \]
    is an isomorphism of sets. 
    
    By direct inspection, an element in the left-hand set is a map $f\colon Y\to \cC(C\otimes X,A)$, an element in the right-hand set is a map $X\times Y\to \cC(C,A)$, and the assignment is given by sending each $f\colon Y\to \cC(C\otimes X,A)$ to the composite of $\gamma$ and $f$. This gives the desired result. 
\end{proof}

\begin{defn}
    Let $F\colon \cC^{\op}\to \cV$ be a $\cV$-functor and $X\in \cV$ be an object. Suppose that $\cC$ has all $\cV$-tensors by $X$. We say that $F\colon \cC^{\op}\to \cV$ \textbf{preserves $\cV$-tensors by $X$} if, for all objects~$C\in \cC$, there is an isomorphism $\varphi\colon [X,FC]\cong F(C\otimes X)$ in $\cV$ making the following diagram commute. 
    \begin{diagram} \label{fig:Fprestensor}
        \node[](1) {$X\times [X,FC]$}; 
        \node[right of=1,xshift=3.5cm](2) {$\cC(C,C\otimes X)\times F(C\otimes X)$}; 
        \node[below of=2](3) {$FC$}; 

        \draw[->] (1) to node[above,la]{$\gamma\times \varphi$}  (2); 
        \draw[->] (2) to node[right,la] {$\ev^F_{C,C\otimes X}$} (3); 
        \draw[->] (1) to node[left,la,yshift=-6pt]{$\ev_X$} (3);
    \end{diagram}
\end{defn}

In particular, the pair $(F(C\otimes X), \ev^F_{C,C\otimes X}(\gamma\times \id_{F(C\otimes X)}))$ satisfies the universal property of the \emph{$\cV$-cotensor} in $\cV$, i.e., the dual version of \cref{def:Vtensors}, and so we get the following. 

\begin{prop} \label{univpropFprestensor}
    Let $F\colon \cC^{\op}\to \cV$ be a $\cV$-functor and $X\in \cV$ be an object. Suppose that $\cC$ has all $\cV$-tensors by $X$ and that $F$ preserves them. Then, for every object $C\in \cC$ and every map $\alpha\colon X\times Y\to FC$ in~$\cV$, there is a unique map $y\colon Y\to F(C\otimes X)$ in $\cV$ making the following diagram commute.  
    \begin{tz} 
        \node[](1) {$X\times Y$}; 
        \node[right of=1,xshift=3cm](2) {$\cC(C,C\otimes X)\times F(C\otimes X)$}; 
        \node[below of=2](3) {$FC$}; 

        \draw[->] (1) to node[above,la]{$\gamma\times y$} (2); 
        \draw[->] (2) to node[right,la] {$\ev^F_{C,C\otimes X}$} (3); 
        \draw[->] (1) to node[left,la,yshift=-6pt]{$\alpha$} (3);
    \end{tz}
\end{prop}

\begin{proof}
    Since $\cV$ is cartesian closed, the universal property of the $\cV$-cotensor in $\cV$ can be expressed by only an isomorphism of hom-sets in $\cV$ and so amounts to requiring that the map
    \[ (\ev^F_{C,C\otimes X})_*((-)\times \gamma)\colon \cV(Y,F(C\otimes X))\to \cV(X\times Y,FC) \]
    is an isomorphism of sets, which gives the desired result.
\end{proof}

Now that we have recalled the notions of (co)tensors in the enriched setting, our goal is to identify $\cC$-internal tensors by an object $X\in \cV$ in the internal category of elements $\int_\cC F$ in the case where $\cC$ has all $\cV$-tensors by $X$ and $F\colon \cC^{\op}\to \cV$ preserves them. That is, for each object $X\in\cV$ and each map of the form $G\colon X\to FC$ in $\cV$, we wish to prove the existence of an internal colimit of the corresponding internal functor $G\colon \cst X\to\int_\cC F$. As we know, this amounts to finding an internal initial object in the internal category $G\downarrow\Delta$ of cocones under $G$, and so we begin by identifying an object in this internal category, which we will eventually show is the desired initial object.

\begin{rem}\label{rem:theinternalcolimitobject}
    Let $F\colon \cC^{\op}\to \cV$ be a $\cV$-functor, and $C\in \cC$ and $X\in \cV$ be objects. Let $G\colon X\to FC$ be a map in $\cV$, which can equivalently be regarded as a map $G\colon \un\to [X,FC]$. Suppose that the $\cV$-tensor $(C\otimes X,\gamma)$ exists and that $F$ preserves it. Then the pair $(\varphi G,\gamma \times \varphi G)$ defines an object in $G\downarrow \Delta$, where $\varphi$ denotes the isomorphism $\varphi\colon [X,FC]\cong F(C\otimes X)$. Indeed, by pre-composing with $\id_X\times G\colon X\cong X\times \un\to X\times [X,FC]$ in the commutative diagram \eqref{fig:Fprestensor}, we see that the following diagram in $\cV$ commutes,
    \begin{diagram} \label{phiGvsG}
        \node[](1) {$X$}; 
        \node[right of=1,xshift=2.6cm](2) {$\cC(C,C\otimes X)\times F(C\otimes X)$}; 
        \node[below of=2](3) {$FC$}; 

        \draw[->] (1) to node[above,la]{$\gamma\times \varphi G$} (2); 
        \draw[->] (2) to node[right,la] {$\ev^F_{C,C\otimes X}$} (3); 
        \draw[->] (1) to node[left,la,yshift=-6pt]{$G$} (3);
    \end{diagram}
    and hence it satisfies the description in \cref{elementsincatofcones}.
\end{rem}

As explained in \cref{univpropcolim}, the universal property of the internal colimit of an internal functor $G\colon\cst X\to\int_\cC F$ requires an understanding of the maps of the form $Y\to (G\downarrow \Delta)_0$, so in the following lemma we pursue an explicit description of these maps for $G\colon X\to FC$.

\begin{lemma} \label{lem:GdownDeltaobj}
    Let $F\colon \cC^{\op}\to \cV$ be a $\cV$-functor, $C\in \cC$ and $X\in \cV$ be objects, and $G\colon X\to FC$ be a map in $\cV$. Given an object $Y\in \cV$, a map $Y\to (G\downarrow \Delta)_0$ in~$\cV$ consists of families of maps $\{ y_A\colon Y_A\to FA \}_{A\in \cC}$ and $\{\beta_A\colon X\times Y_A\to \cC(C,A)\}_{A\in \cC}$ in~$\cV$ such that $\bigsqcup_{A\in \cC} Y_A\cong Y$ and, for every object $A\in \cC$, the following diagram in $\cV$ commutes.
    \begin{diagram}  \label{propbeta}
        \node[](1) {$X\times Y_A$}; 
        \node[below of=1](1') {$X$};
        \node[right of=1,xshift=2.2cm](2) {$\cC(C,A)\times FA$}; 
        \node[below of=2](3) {$FC$}; 

        \draw[->] (1) to node[above,la]{$(\beta_A,y_A\pi_1)$} (2); 
        \draw[->] (1) to node[left,la]{$\pi_0$} (1');
        \draw[->] (2) to node[right,la] {$\ev^F_{C,A}$} (3); 
        \draw[->] (1') to node[below,la]{$G$} (3);
    \end{diagram}
\end{lemma}

\begin{proof}
    Recall from \cref{descrslice} that a map $Y\to (G\downarrow \Delta)_0$ consists of maps $y\colon Y\to (\int_\cC F)_0$ and $\alpha\colon X\times Y\to (\int_\cC F)_1$ in $\cV$ such that $(s,t)\alpha=G\times y$. In particular, the following diagram in $\cV$ commutes,
    \begin{diagram} \label{propalpha} 
        \node[](1) {$X\times Y$}; 
        \node[below of=1](1') {$X$};
        \node[right of=1,xshift=2.5cm](2) {$\bigsqcup_{A\in \cC} \cC(C,A)\times FA$}; 
        \node[below of=2](3) {$FC$}; 
        \node[right of=2,xshift=2.5cm](4) {$(\int_\cC F)_1$}; 
        \node[below of=4](5) {$(\int_\cC F)_0$}; 
        \pullback{2};

        \draw[->,bend left=20] (1) to node[above,la]{$\alpha$} (4); 
        \draw[dashed,->] (1) to (2); 
        \draw[->] (1) to node[left,la]{$\pi_0$} (1');
        \draw[->] (2) to node[left,la] {$\ev^F_{C,A}$} (3);
        \draw[->] (4) to node[right,la] {$s$} (5); 
        \draw[->] (1') to node[below,la]{$G$} (3);
        \draw[->] (2) to node[above,la]{$\iota_C$} (4);
        \draw[->] (3) to node[below,la]{$\iota_C$} (5);
    \end{diagram}
    and so $\alpha$ factors through a unique map $\alpha\colon X\times Y\to \bigsqcup_{A\in \cC} \cC(C,A)\times FA$ by the universal property of the pullback. 
    
    Next, the map $y\colon Y\to (\int_\cC F)_0$ can be considered as a map in $\slice{\cV}{\bigsqcup_{A\in \cC} \un}$, which by extensivity corresponds to a unique family of maps in $\cV$
    \[ \{ y_A\colon Y_A\to FA \}_{A\in \cC},  \]
    where $Y_A$ denotes the fiber at an object $A\in \cC$ of the composite $Y\xrightarrow{y} (\int_\cC F)_0\xrightarrow{(\pi_F)_0} (\Int\cC)_0$. In particular, by \cref{rem:extensive} we have $\bigsqcup_{A\in \cC} Y_A\cong Y$.
    
    Finally, by definition of $(y,\alpha)$, we also have that the following diagram in $\cV$ commutes,
    \begin{tz}
        \node[](1) {$\bigsqcup_{A\in \cC} X\times Y_A$}; 
        \node[below of=1](1') {$\bigsqcup_{A\in \cC} Y_A$};
        \node[right of=1,xshift=3cm](2) {$\bigsqcup_{A\in \cC}\cC(C,A)\times FA$}; 
        \node[below of=2](3) {$\bigsqcup_{A\in \cC} FA$}; 

        \draw[->] (1) to node[above,la]{$\alpha$} (2); 
        \draw[->] (1) to node[left,la]{$\bigsqcup_{A\in \cC} \pi_1$} (1');
        \draw[->] (2) to node[right,la] {$\bigsqcup_{A\in \cC} \pi_1$} (3); 
        \draw[->] (1') to node[below,la]{$y=\bigsqcup_{A\in \cC} y_A$} (3);
    \end{tz}
    and so by the universal property of the products $\cC(C,A)\times FA$, and by fully faithfulness of the functor $\bigsqcup_{A\in \cC}$ given by extensivity, the map $\alpha$ corresponds to a unique family of maps 
    \[ \{ (\beta_A, y_A\pi_1)\colon X\times Y_A\to \cC(C,A)\times FA \}_{A\in \cC} \]
    where $\beta_A$ is a map $X\times Y_A\to \cC(C,A)\times FA$. The commutativity of the square \eqref{propbeta} can then be extracted from the commutativity of the left-hand square in \eqref{propalpha}. 
\end{proof}

We can now show that the object $(\varphi G,\gamma \times \varphi G)$ of $G\downarrow \Delta$ described in \cref{rem:theinternalcolimitobject} is precisely the internal colimit of the internal functor corresponding to $G\colon X\to FC$.

\begin{prop} \label{colimofG}
    Let $\cC$ be a $\cV$-category that admits all $\cV$-tensors by an object $X\in \cV$, and $F\colon \cC^{\op}\to \cV$ be a $\cV$-functor that preserves them. Let $C\in \cC$ be an object and $G\colon X\to FC$ be a map in $\cV$. Then the object $(\varphi G,\gamma \times \varphi G)$ of $G\downarrow \Delta$ is the internal colimit of the corresponding internal functor $G\colon \cst X\to \int_\cC F$.
\end{prop}

\begin{proof}
    By \cref{univpropcolim}, we need to show that for every map $(y,\alpha)\colon Y\to (G\downarrow \Delta)_0$ in~$\cV$, there is a unique map $f\colon Y\to (G\downarrow\Delta)_1$ from $(\varphi G,\gamma\times \varphi G)$ to $(y,\alpha)$. 
    
    Consider a map $(y,\alpha)\colon Y\to (G\downarrow\Delta)_0$ in $\cV$. By \cref{lem:GdownDeltaobj}, the pair $(y,\alpha)$ consists of families $\{y_A\colon Y_A\to FA\}_{A\in \cC}$ and  $\{\beta_A\colon X\times Y_A\to \cC(C,A)\}_{A\in \cC}$ of maps in $\cV$ such that $\bigsqcup_{A\in \cC} Y_A\cong Y$ and the diagram \eqref{propbeta} commutes. By the universal property of the $\cV$-tensor $(C\otimes X,\gamma)$ unpacked in \cref{univproptensor}, the family $\{\beta_A\}_{A\in \cC}$ determines a unique family $\{g_A\colon Y_A\to \cC(C\otimes X,A)\}_{A\in \cC}$ of maps in $\cV$ such that, for every object $A\in \cC$, the following diagram in $\cV$ commutes. 
    \begin{diagram} \label{propg}
        \node[](1) {$X\times Y_A$}; 
        \node[right of=1,xshift=3.2cm](2) {$\cC(C,C\otimes X)\times \cC(C\otimes X,A)$}; 
        \node[below of=2](3) {$\cC(C,A)$}; 

        \draw[->] (1) to node[above,la]{$\gamma\times g_A$} (2); 
        \draw[->] (2) to node[right,la] {$c_{C,C\otimes X,A}$} (3); 
        \draw[->] (1) to node[left,la,yshift=-6pt]{$\beta_A$} (3);
    \end{diagram}
    Putting together the commutative diagrams \eqref{propbeta} and \eqref{propg}, we get that the pair $(y,\alpha)$ determines unique families $\{y_A\colon Y_A\to FA\}_{A\in \cC}$ and $\{g_A\colon Y_A\to \cC(C\otimes X,A)\}_{A\in \cC}$ of maps in $\cV$ making the following diagram commute. 
\begin{diagram} \label{condgy}
        \node[](1) {$X\times Y_A$}; 
        \node[below of=1,yshift=-1.5cm](1') {$X$};
        \node[right of=1,xshift=4.8cm](2) {$\cC(C,C\otimes X)\times \cC(C\otimes X,A)\times FA$};
        \node[below of=2](2') {$\cC(C,C\otimes X)\times F(C\otimes X)$};
        \node[below of=2'](3) {$FC$}; 

        \draw[->] (1) to node[above,la]{$(\gamma \pi_0,g_A\pi_1,y_A\pi_1)$} (2); 
        \draw[->] (1) to node[left,la]{$\pi_0$} (1');
        \draw[->] (2) to node[right,la] {$\id_{\cC(C,C\otimes X)}\times \ev^F_{C\otimes X,A}$} (2');
        \draw[->] (2') to node[right,la] {$\ev^F_{C,C\otimes X}$} (3); 
        \draw[->] (1') to node[below,la]{$G$} (3);
    \end{diagram}
    Here we used the relation from \cref{lem:VfunctortoV}(2) to deduce the commutativity of the above diagram. 
    
    We define $f\coloneqq \bigsqcup_{A\in \cC} (g_A,y_A)\colon \bigsqcup_{A\in \cC}Y_A\to \bigsqcup_{A\in \cC}\cC(C\otimes X,A)\times FA$. It remains to show that the corresponding map $f\colon Y\to (\int_\cC F)_1$ gives a map $Y\to (G\downarrow \Delta)_1$ from $(\varphi G,\gamma\times \varphi G)$ to $(y,\alpha)$. By \cref{descrslice} (ii), this amounts to showing that the following diagrams in $\cV$ commute.
    \begin{diagram} \label{twodiagram}
    \node[](1) {$Y$}; 
        \node[right of=1,xshift=1cm](2) {$(\int_\cC F)_1$}; 
        \node[below of=2](3) {$(\int_\cC F)_0\times (\int_\cC F)_0$}; 

        \draw[->] (1) to node[above,la]{$f$} (2); 
        \draw[->] (2) to node[right,la] {$(s,t)$} (3); 
        \draw[->] (1) to node[left,la,yshift=-6pt]{$(\varphi G,y)$} (3);
        
        \node[right of=2,xshift=1.6cm](1) {$X\times Y$}; 
        \node[right of=1,xshift=3.2cm](2) {$(\int_\cC F)_1\times_{(\int_\cC F)_0}(\int_\cC F)_1$}; 
        \node[below of=2](3) {$(\int_\cC F)_1$}; 

        \draw[->] (1) to node[above,la]{$(\gamma\times \varphi G,f\pi_1)$} (2); 
        \draw[->] (2) to node[right,la] {$c$} (3); 
        \draw[->] (1) to node[left,la,yshift=-6pt]{$\alpha$} (3);
    \end{diagram}
    For the left-hand diagram in \eqref{twodiagram}, it is clear that $tf=y$ by definition of $f$; hence, we only need to show that $sf=\varphi G$. This amounts to showing that the following diagram in $\cV$ commutes, for every object $A\in \cC$.
    \begin{tz}
        \node[](1) {$Y_A$}; 
        \node[below of=1](2) {$\un$}; 
        \node[right of=1,xshift=4.2cm](1') {$\cC(C\otimes X,A)\times FA$};
        \node[right of=2,xshift=1cm](0) {$[X,FC]$}; 
        \node[below of=1'](2') {$F(C\otimes X)$}; 
        \draw[->] (1) to node[left,la]{$!$} (2);
        \draw[->] (1') to node[right,la]{$\ev^F_{C\otimes X,A}$} (2');
        \draw[->] (1) to node[above,la]{$(g_A,y_A)$} (1');
        \draw[->] (2) to node[below,la]{$G$} (0);
        \draw[->] (0) to node[above,la]{$\cong$} node[below,la]{$\varphi$} (2');
    \end{tz}
    However, by the commutativity of the diagrams \eqref{phiGvsG} and \eqref{condgy} both composites $\varphi G !$ and $\ev^F_{C\otimes X,A}(g_A,y_A)$ are maps $h\colon Y_A\to F(C\otimes X)$ in $\cV$ satisfying $\ev^F_{C,C\otimes X}(\gamma\times h)=G\pi_0$. Hence, it follows from \cref{univpropFprestensor} that they must be equal. The commutativity of the right-hand diagram in \eqref{twodiagram} then simply follows from the commutativity of the diagrams \eqref{propg}, for every object $A\in \cC$. To conclude, the unicity of such a map $f$ is straightforward from the unicity of the family $\{g_A\colon Y_A\to \cC(C\otimes X,A)\}_{A\in \cC}$.
\end{proof}

As a consequence, we obtain the following.

\begin{prop}\label{prop:grothtensors}
    Let $\cG$ be a conservative family of objects in $\cV$, $\cC$ be a $\cV$-category that admits all $\cV$-tensors by objects $X\in \cG$, and $F\colon \cC^{\op}\to \cV$ be a $\cV$-functor that preserves them. Then the internal category of elements $\int_\cC F$ has all $\cC$-internal tensors by objects $X\in \cG$.
\end{prop}

\begin{proof}
    This follows directly from \cref{colimofG}.
\end{proof}

We can finally use what we learned to produce a representation theorem involving only the underlying $\cV$-category of elements.

\begin{theorem}\label{thm:enrichedrepthm}
    Let $\cG$ be a conservative family of objects in $\cV$, $\cC$ be a $\cV$-category that admits all $\cV$-tensors by objects $X\in \cG$, and $F\colon \cC^{\op}\to \cV$ be a $\cV$-functor that preserves them. Given a pair $(C,x)$ of an object $C\in \cC$ and an element $x\colon \un\to FC$ in $\cV$, the following are equivalent: 
\begin{rome}
    \item the $\cV$-functor $F$ is $\cV$-representable by $(C,x)$, 
    \item the pair $(C,x)$ is a $\cV$-terminal object of the underlying $\cV$-category of elements $\Und\int_\cC F$. 
\end{rome}
\end{theorem}

\begin{proof}
    This is obtained by combining \cref{thm:representation,thm:ctensorsterminal,prop:grothtensors}.
\end{proof}

\begin{ex} \label{ex:repthmtensors} 
    In our examples, using the 
 interpretation and terminology from \cref{ex:tensors,ex:C-tensors} and the conservative families from \cref{ex:conservatives}, \cref{thm:enrichedrepthm} translates as follows. 
    \begin{enumerate}
        \item When $\cV=\Set$, it does not tell us anything new, for the same reasons as in \cref{ex:doublevsnterminal}(1). 
        \item When $\cV=\Cat$, we retrieve a stricter analogue of \cite[Theorem 6.15]{cM2}. In our case, given a $2$-functor $F\colon \cC^{\op}\to \Cat$, the $2$-category of elements $\Und\int_\cC F$ has as objects pairs $(A,x)$ of objects $A\in \cC$ and $x\in FC$, as morphisms $(A,x)\to (B,y)$ morphisms $f\colon A\to B$ in~$\cC$ such that $Ff(y)=x$, and as $2$-morphisms $\alpha\colon f\Rightarrow g\colon (A,x)\to (B,y)$ $2$-morphisms $\alpha\colon f\Rightarrow g$ in $\cC$ such that $(F\alpha)_y=\id_x$. Then, if $\cC$ has $2$-tensors by $\mathbbm{2}$ and $F$ preserves them, the theorem tells us that to test whether $F$ is $2$-representable by a pair $(C,x)$ it is enough to check whether $(C,x)$ is $2$-terminal in the above described $2$-category. 
        \item[($n$)] When $\cV=(n-1)\Cat$ for $n\geq 3$, given an $n$-functor $F\colon \cC^{\op}\to (n-1)\Cat$, the $n$-category $\Und\int_\cC F$ has as objects pairs $(A,x)$ of objects $A\in \cC$ and $x\in FA$, and as $k$-morphisms with source object $(A,x)$ and target object $(B,y)$, $k$-morphisms $\varphi$ in $\cC$ with source object $A$ and target object $B$ such that $(F\varphi)_y$ is the identity $(n-1)$-morphism at $x$, for all $1\leq k\leq n$. Then, if $\cC$ has $n$-tensors by $C_{n-1}$ and $F$ preserves them, the theorem tells us that to test whether $F$ is $n$-representable by a pair $(C,x)$ it is enough to check whether $(C,x)$ is $n$-terminal in the above described $n$-category. 
    \end{enumerate}
\end{ex}

\section{Weighted enriched limits}

In this section, let $\cV$ be an extensive and cartesian closed category with pullbacks. We apply here our results to a specific type of $\cV$-functors: the ones of the form $\cV^\cI(W,\cC(-,G))\colon \cC^{\op}\to \cV$, for $\cV$-functors $G\colon \cI\to \cC$ and $W\colon \cI\to \cV$, which allow us to define the $W$-weighted $\cV$-limit of~$G$. 

After recalling the notion of weighted $\cV$-limits in \cref{subsec:defnweightedlims}, we use \cref{thm:representation,enrichedstatement} to obtain a representation theorem for weighted $\cV$-limits in \cref{subsec:repthmforweightedlims}. These characterize weighted $\cV$-limits in terms of internal terminal objects in the corresponding internal category of weighted $\cV$-cones, and in terms of $\cV$-terminal objects in shifted underlying $\cV$-categories. As a particular case, we obtain a description of conical $\cV$-limits as well. In \cref{sec:weightedlimittensors} we explore how we can apply \cref{thm:enrichedrepthm} in the case where $\cC$ has certain $\cV$-tensors, which are now automatically preserved by the relevant $\cV$-functor, to obtain a further characterization simply in terms of $\cV$-terminal objects in the $\cV$-category of weighted $\cV$-cones.

The main result of this section is \cref{thm:weightedasconical}, which proves that weighted $\cV$-limits can be computed as conical internal limits, by replacing the indexing $\cV$-category by the internal category of elements of the weight. This is the content of \cref{subsec:weightedlimsasconical}.

\subsection{Definition}\label{subsec:defnweightedlims}

We begin by recalling the necessary background.

\begin{defn}
    Let $G\colon \cI\to \cC$ be a $\cV$-functor. We define a $\cV$-functor $\cC(-,G)\colon \cC^{\op}\to \cV^\cI$ such that 
    \begin{rome}
        \item it sends an object $A\in \cC$ to the $\cV$-functor $\cC(A,G-)\colon \cI\to \cV$ given by the composite of $\cV$-functors 
        \[ \cI\xrightarrow{G}\cC\xrightarrow{\cC(A,-)} \cV, \]
        \item for objects $A,B\in\cC$, it is given on hom-objects by the map \[ \cC(A,B)\to \cV^\cI(\cC(B,G-),\cC(A,G-)) \]
        induced by the unique maps $\cC(A,B)\to [\cC(B,Gi),\cC(A,Gi)]$ in $\cV$ corresponding under the adjunction $(-)\times \cC(B,Gi)\dashv [\cC(B,Gi),-]$ to the composition maps of~$\cC$
        \[ c_{A,B,Gi}\colon \cC(A,B)\times \cC(B,Gi)\to \cC(A,Gi), \]
        for all objects $i\in \cI$.
    \end{rome}
    Compatibility with identities and compositions follows from the unitality and associativity of composition maps in $\cC$.
\end{defn}

\begin{defn}\label{defn:weightedlim}
    Let $G\colon \cI\to \cC$ and $W\colon \cI\to \cV$ be $\cV$-functors. A \textbf{$W$-weighted $\cV$-limit} of $G$ is a pair $(L,\lambda)$ of an object $L\in \cC$ and a $\cV$-natural transformation $\lambda\colon W\Rightarrow \cC(-,G)$ in $\VCat(\cI,\cV)$ such that the induced $\cV$-natural transformation 
    \[ \lambda^*\colon \cC(-,L)\Rightarrow \cV^\cI(W,\cC(-,G)) \]
    is an isomorphism in $\VCat(\cC^{\op},\cV)$. In other words, this says that the $\cV$-functor \[ \cV^\cI(W,\cC(-,G))\colon \cC^{\op}\to \cV \]
    is representable by $(L,\lambda)$. 
\end{defn}

In the special case where the weight $W\colon \cI\to \cV$ is the constant $\cV$-functor $\Delta \un$ at the terminal object of $\cV$, we obtain the following.

\begin{defn}
    Let $G\colon \cI\to \cC$ be a $\cV$-functor. A \textbf{(conical) $\cV$-limit} of $G$ is a pair $(L,\lambda)$ of an object $L\in \cC$ and a $\cV$-natural transformation $\lambda\colon \Delta\un\Rightarrow \cC(L,G-)$ in $\VCat(\cI,\cV)$---or equivalently, a $\cV$-natural transformation $\lambda\colon \Delta L\Rightarrow G$ in $\VCat(\cI,\cC)$--- such that the $\cV$-functor 
   \[ \cV^\cI(\Delta\un,\cC(-,G))\cong \cC^\cI(\Delta(-),G)\colon \cC^{\op}\to \cV \]
   is $\cV$-representable at $(L,\lambda)$, where $\Delta\colon \cC\to \cC^\cI$ denotes the diagonal $\cV$-functor induced by the unique $\cV$-functor $\cI\to \mathbbm{1}$. 
\end{defn}

\subsection{The representation theorem for weighted \texorpdfstring{$\cV$}{V}-limits}\label{subsec:repthmforweightedlims} 

In order to obtain an explicit formulation of \cref{thm:representation,enrichedstatement} in the case of weighted $\cV$-limits, we first identify the internal category of elements of the $\cV$-functor $\cV^\cI(W,\cC(-,G))\colon \cC^{\op}\to \cV$.

\begin{notation}
   Given $\cV$-functors $G\colon \cI\to \cC$ and $W\colon \cI\to \cV$, we denote by \[ W\downarrow \Int\cC(-,G)\coloneqq \Int\cC(-,G)^{\op}\downarrow W \]
   the comma internal category from \cref{defn:internalcomma} in the case where $H$ is the internal functor $\Int\cC(-,G)^{\op}\colon \Int\cC\to \Int(\cV^\cI)^{\op}$ and $B$ is the object $W\in \cV^\cI$.
\end{notation}

The following result is the main ingredient in our characterization of weighted $\cV$-limits.

\begin{prop} \label{GCofVweightedlimit}
    Let $G\colon\cI\to\cC$ and $W\colon \cI\to \cV$ be $\cV$-functors. Then there is an isomorphism in~$\slice{\CatV}{\Int\cC}$.
\begin{tz}
\node[](1) {$\int_\cC \cV^\cI (W,\cC(-,G))$}; 
\node[below of=1,xshift=2.7cm,yshift=.2cm](3) {$\Int\cC$};
\node[above of=3,xshift=2.7cm,yshift=-.2cm](2) {$W\downarrow \Int\cC(-,G)$}; 

\draw[->] (1) to node[above,la]{$\cong$} (2);
\draw[->] (2) to (3);
\draw[->] (1) to (3);
\end{tz}
\end{prop}

\begin{proof}
We first construct an internal functor $\int_\cC \cV^\cI (W,\cC(-,G))\to W\downarrow \Int\cC(-,G)$ over $\Int\cC$. For this, consider the following commutative diagram in $\CatV$, in which the outer rectangle is the pullback defining $W\downarrow \Int\cC(-,G)$.
\begin{tz}
\node[](1') {$W\downarrow \Int\cC(-,G)$}; 
\node[below of=1'](3') {$\Int\cC$};
\node[right of=1',xshift=2.6cm](1) {$\slice{\Int(\cV^\cI)^{\op}}{W}$}; 
\node[below of=1](3) {$\Int(\cV^\cI)^{\op}$};
\node[right of=1,xshift=3.5cm](2) {$\llbracket\mathbbm{2}, \Int(\cV^\cI)^{\op}\rrbracket$}; 
\node[below of=2](4) {$\Int(\cV^\cI)^{\op}\times \Int(\cV^\cI)^{\op}$}; 

\pullback{1};

\draw[->,dashed] (1') to (1);
\draw[->,bend left=15] (1') to (2); 
\draw[->] (3') to node[below,la]{$\Int\cC(-,G)^{\op}$} (3);
\draw[->] (1) to (2);
\draw[->] (2) to node[right,la]{$\langle 0,1\rangle^*$} (4);
\draw[->] (1) to (3);
\draw[->] (1') to (3');
\draw[->] (3) to node[below,la]{$\id_{\Int(\cV^\cI)^{\op}}\times W$} (4);
\end{tz} 
By the universal property of the pullback, there is a unique dashed internal functor as depicted. Moreover, since the right and outer diagrams are pullbacks, by pullback cancellation we get that the square on the left is also a pullback. Recall that from \cref{constr:changeofbase} we have an internal functor $\int_{\cC(-,G)}\cV^\cI(W,-)\colon \int_\cC \cV^\cI (W,\cC(-,G))\to \int_{(\cV^\cI)^{\op}}\cV^\cI (W,-)$ living over $\Int\cC(-,G)$, and furthermore, that $\int_{(\cV^\cI)^{\op}}\cV^\cI (W,-)$ is isomorphic to the slice internal category $\slice{\Int(\cV^\cI)^{\op}}{W}$  over $\Int(\cV^\cI)^{\op}$ by \cref{cor:slicehom}. By the universal property of the pullback, this induces a unique internal functor $\int_\cC \cV^\cI (W,\cC(-,G))\to W\downarrow \Int\cC(-,G)$, as desired. 

Recall from \cref{prop:commaisdiscfib,grothdiscfib} that the internal functors $W\downarrow \Int\cC(-,G)\to \Int\cC$ and $\int_\cC \cV^\cI(W,\cC(-,G))\to \Int\cC$ are internal discrete fibrations. Hence, by \cref{prop:isoonfibers}, to show that $\int_\cC \cV^\cI(W,\cC(-,G))$ and $W\downarrow \Int\cC(-,G)$ are isomorphic over $\Int\cC$, it suffices to show that their fibers are isomorphic. 

Using the pullback above left and the fact that there is an isomorphism over $\Int(\cV^\cI)^{\op}$ 
\[ \textstyle \slice{\Int(\cV^\cI)^{\op}}{W}\cong\int_{(\cV^\cI)^{\op}}\cV^\cI (W,-), \]
we can directly compute the fiber of $W\downarrow \Int\cC(-,G)$ at an object $A\in \cC$. Indeed, this is given by $\cV^\cI(W,\cC(A,G-))$, which is precisely the fiber of $\int_\cC \cV^\cI(W,\cC(-,G))$ at the object $A\in \cC$.
\end{proof}

We can finally prove our claimed characterization.

\begin{theorem}\label{thm:repweightedlimits}
Let $G\colon \cI\to \cC$ and $W\colon \cI\to \cV$ be $\cV$-functors. Given a pair $(L,\lambda)$ of an object~$L\in \cC$ and a $\cV$-natural transformation $\lambda\colon W\Rightarrow \cC(L,G-)$ in $\VCat(\cI,\cV)$, the following are equivalent: 
\begin{rome}
    \item the pair $(L,\lambda)$ is a $W$-weighted $\cV$-limit of $G$, 
    \item the pair $(L,\lambda)$ is an internal terminal object in the internal category $W\downarrow \Int\cC(-,G)$, 
    \item for any object $X$ in a conservative family $\cG$ of objects in $\cV$, the pair $(L,\lambda)$ is a $\cV$-terminal object in the $\cV$-category $\Und\llbracket \cst X, W\downarrow \Int\cC(-,G)\rrbracket$.
\end{rome}
\end{theorem}

\begin{proof}
    This follows directly from \cref{thm:representation,enrichedstatement,GCofVweightedlimit}.
\end{proof}

In the case where the weight $W\colon \cI\to \cV$ is the constant $\cV$-functor $\Delta \un$ at the terminal object of $\cV$, \cref{thm:repweightedlimits} yields the following characterization of (conical) $\cV$-limits. 

\begin{notation}
    Given a $\cV$-functor $G\colon\cI\to \cC$, we denote by $\Int\Delta\downarrow G$ the comma internal category from \cref{defn:internalcomma} in the case of the internal functor $\Int\Delta\colon \Int\cC\to \Int(\cC^\cI)$ and the object $G\in \cC^{\cI}$.
\end{notation}

\begin{cor} \label{cor:repconicallimit}
Let $G\colon \cI\to \cC$ be a $\cV$-functor. Given a pair $(L,\lambda)$ of an object $L\in \cC$ and a $\cV$-natural transformation $\lambda\colon \Delta L\Rightarrow G$ in $\VCat(\cI,\cC)$, the following are equivalent: 
\begin{rome}
    \item the pair $(L,\lambda)$ is a $\cV$-limit of $G$, 
    \item the pair $(L,\lambda)$ is an internal terminal object in the internal category $\Int\Delta\downarrow G$,
    \item for any object $X$ in a conservative family $\cG$ of objects in $\cV$, the pair $(L,\lambda)$ is a $\cV$-terminal object in the $\cV$-category $\Und\llbracket \cst X, \Int\Delta\downarrow G\rrbracket$.
\end{rome}
\end{cor}

\begin{proof}
    This is \cref{thm:repweightedlimits} in the case where $W=\Delta \un$, using the isomorphism in $\CatV$
   \[ \Delta \un\downarrow \Int\cC(-,G)\cong \Int\Delta\downarrow G. \qedhere \]
\end{proof}

In certain cases, the $\cV$-categories involved in the above statements can be described using constructions that are made fully in the world of enriched categories. 

\begin{rem} \label{rem:shifted}
    Given an object $X\in \cV$, the functor $[X,-]\colon \cV\to \cV$ induces by base-change a functor $\mathrm{Ar}_X\colon \VCat\to \VCat$ which sends a $\cV$-category $\cC$ to the $\cV$-category $\mathrm{Ar}_X\cC$ with the same set of objects and with hom-objects $\mathrm{Ar}_X\cC(A,B)\coloneqq [X,\cC(A,B)]$, for all objects $A,B\in \cC$. 
    
    Under the assumptions that the internalization functor $\Int\colon \VCat\to \CatV$ is fully faithful and the functor $[X,-]\colon \cV\to \cV$ preserves coproducts, a straightforward computation shows that there is an equality of functors \[
    \mathrm{Ar}_X=\Und\llbracket \cst X,\Int(-)\rrbracket\colon \VCat\to \VCat. \]
    Then, if the internalization functor $\Int\colon \VCat\to \CatV$ further preserves binary products, the $\cV$-categories $\Und\llbracket \cst X,W\downarrow \Int\cC(-,G)\rrbracket$ and $\Und\llbracket \cst X,\Int\Delta\downarrow G\rrbracket$ can be computed as the following pullbacks in $\VCat$.
    \begin{tz}
\node[](5) {$W\downarrow^\cV \mathrm{Ar}_X\cC(-,G)$}; 
\node[below of=5](6) {$\mathrm{Ar}_X\cC$};
\node[right of=5, xshift=3.85cm](1) {$\mathrm{Ar}_X((\cV^{\cI})^{\op})^{\mathbbm 2}$}; 
\node[below of=1](3) {$\mathrm{Ar}_X(\cV^\cI)^{\op}\times \mathrm{Ar}_X(\cV^\cI)^{\op}$};
\pullback{5};

\draw[->] (5) to (1);
\draw[->] (5) to (6);
\draw[->] (6) to node[below,la,xshift=-2pt]{$\mathrm{Ar}_X\cC(-,G)^{\op}\times W$} (3);
\draw[->] (1) to node[right,la]{$\langle 0,1\rangle^*$} (3);

\node[right of=1,xshift=1.7cm](5) {$\mathrm{Ar}_X\Delta\downarrow^\cV G$}; 
\node[below of=5](6) {$\mathrm{Ar}_X\cC$};
\node[right of=5, xshift=2.35cm](1) {$\mathrm{Ar}_X(\cC^{\cI})^{\mathbbm{2}}$}; 
\node[below of=1](3) {$\mathrm{Ar}_X(\cC^\cI)\times \mathrm{Ar}_X(\cC^\cI)$};
\pullback{5};

\draw[->] (5) to (1);
\draw[->] (5) to (6);
\draw[->] (6) to node[below,la,xshift=-2pt]{$\mathrm{Ar}_X\Delta\times G$} (3);
\draw[->] (1) to node[right,la]{$\langle 0,1\rangle^*$} (3);
\end{tz}
This follows from the fact that, under the given hypotheses, the functors $\Und$ and $\llbracket \cst X,-\rrbracket$ preserve products, pullbacks, and $2$-cotensors by $\mathbbm{2}$. 
\end{rem}

\begin{ex} \label{ex:weightedlimit}
    In our examples, using the interpretations from \cref{ex:functors,ex:doubleterminal} and the conservative families from \cref{ex:conservatives}, \cref{thm:repweightedlimits,cor:repconicallimit} translate as follows.
    \begin{enumerate}
        \item When $\cV=\Set$, given functors $W\colon \cI\to \Set$ and $G\colon \cI\to \cC$, the theorem says that a $W$-weighted limit of $G$ coincides with a terminal object in the category $W\downarrow \cC(-,G)$ of $W$-weighted cones over $G$. 
        
        When $W=\Delta \{*\}$ is the constant functor at the one-point set, the corollary retrieves the well-known result that a limit of $G$ coincides with a terminal object in the category~$\Delta\downarrow G$ of cones over $G$. 
        \item When $\cV=\Cat$, we retrieve stricter analogues of  \cite[Theorem 7.19 and Corollary~7.22]{cM2}. In our case, the theorem says that given $2$-functors $W\colon \cI\to \Cat$ and $G\colon \cI\to \cC$, a $W$-weighted $2$-limit of $G$ corresponds to a double terminal object in the double category $W\downarrow \Int\cC(-,G)$ of $W$-weighted $2$-cones over $G$, or equivalently, using \cref{rem:shifted}, to a $2$-terminal object in the shifted $2$-category $W\downarrow^2 \mathrm{Ar}_{\mathbbm 2}\cC(-,G)$ whose objects are morphisms of $W$-weighted $2$-cones over $G$, also called \emph{modifications}. 
        
        When $W=\Delta \mathbbm 1$ is the constant $2$-functor at the terminal category, the corollary shows that a $2$-limit of $G$ corresponds to a double terminal object in the double category $\Int\Delta\downarrow G$ of $2$-cones over $G$, or equivalently, using \cref{rem:shifted}, to a $2$-terminal object in the shifted $2$-category $\mathrm{Ar}_{\mathbbm 2}\Delta\downarrow G$ whose objects are modifications of $2$-cones over $G$.
        \item[($n$)] When $\cV=(n-1)\Cat$ for $n\geq 3$, given $n$-functors $W\colon \cI\to (n-1)\Cat$ and $G\colon \cI\to \cC$, the theorem says that a $W$-weighted $n$-limit of $G$ corresponds to a double $(n-1)$-terminal object in the double $(n-1)$-category $W\downarrow\Int\cC(-,G)$ of $W$-weighted $n$-cones over $G$, or equivalently, using \cref{rem:shifted}, to an $n$-terminal object in the shifted $n$-category $W\downarrow^n \mathrm{Ar}_{C_{n-1}}\cC(-,G)$ whose objects are $(n-1)$-morphisms of $W$-weighted $n$-cones over~$G$. 
        
        When $W=\Delta \mathbbm 1$ is the constant $n$-functor at the terminal $(n-1)$-category, the corollary shows that an $n$-limit corresponds to a double $(n-1)$-terminal object in the double $(n-1)$-category $\Int\Delta \downarrow G$ of $n$-cones over $G$, or equivalently, using \cref{rem:shifted}, to an $n$-terminal object in the shifted $n$-category $ \mathrm{Ar}_{C_{n-1}}\Delta \downarrow^nG$ whose objects are $(n-1)$-morphisms of $n$-cones over $G$.
    \end{enumerate}
\end{ex}

\subsection{Weighted \texorpdfstring{$\cV$}{V}-limits in the presence of \texorpdfstring{$\cV$}{V}-tensors} \label{sec:weightedlimittensors}

We can also study what happens in the case where the $\cV$-category $\cC$ admits certain $\cV$-tensors. First, we show that the $\cV$-functor $\cV^\cI(W,\cC(-,G))$ preserves any $\cV$-tensors which are present in $\cC$.

\begin{prop} \label{lem:weightedconeprestensors}
    Let $\cC$ be a $\cV$-category that admits all $\cV$-tensors by an object $X\in \cV$, and $G\colon \cI\to \cC$ and $W\colon \cI\to \cV$ be $\cV$-functors. Then the $\cV$-functor
    $\cV^\cI(W,\cC(-,G))\colon \cC^{\op}\to \cV$ preserves all $\cV$-tensors by $X$. 
\end{prop}

\begin{proof}
    By \cite[(3.16)]{Kelly}, $\cV$-cotensors in the internal hom $\cV$-categories $\cV^\cI$ are computed levelwise in $\cV$. Hence, for every object $C\in \cC$, also using the universal property of the $\cV$-tensor $C\otimes X$ in $\cC$, we have the following isomorphisms in $\cV$
    \begin{align*}
        \cV^\cI(W,\cC(C\otimes X,G-)) \cong \cV^\cI(W,[X,\cC(C,G-)])\cong [X,\cV^\cI(W,\cC(C,G-))]
    \end{align*}
    which shows that $\cV^\cI(W,\cC(-,G))\colon \cC^{\op}\to \cV$ preserves $\cV$-tensors by $X$.
\end{proof}

The above result allows us to use \cref{thm:enrichedrepthm} to obtain a further characterization of weighted $\cV$-limits.

\begin{theorem} \label{thm:limitswithtensors}
    Let $\cG$ be a conservative family of objects in $\cV$, $\cC$ be a $\cV$-category that admits all $\cV$-tensors by objects $X\in \cG$, and $G\colon \cI\to \cC$ and $W\colon \cI\to \cV$ be $\cV$-functors. Given a pair~$(L,\lambda)$ of an object $L\in \cC$ and a $\cV$-natural transformation $\lambda\colon W\Rightarrow \cC(L,G-)$ in $\VCat(\cI,\cV)$, the following are equivalent: 
\begin{rome}
    \item the pair $(L,\lambda)$ is a $W$-weighted $\cV$-limit of $G$, 
    \item the pair $(L,\lambda)$ is a $\cV$-terminal object in the underlying $\cV$-category $\Und(W\downarrow \Int\cC(-,G))$.
\end{rome}
\end{theorem}

\begin{proof}
    By definition of a $W$-weighted $\cV$-limit from \cref{defn:weightedlim}, this is obtained by combining \cref{thm:enrichedrepthm,GCofVweightedlimit,lem:weightedconeprestensors}.
\end{proof}

Again specializing to the case where $W=\Delta \un$ is the constant $\cV$-functor at the terminal object, we obtain the following. 

\begin{cor} \label{cor:conicalwithtensors}
    Let $\cG$ be a conservative family of objects in $\cV$, $\cC$ be a $\cV$-category that admits all $\cV$-tensors by objects $X\in \cG$, and $G\colon \cI\to \cC$ be a $\cV$-functor. Given a pair $(L,\lambda)$ of an object~$L\in \cC$ and a $\cV$-natural transformation $\lambda\colon \Delta L\Rightarrow G$ in $\VCat(\cI,\cC)$, the following are equivalent: 
\begin{rome}
    \item the pair $(L,\lambda)$ is a $\cV$-limit of $G$, 
    \item the pair $(L,\lambda)$ is a $\cV$-terminal object in the underlying $\cV$-category $\Und(\Int\Delta\downarrow G)$.
\end{rome}
\end{cor}

\begin{rem} \label{rem:underofcones}
    Under the assumption that the internalization functor $\Int\colon \VCat\to \CatV$ is fully faithful and preserves binary products, the $\cV$-categories $\Und(W\downarrow \Int\cC(-,G))$ and $\Und(\Int\Delta\downarrow G)$ can be computed as the following pullbacks in $\VCat$,
    \begin{tz}
\node[](5) {$W\downarrow^\cV \cC(-,G)$}; 
\node[below of=5](6) {$\cC$};
\node[right of=5, xshift=2.5cm](1) {$((\cV^\cI)^{\op})^{\mathbbm{2}}$}; 
\node[below of=1](3) {$(\cV^\cI)^{\op}\times (\cV^\cI)^{\op}$};
\pullback{5};

\draw[->] (5) to (1);
\draw[->] (5) to (6);
\draw[->] (6) to node[below,la]{$\cC(-,G)^{\op}\times W$} (3);
\draw[->] (1) to node[right,la]{$\langle 0,1\rangle^*$} (3);

\node[right of=1,xshift=2cm](5) {$\Delta\downarrow^\cV G$}; 
\node[below of=5](6) {$\cC$};
\node[right of=5, xshift=1.4cm](1) {$(\cC^{\cI})^{\mathbbm{2}}$}; 
\node[below of=1](3) {$\cC^\cI\times \cC^\cI$};
\pullback{5};

\draw[->] (5) to (1);
\draw[->] (5) to (6);
\draw[->] (6) to node[below,la]{$\Delta\times G$} (3);
\draw[->] (1) to node[right,la]{$\langle 0,1\rangle^*$} (3);
\end{tz}
which we refer to as the \emph{$\cV$-categories of ($W$-weighted) $\cV$-cones over $G$}. This follows from the fact that, under the given hypotheses, the unit $\id_{\VCat}\cong \Und\Int$ is an isomorphism and the functor $\Und$ preserves products, pullbacks, and $2$-cotensors by $\mathbbm{2}$. 
\end{rem}

\begin{ex} \label{ex:weightedlimittensor}
    In our cases of interest, using the facts from \cref{ex:functors} and the conservative families from \cref{ex:conservatives}, \cref{thm:limitswithtensors,cor:conicalwithtensors} translate as follows.
    \begin{enumerate}
        \item When $\cV=\Set$, they do not tell us anything new, for the same reasons as in \cref{ex:doublevsnterminal}(1).
        \item When $\cV=\Cat$, we retrieve stricter analogues of \cite[Theorem 7.21 and Corollary~7.25]{cM2}. In our case, using \cref{rem:underofcones}, the theorem says that given $2$-functors $W\colon \cI\to \Cat$ and $G\colon \cI\to \cC$, if $\cC$ has all $2$-tensors by $\mathbbm 2$, a $W$-weighted $2$-limit of $G$ corresponds to a $2$-terminal object in the usual $2$-category $W\downarrow^2\cC(-,G)$ of $W$-weighted $2$-cones over $G$.
        
        When $W=\Delta\mathbbm{1}$ is the constant $2$-functor at the terminal category, using \cref{rem:underofcones}, the corollary shows that, if $\cC$ has all tensors by~$\mathbbm 2$, a $2$-limit of $G$ corresponds to a $2$-terminal object in the usual $2$-category $\Delta\downarrow^2 G$ of $2$-cones over $G$. 
        \item[($n$)] When $\cV=(n-1)\Cat$ for $n\geq 3$, using \cref{rem:underofcones}, the theorem says that given $n$-functors $W\colon \cI\to (n-1)\Cat$ and $G\colon \cI\to \cC$, if $\cC$ has all $n$-tensors by~$C_{n-1}$, a $W$-weighted $n$-limit of $G$ corresponds to an $n$-terminal object in the $n$-category $W\downarrow^n \cC(-,G)$ of $W$-weighted $n$-cones over $G$. 
        
        When $W=\Delta\mathbbm{1}$ is the constant $n$-functor at the terminal $(n-1)$-category, using \cref{rem:underofcones}, the corollary shows that if $\cC$ has all $n$-tensors by~$C_{n-1}$, an $n$-limit of $G$ corresponds to an $n$-terminal object in the $n$-category $\Delta\downarrow^n G$ of $n$-cones over $G$. 
    \end{enumerate}
\end{ex}

\subsection{Weighted \texorpdfstring{$\cV$}{V}-limits as conical internal limits}\label{subsec:weightedlimsasconical}

In order to express weighted $\cV$-limits in terms of internal limits, let us first define this latter notion, which is dual to the one from \cref{def:internalcolimit}. 

\begin{notation} \label{not:catofcones}
    Given an internal functor $G\colon \bI\to \bA$ to $\cV$, we denote by $\Delta\downarrow G$ the comma internal category from \cref{defn:internalcomma} in the case of the internal functor $\Delta\colon \bA\to \llbracket\bI,\bA\rrbracket$ and the element $G\colon \un\to \llbracket\bI,\bA\rrbracket$. We refer to $\Delta\downarrow G$ as the \textbf{internal category of cones over $G$}.
\end{notation}

\begin{defn}
    Let $G\colon\bI\to \bA$ be an internal functor. An \textbf{internal limit} of $G$ is an internal terminal object in the internal category $\Delta\downarrow G$ of cones over $G$. 
\end{defn}

\begin{ex} \label{ex:doublelimit}
    In our examples of interest, we have the following. 
    \begin{enumerate}
        \item When $\cV=\Set$, internal limits in a category $\bA$ coincide with the usual notion of limits for categories. 
        \item When $\cV=\Cat$, internal limits in a double category $\bA$ correspond to the \emph{double limits} considered in \cite[\S 4.2]{GrandisPare}. 
        \item[($n$)] When $\cV=(n-1)\Cat$ for $n\geq 3$, we refer to internal limits in a double $(n-1)$-category~$\bA$ as \emph{double $(n-1)$-limits}.
    \end{enumerate}
\end{ex}

We also introduce the following functors, which play a key role in the proof of the main result.

\begin{defn}
    The \textbf{suspension functor} $\Sigma\colon \cV\to \sliceunder{\VCat}{\{0,1\}}$ sends an object $X$ in $\cV$ to the $\cV$-category $\Sigma X$ whose
    \begin{rome}
        \item object set is $\{0,1\}$, 
        \item hom-objects in $\cV$ are given by 
        \[ \Sigma X(0,0)=\Sigma X(1,1)\coloneqq \un, \quad \Sigma X(0,1)\coloneqq X, \quad \text{and} \quad \Sigma X(1,0)\coloneqq \emptyset, \]
        where $\emptyset$ denotes the initial object of $\cV$,
        \item composition is uniquely determined, 
        \item identity maps are given by the identity at $\un$. 
    \end{rome}
    It sends a morphism $f\colon X\to Y$ in $\cV$ to the $\cV$-functor $\Sigma f\colon \Sigma X\to \Sigma Y$ which is the identity on objects and is given by 
    \[ f\colon \Sigma X(0,1)=X\to Y=\Sigma Y(0,1)\]
    on the non-trivial hom-objects.
\end{defn} 

\begin{defn}
    The \textbf{hom-object functor} $\Hom\colon \sliceunder{\VCat}{\{0,1\}}\to \cV$ sends an object $\cC_{A,B}$ in $\sliceunder{\VCat}{\{0,1\}}$, i.e., a $\cV$-category $\cC$ together with objects $A,B\in \cC$, to the hom-object $\cC(A,B)$ in~$\cV$. It sends a $\cV$-functor $F\colon \cC_{A,B}\to \cD_{FA,FB}$ to the map $F_{A,B}\colon \cC(A,B)\to \cD(FA,FB)$ in $\cV$. 
\end{defn}

It is not hard to observe the following.

\begin{prop}
    There is an adjunction
    \begin{tz}
\node[](1) {$\cV$}; 
\node[right of=1,xshift=1.4cm](2) {$\sliceunder{\VCat}{\{0,1\}}$}; 
\punctuation{2}{.};

\draw[->] ($(1.east)+(0,5pt)$) to node[above,la]{$\Sigma$} ($(2.west)+(0,5pt)$);
\draw[->] ($(2.west)-(0,5pt)$) to node[below,la]{$\Hom$} ($(1.east)-(0,5pt)$);

\node[la] at ($(1.east)!0.5!(2.west)$) {$\bot$};
\end{tz}
\end{prop}

 Note that there is a dual of the internal Grothendieck construction 
\[ \textstyle\int_\cI\colon \VCat(\cI,\cV)\to \Dopfib(\cI) \]
which takes a $\cV$-functor $W\colon \cI\to \cV$ to an \emph{internal discrete opfibration} $\pi_W\colon \int_\cI W\to \Int\cI$, where these are defined as in \cref{defn:discfib} by replacing the target maps with the source maps. We can now formally state the main result of this section. 

\begin{theorem} \label{thm:weightedasconical}
    Let $G\colon \cI\to \cC$ and $W\colon \cI\to \cV$ be $\cV$-functors. Given a pair $(L,\lambda)$ of an object $L\in \cC$ and a $\cV$-natural transformation $\lambda\colon W\Rightarrow \cC(-,G)$ in $\VCat(\cI,\cV)$, the following are equivalent: 
    \begin{rome}
        \item the pair $(L,\lambda)$ is a $W$-weighted limit of $G$, 
        \item the pair $(L,\lambda)$ is an internal limit of the internal functor 
        \[ \textstyle\int_\cI W\xrightarrow{\pi_W} \Int\cI\xrightarrow{\Int G} \Int\cC. \]
    \end{rome}
\end{theorem}

\begin{proof}
A $W$-weighted limit of $G$ corresponds by \cref{thm:repweightedlimits} to an internal terminal object in the comma internal category \[ \textstyle\int_\cC \cV^{\cI}(W,\cC(-,G))\cong W\downarrow \Int\cC(-,G)\coloneqq \Int\cC(-,G)^{\op}\downarrow W. \] 
On the other hand, an internal limit of $\Int G\circ \pi_W\colon \int_\cI W\to \Int\cC$ is an internal terminal object in the internal category $\Delta\downarrow (\Int G\circ \pi_W)$ of cones over $\Int G\circ \pi_W$ by definition. Then, since internal terminal objects are invariant under isomorphisms by \cref{terminalinvariance}, in order to prove the result it is enough to show that there is an isomorphism in $\CatV$
\setcounter{equation}{\value{theorem}}
\refstepcounter{theorem}\refstepcounter{equation}
\begin{equation} \label{psi}
W\downarrow \Int\cC(-,G)\cong\Delta\downarrow (\Int G\circ \pi_W).
\end{equation}

Moreover, the internal functors \[ W\downarrow \Int\cC(-,G)\to \Int\cC \quad \text{and} \quad \Delta\downarrow (\Int G\circ \pi_W)\to \Int\cC\] are both internal discrete fibrations by \cref{prop:commaisdiscfib}. Hence, to see that we have an isomorphism as in \eqref{psi} over $\Int\cC$, by applying the inverse of the internal Grothendieck construction from \cref{subsec:inversegroth} it is enough to construct a $\cV$-natural transformation between their fibers. In other words, we must construct an isomorphism in $\cV$
\[ \textstyle\cV^{\cI}(W,\cC(A,G-))\cong \Und \llbracket\int_\cI W,\Int\cC\rrbracket(\Delta A,\Int G\circ \pi_W), \]
which is $\cV$-natural in $A\in\cC$.

By the Yoneda lemma, this is equivalent to constructing a natural isomorphism of sets
\[ \textstyle\cV(X,\cV^{\cI}(W,\cC(A,G-)))\cong \cV(X,\Und\llbracket\int_\cI W,\Int\cC\rrbracket(\Delta A,\Int G\circ\pi_W)),\] for every object $X\in\cV$. But note that we have the following natural isomorphisms of sets:
    \begin{align*}
    \cV(X,\cV^{\cI}(W,\cC(A,G-))) &\cong \sliceunder{\VCat}{\{0,1\}}(\Sigma X, (\cV^{\cI})_{W,\cC(A,G-)}) & \Sigma\dashv \Hom \\
    &\cong \sliceunder{\VCat}{\cI\times \{0,1\}}(\cI\times \Sigma X,\cV_{W,\cC(A,G-)}) & \cI\times (-)\dashv (-)^\cI 
    \end{align*}
and
    \begin{align*}
    \cV(X,\Und&\textstyle\llbracket\int_\cI W,\Int\cC\rrbracket(\Delta A,\Int G\circ\pi_W)) & \\
    &\cong \textstyle\sliceunder{\VCat}{\{0,1\}}(\Sigma X,\Und\llbracket\int_\cI W,\Int\cC\rrbracket_{\Delta A,\Int G\circ\pi_W})  &\Sigma\dashv \Hom
      \\
    &\cong \textstyle\sliceunder{\CatV}{\{0,1\}}(\Int\Sigma X,\llbracket\int_\cI W,\Int\cC\rrbracket_{\Delta A,\Int G\circ\pi_W}) & \Int\dashv \Und  \\
    &\cong \textstyle\sliceunder{\CatV}{\int_\cI W\times \{0,1\}}(\int_\cI W\times \Int\Sigma X,\Int\cC_{\Delta A,\Int G\circ\pi_W}) & \textstyle\int_\cI W\times (-)\dashv \llbracket\int_\cI W,-\rrbracket\\
\end{align*}
Then, our proof will be done if we show that, for all objects $A\in \cC$ and $X\in\cV$, there is a natural isomorphism of sets
    \[ \textstyle\sliceunder{\VCat}{\cI\times \{0,1\}}(\cI\times \Sigma X,\cV_{W,\cC(A,G-)})\cong \sliceunder{\CatV}{\int_\cI W\times \{0,1\}}(\int_\cI W\times \Int\Sigma X,\Int\cC_{\Delta A,\Int G\circ\pi_W}). \] 
    We prove this fact as a separate result below.
\end{proof}

\begin{lemma} \label{lem:isoofslices}
    For all objects $X\in\cV$ and $A\in \cC$, there is an isomorphism of sets
    \[ \textstyle\sliceunder{\VCat}{\cI\times \{0,1\}}(\cI\times \Sigma X,\cV_{W,\cC(A,G-)})\cong \sliceunder{\CatV}{\int_\cI W\times \{0,1\}}(\int_\cI W\times \Int\Sigma X,\Int\cC_{\Delta A,\Int G\circ\pi_W}). \]
\end{lemma}

\begin{proof}
    Let us begin by  unpacking the data of the elements of each of the sets involved. We first claim that the data of an element in $\sliceunder{\VCat}{\cI\times \{0,1\}}(\cI\times \Sigma X,\cV_{W,\cC(A,G-)})$ 
    amounts to a family of maps in $\cV$
    \[ \{\ev^H_{i,j}\colon Wi\times \cI(i,j)\times X\to \cC(A,Gj)\}_{i,j\in \cI} \]
    making the following diagram in $\cV$ commute, for all objects $i,j,k\in \cI$. 
    \begin{diagram} \label{comp1}
        \node[](0) {$Wi\times \cI(i,j)\times \cI(j,k)\times X$}; 
        \node[below right of=0,xshift=2.5cm](0') {$Wi\times \cI(i,k)\times X$}; 
        \node[below right of=0',xshift=2.5cm](0'') {$\cC(A,Gk)$};

        \node[above right of=0',xshift=2.5cm](1){$Wj\times \cI(j,k)\times X$};
        \node[below left of=0',xshift=-2.5cm](2) {$\cC(A,Gj)\times \cI(j,k)$};
        
        \draw[->] (0) to node[right,la,yshift=4pt]{$\id_{Wi}\times c_{i,j,k}\times \id_X$} (0');
        \draw[->] (0') to node[right,la,yshift=4pt]{$\ev^H_{i,k}$} (0'');
        \draw[->] (0) to node[above,la]{$\ev^W_{i,j}\times \id_{\cI(j,k)\times X}$} (1); 
        \draw[->] (1) to node[right,la]{$\ev^H_{j,k}$} (0'');
        \draw[->] (0) to node[left,la]{$\ev^H_{i,j}\times \id_{\cI(j,k)}$} (2); 
        \draw[->] (2) to node[below,la]{$\ev^{\cC(A,G-)}_{j,k}$} (0'');
    \end{diagram}

    To see this note that, an element in $\sliceunder{\VCat}{\cI\times \{0,1\}}(\cI\times \Sigma X,\cV_{W,\cC(A,G-)})$ consists by definition of a $\cV$-functor $H\colon \cI\times \Sigma X\to \cV$ making the following diagram in $\VCat$ commute.
    \begin{tz}
        \node[](1) {$\cI\times\{0,1\}$}; 
        \node[below of=1](2) {$\cI\times \Sigma X$}; 
        \node[right of=2,xshift=1.1cm](3) {$\cV$}; 
        \draw[right hook->] (1) to (2); 
        \draw[->] (2) to node[below,la]{$H$} (3); 
        \draw[->] (1) to node[above,la,xshift=29pt]{$(W,\cC(A,G-))$} (3);
    \end{tz}
    Then, we can see that the $\cV$-functor $H$ is completely determined on objects by $W$ and $\cC(A,G-)$. Moreover, it acts as $W$ on hom-objects of the form $(\cI\times \Sigma X)((i,0),(j,0))=\cI(i,j)$ and as $\cC(A,G-)$ on hom-objects of the form $(\cI\times \Sigma X)((i,1),(j,1))=\cI(i,j)$. In particular, the compatibility of $H$ with identities and compositions when restricted to the sub-$\cV$-categories $\cI\times \{0\}$ and $\cI\times \{1\}$ is automatic.
    
    Hence, the data of $H$ amounts to maps on hom-objects 
    \[ \ev^H_{i,j}\colon H(i,0)\times (\cI\times \Sigma X)((i,0),(j,1))=Wi\times \cI(i,j)\times X\to \cC(A,Gj)=H(j,1), \]
    which are compatible with compositions. In particular, the commutativity of the diagram~\eqref{comp1} expresses the fact that the maps $\ev^H_{i,j}$ are compatible with compositions with respect to hom-objects of the form 
    \[ (\cI\times \Sigma X)((i,0),(j,0))\times (\cI\times \Sigma X)((j,0),(k,1))=\cI(i,j)\times \cI(j,k)\times X \]
    in the upper triangle, and with respect to hom-objects of the form 
    \[ (\cI\times \Sigma X)((i,0),(j,1))\times (\cI\times \Sigma X)((j,1),(k,1))=\cI(i,j)\times X\times \cI(j,k) \]
    in the lower triangle, which recovers our desired description.

    Next, we focus on the set $\sliceunder{\CatV}{\int_\cI W\times \{0,1\}}(\int_\cI W\times \Int\Sigma X,\Int\cC_{\Delta A,\Int G\circ\pi_W}),$ and we prove that the data of an element in this set 
    amounts to a family of maps in $\cV$ 
    \[ \{ \ev^H_{i,j}\colon Wi\times \cI(i,j)\times X\to \cC(A,Gj)\}_{i,j\in \cI} \]
   making the following diagram in $\cV$ commute, for all objects $i,j,k\in \cI$. 
    \begin{diagram} \label{comp2}
        \node[](0) {$Wi\times \cI(i,j)\times \cI(j,k)\times X$}; 
        \node[below right of=0,xshift=2.5cm](0') {$Wi\times \cI(i,k)\times X$}; 
        \node[below right of=0',xshift=2.5cm](0'') {$\cC(A,Gk)$};

        \node[above right of=0',xshift=2.5cm](1){$Wj\times \cI(j,k)\times X$};
        \node[below left of=0',xshift=-2.5cm](2) {$\cC(A,Gj)\times \cC(Gj,Gk)$};
        
        \draw[->] (0) to node[right,la,yshift=4pt]{$\id_{Wi}\times c_{i,j,k}\times \id_X$} (0');
        \draw[->] (0') to node[right,la,yshift=4pt]{$\ev^H_{i,k}$} (0'');
        \draw[->] (0) to node[above,la]{$\ev^W_{i,j}\times \id_{\cI(j,k)\times X}$} (1); 
        \draw[->] (1) to node[right,la]{$\ev^H_{j,k}$} (0'');
        \draw[->] (0) to node[left,la]{$\ev^H_{i,j}\times G_{j,k}$} (2); 
        \draw[->] (2) to node[below,la]{$c_{A,Gj,Gk}$} (0'');
    \end{diagram}

 To see this, note that an element in the set $\sliceunder{\CatV}{\int_\cI W\times \{0,1\}}(\int_\cI W\times \Int\Sigma X,\Int\cC_{\Delta A,\Int G\circ\pi_W})$ consists by definition of an internal functor $H\colon \int_\cI W\times \Int\Sigma X\to \Int\cC$ making the following diagram in $\CatV$ commute.
    \begin{tz}
        \node[](1) {$\int_\cI W\times \{0,1\}$}; 
        \node[below of=1](2) {$\int_\cI W\times \Int\Sigma X$}; 
        \node[right of=2,xshift=2cm](3) {$\Int\cC$}; 
        \draw[right hook->] (1) to (2); 
        \draw[->] (2) to node[below,la]{$H$} (3); 
        \draw[->] (1) to node[above,la,xshift=33pt]{$(\Delta A,\Int G\circ \pi_W)$} (3);
    \end{tz}
Then, we can see that the internal functor $H$ is completely determined on level $0$ by~$\Delta A$ and~$\Int G\circ \pi_W$, and hence the data of the map $H_1$ amounts to a map in~$\cV$ of the form
    \[ \textstyle H_1\colon (\bigsqcup_{i,j\in \cI} Wi\times \cI(i,j))\times (\{0,1\}\sqcup X) \to \bigsqcup_{A,B\in \cC} \cC(A,B).\]
    Moreover, the map $H_1$ acts as $\Delta A$ on the components of the form $Wi\times \cI(i,j)\times \{0\}$ and as~$\Int G\circ \pi_W$ on the components of the form $Wi\times \cI(i,j)\times \{1\}$. In particular, the compatibility of~$F$ with identity and composition when restricted to the internal subcategories $\int_\cI W\times \{0\}$ and $\int_\cI W\times \{1\}$ is automatic.

    Hence, the data of $H_1$ amount to maps of components in $\cV$
    \[ \ev^H_{i,j}\colon Wi\times \cI(i,j)\times X\to \cC(A,Gj) \]
    which are compatible with composition. In particular, the commutativity of the diagram \eqref{comp2} expresses the fact that the maps $\ev^H_{i,j}$ are compatible with composition with respect to components of the form 
    \[ (Wi\times \cI(i,j)\times \{0\}) \times_{Wj\times \{0\}} (Wj\times \cI(j,k)\times X)= Wi\times \cI(i,j)\times \cI(j,k)\times X\]
    in the upper triangle, and with respect to components of the form 
    \[ (Wi\times \cI(i,j)\times X) \times_{Wj\times \{1\}} (Wj\times \cI(j,k)\times \{1\})= Wi\times \cI(i,j)\times X\times \cI(j,k)\]
    in the lower triangle. This recovers our claimed description.

    Finally, to see that the required isomorphism in the statement of this lemma holds, given our descriptions above it suffices to show that the following diagram in $\cV$ commutes. 
    \begin{diagram} \label{square1}
\node[](1) {$Wi\times \cI(i,j)\times \cI(j,k)\times X$}; 
\node[below of=1](3) {$\cC(A,Gj)\times \cC(Gj,Gk)$};
\node[right of=1,xshift=4.5cm](2) {$\cC(A,Gj)\times \cI(j,k)$}; 
\node[below of=2](4) {$\cC(A,Gk)$}; 

\draw[->] (1) to node[above,la]{$\ev^H_{i,j}\times \id_{\cI(j,k)}$} (2);
\draw[->] (2) to node[right,la]{$\ev^{\cC(A,G-)}_{j,k}$} (4);
\draw[->] (1) to node[left,la]{$\ev^H_{i,j}\times G_{j,k}$} (3);
\draw[->] (3) to node[below,la]{$c_{A,Gj,Gk}$} (4);
\end{diagram}
By definition, the $\cV$-functor $\cC(A,G-)$ is given by the composite 
\[ \cI\xrightarrow{G}\cC\xrightarrow{\cC(A,-)} \cV \]
and so the below left diagram in $\cV$ commutes.
\begin{tz}
\node[](1) {$\cI(j,k)$}; 
\node[below of=1](3) {$\cC(Gj,Gk)$};
\node[right of=3,xshift=3.1cm](4) {$[\cC(A,Gj),\cC(A,Gk)]$}; 

\draw[->] (1) to node[above,la,xshift=14pt]{$\cC(A,G-)_{j,k}$} (4);
\draw[->] (1) to node[left,la]{$G_{j,k}$} (3);
\draw[->] (3) to node[below,la]{$\cC(A,-)_{Gj,Gk}$} (4);

\node[right of=1,xshift=8cm](1) {$\cC(A,Gj)\times \cI(j,k)$}; 
\node[below of=1](3) {$\cC(A,Gj)\times \cC(Gj,Gk)$};
\node[right of=3,xshift=2.8cm](4) {$\cC(A,Gk)$}; 

\draw[->] (1) to node[above,la,xshift=11pt]{$\ev^{\cC(A,G-)}_{j,k}$} (4);
\draw[->] (1) to node[left,la]{$\id_{\cC(A,Gj)}\times G_{j,k}$} (3);
\draw[->] (3) to node[below,la]{$c_{A,Gj,Gk}$} (4);
\end{tz}
Using the adjunction $\cC(A,Gj)\times (-)\dashv [\cC(A,Gj),-]$ and the definition of the representable $\cV$-functor $\cC(A,-)$, we can rewrite the above left diagram as the above right commutative diagram in $\cV$. The commutativity of \eqref{square1} then follows from the commutativity of this last triangle. 
\end{proof}

As a particular case of \cref{thm:weightedasconical}, we get the following characterization of (conical) $\cV$-limits.

\begin{cor} \label{cor:enrichedlimitasinternallimit}
    Let $G\colon \cI\to \cV$ be a $\cV$-functor. Given a pair $(L,\lambda)$ of an object $L\in \cC$ and a $\cV$-natural transformation $\lambda\colon \Delta L\Rightarrow G$ in $\VCat(\cI,\cC)$, the following are equivalent: 
    \begin{rome}
        \item the pair $(L,\lambda)$ is a $\cV$-limit of $G$, 
        \item the pair $(L,\lambda)$ is an internal limit of $\Int G$.
    \end{rome}
\end{cor}

\begin{proof}
    This result is obtained by taking $W=\Delta \un\colon \cI\to \cV$ to be the constant $\cV$-functor at the terminal object $\un$ of $\cV$ in \cref{thm:weightedasconical} and noticing that the projection $\int_\cI \Delta\un\to \Int\cI$ is simply the identity internal functor at $\Int\cI$.
\end{proof}

\begin{ex} \label{ex:weightedasconical}
    In our examples of interest, using the terminology from \cref{ex:doublelimit}, \cref{thm:weightedasconical,cor:enrichedlimitasinternallimit} translate as follows. 
    \begin{enumerate}
        \item When $\cV=\Set$, the theorem recovers the classical result that all weighted limits of categories are in particular conical limits; see e.g.~\cite[(7.1.8)]{Riehlcathtpyth}. The corollary does not tell us anything new.
        \item When $\cV=\Cat$, the theorem retrieves the result from \cite[Proposition 1.4]{GraParPersistentII}, showing that every weighted $2$-limit can be obtained as a certain double limit. The corollary then implies that a $2$-limit of a $2$-functor $G$ corresponds to a double limit of the double functor $\Int G$, as mentioned in \cite[\S 4.2 c)]{GrandisPare}.
        \item[($n$)] When $\cV=(n-1)\Cat$ for $n\geq 3$, the theorem tells us that every weighted $n$-limit can be obtained as a certain double $(n-1)$-limit. In particular, the corollary implies that an $n$-limit of an $n$-functor $G$ corresponds to a double $(n-1)$-limit of the double $(n-1)$-functor $\Int G$. 
    \end{enumerate}
\end{ex}

\begin{rem} \label{rem:justifytensors}
    Given a $\cV$-category $\cC$ and objects $C\in \cC$ and $X\in \cV$, then the $\cV$-tensor of $C$ by $X$ can equivalently be defined as the weighted $\cV$-colimit of the $\cV$-functor $C\colon \mathbbm{1}\to \cC$ by the weight $X\colon \mathbbm{1}\to \cV$; see \cite[\S 3.7]{Kelly}. Then, by the dual version of \cref{thm:weightedasconical}, this amounts to the internal colimit of the internal functor 
\[ \textstyle X=\int_\mathbbm{1} X\to \un\xrightarrow{C}\Int\cC. \]
Hence $\cV$-tensors correspond to internal tensors of a specific form, which justifies our terminology.
\end{rem}

\bibliographystyle{alpha}
\bibliography{References}

\begin{thebibliography}{CLPS22}

\bibitem[Ber07]{bergner}
Julia~E. Bergner.
\newblock A model category structure on the category of simplicial categories.
\newblock {\em Trans. Amer. Math. Soc.}, 359(5):2043--2058, 2007.

\bibitem[Buc14]{Buckley}
Mitchell Buckley.
\newblock Fibred 2-categories and bicategories.
\newblock {\em J. Pure Appl. Algebra}, 218(6):1034--1074, 2014.

\bibitem[BW19]{BeardsleyWong}
Jonathan Beardsley and Liang~Ze Wong.
\newblock The enriched {G}rothendieck construction.
\newblock {\em Adv. Math.}, 344:234--261, 2019.

\bibitem[CFP17]{Extensive}
Thomas Cottrell, Soichiro Fujii, and John Power.
\newblock Enriched and internal categories: an extensive relationship.
\newblock {\em Tbilisi Math. J.}, 10(3):239--254, 2017.

\bibitem[CLPS22]{CLPS}
Geoffrey S.~H. Cruttwell, Michael~J. Lambert, Dorette~A. Pronk, and Martin
  Szyld.
\newblock Double fibrations.
\newblock {\em Theory Appl. Categ.}, 38:Paper No. 35, 1326--1394, 2022.

\bibitem[cM22a]{cM1}
tslil clingman and Lyne Moser.
\newblock 2-limits and 2-terminal objects are too different.
\newblock {\em Appl. Categ. Structures}, 30(6):1283--1304, 2022.

\bibitem[cM22b]{cM2}
tslil clingman and Lyne Moser.
\newblock Bi-initial objects and bi-representations are not so different.
\newblock {\em Cah. Topol. G\'{e}om. Diff\'{e}r. Cat\'{e}g.}, 63(3):259--330,
  2022.

\bibitem[GHL22]{GHL}
Andrea Gagna, Yonatan Harpaz, and Edoardo Lanari.
\newblock Bilimits are bifinal objects.
\newblock {\em J. Pure Appl. Algebra}, 226(12):Paper No. 107137, 36, 2022.

\bibitem[GP99]{GrandisPare}
Marco Grandis and Robert Paré.
\newblock Limits in double categories.
\newblock {\em Cahiers Topologie G\'{e}om. Diff\'{e}rentielle Cat\'{e}g.},
  40(3):162--220, 1999.

\bibitem[GP19]{GraParPersistentII}
Marco Grandis and Robert Par\'{e}.
\newblock Persistent double limits and flexible weighted limits.
\newblock \url{https://www.mscs.dal.ca/~pare/DblPrs2.pdf}, 2019.

\bibitem[Gra20]{Grandis}
Marco Grandis.
\newblock {\em Higher dimensional categories}.
\newblock World Scientific Publishing Co. Pte. Ltd., Hackensack, NJ, 2020.
\newblock From double to multiple categories.

\bibitem[HM15]{HeutsMoerdijk}
Gijs Heuts and Ieke Moerdijk.
\newblock Left fibrations and homotopy colimits.
\newblock {\em Math. Z.}, 279(3-4):723--744, 2015.

\bibitem[Joh02]{elephant}
Peter~T. Johnstone.
\newblock {\em Sketches of an elephant: a topos theory compendium. {V}ol. 1},
  volume~43 of {\em Oxford Logic Guides}.
\newblock The Clarendon Press, Oxford University Press, New York, 2002.

\bibitem[Joy08]{joyal2008notes}
Andr{\'e} Joyal.
\newblock Notes on quasi-categories.
\newblock preprint available at
  \url{https://www.math.uchicago.edu/~may/IMA/Joyal.pdf}, 2008.

\bibitem[JY21]{JohYau}
Niles Johnson and Donald Yau.
\newblock {\em 2-dimensional categories}.
\newblock Oxford University Press, Oxford, 2021.

\bibitem[Kel89]{Kellyelementary}
G.~Maxwell Kelly.
\newblock Elementary observations on {$2$}-categorical limits.
\newblock {\em Bull. Austral. Math. Soc.}, 39(2):301--317, 1989.

\bibitem[Kel05]{Kelly}
G.~Maxwell Kelly.
\newblock Basic concepts of enriched category theory.
\newblock {\em Repr. Theory Appl. Categ.}, (10):vi+137, 2005.
\newblock Reprint of the 1982 original [Cambridge Univ. Press, Cambridge;
  MR0651714].

\bibitem[Lam21]{Lambert}
Michael Lambert.
\newblock Discrete double fibrations.
\newblock {\em Theory Appl. Categ.}, 37:671--708, 2021.

\bibitem[LR20]{LoregianRiehl}
Fosco Loregian and Emily Riehl.
\newblock Categorical notions of fibration.
\newblock {\em Expo. Math.}, 38(4):496--514, 2020.

\bibitem[Lur09]{htt}
Jacob Lurie.
\newblock {\em Higher topos theory}, volume 170 of {\em Annals of Mathematics
  Studies}.
\newblock Princeton University Press, Princeton, NJ, 2009.

\bibitem[MRR]{MRR3}
Lyne Moser, Nima Rasekh, and Martina Rovelli.
\newblock Limits in an $(\infty,n)$-category.
\newblock Work in progress.

\bibitem[MRR23]{MRR2}
Lyne Moser, Nima Rasekh, and Martina Rovelli.
\newblock An $(\infty,n)$-categorical straightening-unstraightening
  construction.
\newblock \href{https://arxiv.org/abs/2307.07259}{arXiv:2307.07259}, 2023.

\bibitem[Ras21]{RasekhD}
Nima Rasekh.
\newblock Yoneda lemma for $\mathcal{D}$-simplicial spaces.
\newblock \href{https://arxiv.org/abs/2108.06168}{arXiv:2108.06168}, 2021.

\bibitem[Ras23]{RasekhYoneda}
Nima Rasekh.
\newblock Yoneda {L}emma for {S}implicial {S}paces.
\newblock {\em Appl. Categ. Structures}, 31(4):27, 2023.

\bibitem[Rie14]{Riehlcathtpyth}
Emily Riehl.
\newblock {\em Categorical homotopy theory}, volume~24 of {\em New Mathematical
  Monographs}.
\newblock Cambridge University Press, Cambridge, 2014.

\bibitem[Rie16]{Riehlcontext}
Emily Riehl.
\newblock {\em Category theory in context}.
\newblock Aurora Modern Math Originals. Dover Publications, 2016.

\bibitem[RV22]{RiehlVerity}
Emily Riehl and Dominic Verity.
\newblock {\em Elements of {$\infty$}-category theory}, volume 194 of {\em
  Cambridge Studies in Advanced Mathematics}.
\newblock Cambridge University Press, Cambridge, 2022.

\bibitem[Tam09]{tamaki}
Dai Tamaki.
\newblock The grothendieck construction and gradings for enriched categories.
\newblock arXiv:0907.0061, 2009.

\bibitem[Tom20]{Tomasic}
Ivan Toma\v{s}i\'{c}.
\newblock A topos-theoretic view of difference algebra.
\newblock \href{https://arxiv.org/abs/2001.09075}{arXiv:2001.09075}, 2020.

\bibitem[Ver11]{VerityThesis}
Dominic Verity.
\newblock Enriched categories, internal categories and change of base.
\newblock {\em Repr. Theory Appl. Categ.}, (20):1--266, 2011.

\end{thebibliography}

\end{document}